\documentclass[12pt, usenames,dvipsnames]{amsart}
\usepackage{amssymb,amsmath,amsfonts,amscd,euscript,mathrsfs,bm,bbm}
\usepackage{lineno}

\usepackage{amsthm}
\usepackage{amsopn} 
\usepackage{verbatim} 
\usepackage{mathtools}
\usepackage{enumitem}
\usepackage{longtable,booktabs,array}

\usepackage[margin=1in, headsep=18pt]{geometry}

\usepackage[backref=page]{hyperref}

\usepackage{tikz}
\usepackage{tikz-cd}
\usetikzlibrary{babel}
\usetikzlibrary{matrix}
\usetikzlibrary{shapes}
\usetikzlibrary{arrows}
\usetikzlibrary{calc,3d}
\usetikzlibrary{decorations,decorations.pathmorphing}
\usetikzlibrary{through}
\tikzset{ext/.style={circle, draw,inner sep=1pt},int/.style={circle,draw,fill,inner sep=1pt},nil/.style={inner sep=1pt}}
\tikzset{exte/.style={circle, draw,inner sep=3pt},inte/.style={circle,draw,fill,inner sep=3pt}}
\tikzset{diagram/.style={matrix of math nodes, row sep=3em, column sep=2.5em, text height=1.5ex, text depth=0.25ex}}
\tikzset{diagram2/.style={matrix of math nodes, row sep=0.5em, column sep=0.5em, text height=1.5ex, text depth=0.25ex}}
\tikzset{every picture/.append style={baseline=-.65ex}}

\usepackage{hyperref}

\usepackage{ulem}
\usepackage{color}

\newcommand\cbl[1]{{\color{blue} #1}}

\let\le\leqslant
\let\ge\geqslant
\let\leq\leqslant
\let\geq\geqslant

\newcounter{alphaequation}

\newcommand{\nc}{\newcommand}
\newtheorem{theoremA}{Theorem}

\numberwithin{equation}{section}
\newtheorem{conjecture}[equation]{Conjecture}
\newtheorem*{conj*}{Conjecture}
\newtheorem{theorem}[equation]{Theorem}
\newtheorem{proposition}[equation]{Proposition}
\newtheorem{proposition-definition}[equation]{Proposition-Definition}

\newtheorem{lemma}[equation]{Lemma}
\newtheorem{corollary}[equation]{Corollary}
\newtheorem*{cor*}{Corollary}
\newtheorem{remark}[equation]{Remark}
\newtheorem{prop::def}[equation]{Proposition-Definition}

\newtheorem{dfn}[equation]{Definition}
\newtheorem{definition}[equation]{{\bf Definition}}

\newtheorem{algorithm}[equation]{{\bf Algorithmic Definition}}
\newtheorem{notation}[equation]{{\bf Notation}}
\newtheorem{fact}[equation]{Fact}

\theoremstyle{definition}
\newtheorem{example}[equation]{Example}

\nc{\fB}{\mathfrak{B}}
\nc{\gl}{\mathfrak{gl}}
\nc{\GL}{\mathfrak{GL}}
\nc{\g}{\mathfrak{g}}
\nc{\gh}{\widehat\g}
\nc{\h}{\mathfrak{h}}
\nc{\wfh}{\widehat{\mathfrak{h}}}
\nc{\la}{\lambda}
\nc{\al}{\alpha }
\nc{\be}{\beta }
\nc{\ve}{\varepsilon }
\nc{\om}{\omega }
\nc{\lr}{\text{-}}
\nc{\ta}{\theta}
\nc{\ch}{{\mathop {\rm ch}}}
\nc{\Tr}{{\mathop {\rm Tr}\,}}
\nc{\tr}{{\mathrm tr}}
\nc{\Id}{{\mathop {\rm Id}}}
\nc{\ad}{{\mathop {\rm ad}}}
\nc{\End}{{\mathop {\rm End}}}

\nc{\bra}{\langle}
\nc{\ket}{\rangle}
\nc{\bi}{{\bf i}}
\nc{\pa}{\partial}
\nc{\ld}{\ldots}
\nc{\cd}{\cdots}
\nc{\hk}{\hookrightarrow}
\nc{\T}{\otimes}
\nc{\gr}{\mathrm{gr}}
\nc{\ov}{\overline}

\nc{\msl}{\mathfrak{sl}}
\nc{\mgl}{\mathfrak{gl}}
\nc{\U}{\mathrm U}
\nc{\euJ}{\EuScript J}
\nc{\cO}{\mathcal{O}}
\nc{\cL}{\mathcal{L}}
\nc{\Res}{{\mathbf{Res}}}
\nc{\Ind}{{\mathbf{Ind}}}
\nc{\Coind}{{\mathbf{CoInd}}}
\nc{\LInd}{{\mathbf{LInd}}}
\nc{\RCoind}{{\mathbf{RCoInd}}}

\nc{\tI}{{\mathsf{I}}}
\nc{\sR}{\mathbf{R}}
\nc{\sZ}{\mathbf{Z}}

\newcommand{\Lattice}{\mathbb{L}} 
\newcommand{\bC}{{\Bbbk}}

\newcommand{\bD}{{\mathbb D}}  
\newcommand{\bA}{{\mathbb A}} 
\newcommand{\bAD}{{\mathbb B}}

\newcommand{\bZ}{{\mathbb Z}}

\newcommand{\fh}{{\mathfrak h}}

\newcommand{\fg}{{\mathfrak g}}

\newcommand{\fb}{{\mathfrak b}}

\newcommand{\fn}{{\mathfrak n}}

\newcommand{\euE}{\EuScript{E}}
\newcommand{\EL}{\mathsf{EL}}

\nc{\fI}{\mathfrak I}
\nc{\bfI}{\mathbf I}

\nc{\Q}{\mathfrak Q}
\nc{\fr}{\mathfrak r}
\nc{\W}{\mathbb W}
\nc{\bU}{\mathbb U}

\nc{\Gm}{\mathbb{G}_{m}}

\newcommand{\bS}{\mathbb{S}}

\newcommand{\calF}{\mathcal{F}}

\newcommand{\bbY}{\mathbf{Y}}

\newcommand{\sT}{\mathsf{T}}
\newcommand{\sS}{\mathsf{S}}

\newcommand{\de}{\text{-}}

\newcommand{\KK}{\mathsf{K}}
\newcommand{\Mat}{\mathsf{Mat}}

\newcommand{\ldot}{{\:\raisebox{1.5pt}{\selectfont\text{\circle*{1.5}}}}}
\newcommand{\udot}{{\:\raisebox{4pt}{\selectfont\text{\circle*{1.5}}}}}

\newcommand\ttt{\text{-}}

\newcommand{\St}{\mathbf{S}\kern -.25em{\mathbf{c}}}

\newcommand{\DL}{\mathsf{DL}}

\newcommand{\Bruhat}{\small\mathsf{Br}}

\newcommand{\vrt}{\mathsf{vrt}}
\newcommand{\hor}{\mathsf{hor}}
\newcommand{\hb}{\mathsf{hbs}} 
\newcommand{\bbs}{\mathsf{bs}} 

\newcommand{\pth}{\mathsf{path}}

\newcommand{\gridV}[4]{
	{
		\begin{tikzpicture}[scale=0.37]
			\draw[step=1cm] (0,0) grid (1,-4);
			\node (v1) at (0.5,-0.8) {{\cbl{\tiny {#1}}}};
			\node (v2) at (0.5,-1.8) {{\cbl{\tiny {#2}}}};
			\node (v3) at (0.5,-2.8) {{\cbl{\tiny {#3}}}};
			\node (v4) at (0.5,-3.8) {{\cbl{\tiny {#4}}}};
		\end{tikzpicture}
}}

\newcommand{\gridH}[4]{
	{
		\begin{tikzpicture}[scale=0.37]
			\draw[step=1cm] (0,0) grid (4,-1);
			\node (v1) at (0.5,-0.8) {{\cbl{\tiny {#1}}}};
			\node (v2) at (1.5,-0.8) {{\cbl{\tiny {#2}}}};
			\node (v3) at (2.5,-0.8) {{\cbl{\tiny {#3}}}};
			\node (v4) at (3.5,-0.8) {{\cbl{\tiny {#4}}}};
		\end{tikzpicture}
}}
\newcommand{\gridDL}[4]{
	{
		\begin{tikzpicture}[scale=0.37]
			\draw[step=1cm] (0,0) grid (4,-2);
			\draw[step=1cm] (1,0) grid (4,-3);
			\draw[step=1cm] (3,0) grid (4,-4);
			\node (v21) at (0.5,-1.8) {{\cbl{\tiny {#1}}}};
			\node (v32) at (1.5,-2.8) {{\cbl{\tiny {#2}}}};
			\node (v13) at (2.5,-.8) {{\cbl{\tiny {#3}}}};
			\node (v44) at (3.5,-3.8) {{\cbl{\tiny {#4}}}};
		\end{tikzpicture}
}}

\newcommand{\gridDLc}[5]{
	{
		\begin{tikzpicture}[scale=0.37]
			\fill[#1!20] (0,0) -- (4,0) -- (4,-4) -- (3,-4) -- (3,-3) -- (1,-3) -- (1,-2) -- (0,-2) -- cycle;
			\draw[step=1cm] (0,0) grid (4,-2);
			\draw[step=1cm] (1,0) grid (4,-3);
			\draw[step=1cm] (3,0) grid (4,-4);
			\draw[#1, line width=1.2pt, line join=round] (0,0) -- (4,0) -- (4,-4) -- (3,-4) -- (3,-3) -- (1,-3) -- (1,-2) -- (0,-2) -- cycle;
			\node (v21) at (0.5,-1.8) {{\cbl{\tiny {#2}}}};
			\node (v32) at (1.5,-2.8) {{\cbl{\tiny {#3}}}};
			\node (v13) at (2.5,-.8) {{\cbl{\tiny {#4}}}};
			\node (v44) at (3.5,-3.8) {{\cbl{\tiny {#5}}}};
		\end{tikzpicture}
}}

\newcommand{\fgrid}[5]{
	\begin{tikzpicture}[scale=0.37]
		\draw[step=1cm] (0,0) grid (5,1);
		\node (v1) at (0.5,0.4) {{\cbl{\small {#1}}}};
		\node (v2) at (1.5,0.4) {{\cbl{\small {#2}}}};
		\node (v3) at (2.5,0.4) {{\cbl{\small {#3}}}};
		\node (v4) at (3.5,0.4) {{\cbl{\small {#4}}}};
		\node (v5) at (4.5,0.4) {{\cbl{\small {#5}}}};
	\end{tikzpicture}
}

\newcommand{\grDLex}[5]{
	{
		\begin{tikzpicture}[scale=0.3,shift={(0,2)}]
			\draw[step=1cm] (0,0) grid (6,-2);
			\draw[step=1cm] (1,0) grid (6,-3);
			\draw[step=1cm] (4,0) grid (6,-5);
			\node (v21) at (0.5,-1.5) {{\color{blue}{\tiny {#1}}}};
			\node (v32) at (1.5,-2.5) {{\color{blue}{\tiny {#2}}}};
			\node (v13) at (2.5,-.5) {{\color{blue}{\tiny {#3}}}};
			\node (v54) at (4.5,-4.5) {{\color{blue}{\tiny {#4}}}};
			\node (v46) at (5.5,-3.5) {{\color{blue}{\tiny {#5}}}};  
		\end{tikzpicture}
}}

\nc{\sD}{{{{D}\hspace*{-.9em}\text{\bf{--}}\hspace*{0.3em}}}}
\nc{\sL}{{{{L}\hspace*{-.7em}\text{\bf{--}}\hspace*{0.3em}}}}
\nc{\sLD}{{{{LD}\hspace*{-.9em}\text{\bf{---}}\hspace*{0.3em}}}}
\nc{\sDL}{{{{DL}\hspace*{-.9em}\text{\bf{---}}\hspace*{0.3em}}}}
\nc{\ssD}{{{{D}\hspace*{-.65em}\text{\bf{--}}\hspace*{0.2em}}}}
\nc{\ssL}{{{{L}\hspace*{-.55em}\text{\bf{--}}\hspace*{0.2em}}}}

\nc{\Ar}{{\mathsf{Ar}}}

\nc{\disorder}{\mathsf{Disorder}}
\nc{\rk}{\mathsf{rk}}
\nc{\op}{\mathsf{op}}
\nc{\precsucc}{\overset{\prec}{\succ}}

\makeatletter
\newcommand{\nlabel}[1]{\phantomsection\begingroup
	\edef\@currentlabel{\ifnum\value{subsection}>0 \thesubsection\else\thesection\fi}%
	\label{#1}\endgroup}
\makeatother
\newcommand{\ir}[2]{#1~\ref{#2}, p.~\pageref{#2}}      
\newcommand{\irs}[1]{\S\ref{#1}, p.~\pageref{#1}}     
\newcommand{\ire}[1]{\eqref{#1}, p.~\pageref{#1}}     

\newcommand{\ndots}{\unskip\nobreak\leaders\hbox to 0.45em{\hss.\hss}\hskip 0pt plus 1fill\kern0pt\relax}
\newcommand{\ngrouprow}[1]{\multicolumn{3}{@{}p{\textwidth}@{}}{\textbf{#1}\par\nointerlineskip\vskip2pt\hrule height 0.4pt\vskip2.5pt}\\*}

\newcommand{\ngroup}[1]{\addlinespace[1.6ex]\ngrouprow{#1}}
\newcommand{\ngroupfirst}[1]{\ngrouprow{#1}}

\begin{document}
	\title{Bubble sort and Howe duality for staircase matrices }
	

	\author[Khoroshkin]{Anton Khoroshkin}
	\address{Anton Khoroshkin: \newline
		Department of Mathematics, University of Haifa, Mount Carmel, 3498838, Haifa, Israel
	}
	
	\email{khoroshkin@gmail.com}
	
	\author[Makedonskyi]{Ievgen Makedonskyi}
	\address{Ievgen Makedonskyi:\newline
		Yanqi Lake Beijing Institute of Mathematical Sciences And Applications (BIMSA), 
		No. 544, Hefangkou Village, Huaibei Town, Huairou District, Beijing 101408.}
	\email{makedonskii\_e@mail.ru}

 \pagestyle{headings}

	\begin{abstract}
		We prove the alternating Cauchy identity for staircase matrices conjectured in~\cite{FKM::Cauchy}, together with an explicit description of the coefficients occurring in it. As a byproduct, our approach also yields a new, independent (more combinatorial) proof of the Cauchy identities for staircase matrices established in~\cite{FKM::Cauchy}.
		
		The first part of the paper focuses on combinatorial aspects. It is self-contained, of independent interest, and introduces a generalization of parabolic Bruhat graphs for monotone functions on an arborescent poset.
		The second part centers on representation theory. We propose a generalization of the classical Howe duality for staircase matrices in terms of the distributive lattice of Demazure submodules within a given integrable representation. Computing the associated character yields all desired Cauchy identities for staircase matrices.
	\end{abstract}
	
	\maketitle
	
	\setcounter{tocdepth}{1}
	\tableofcontents

	\setcounter{section}{-1}
	\section{Introduction}
	\label{sec::intro}

	This paper proves a family of Cauchy-type identities for \emph{staircase matrices} and explains them by a Howe-type duality in which Demazure modules of Borel subalgebras play the role of irreducible representations. We begin with the classical picture (\S\ref{intro::motivation}), state the identities in a simplified form (\S\ref{intro::ingredients}), and then describe the representation-theoretic (\S\ref{intro::Howe}) and combinatorial (\S\ref{intro::combinatorics}) ingredients of the proof. A guide to the rest of the paper and to the notation is given in \S\ref{intro::structure}.

	\subsection{The classical picture: Cauchy identity and Howe duality}
	\label{intro::motivation}

	Fix positive integers $n,m$ and variables $x=(x_1,\dots,x_n)$, $y=(y_1,\dots,y_m)$. The classical \emph{Cauchy identity} states that
	\begin{equation}
		\label{eq::classical::Cauchy::intro}
		\prod_{i=1}^n\prod_{j=1}^m \frac{1}{1-x_iy_j} = \sum_{\lambda} s_\lambda(x)\, s_\lambda(y),
	\end{equation}
	where the sum runs over the partitions $\lambda$ with at most $\min(n,m)$ non-zero parts and $s_\lambda$ is the Schur polynomial.\nlabel{ni:partition} (We write $\lambda\vdash N$ if $\lambda=(\lambda_1\ge\lambda_2\ge\dots)$ is a partition of $N$, and $l(\lambda)$ for its number of non-zero parts.) The identity has two well-known interpretations.
	\begin{itemize}[itemsep=2pt,topsep=2pt]
		\item \emph{Combinatorial.} It is the generating function of the Robinson--Schensted--Knuth correspondence between matrices with non-negative integer entries and pairs of semistandard tableaux of the same shape.
		\item \emph{Representation-theoretic.} It is the character identity behind \emph{Howe duality}. The Lie algebras $\gl_n$ and $\gl_m$ act on the space $\Mat_{n\times m}$ of $n\times m$ matrices from the left and from the right, and by a theorem of R.~Howe~\cite{Ho} the symmetric algebra decomposes without multiplicities:
		\begin{equation}
			\label{eq::Howe::intro}
			S^{\bullet}(\Mat_{n\times m}) \simeq \bigoplus_{\lambda} V_\lambda^{\gl_n}\otimes (V_\lambda^{\gl_m})^{\op}, \qquad l(\lambda)\leq \min(m,n).
		\end{equation}
		Here $V_\lambda^{\gl_n}$ is the irreducible $\gl_n$-module with highest weight $\lambda$, and $\op$ marks a right module. The matrix unit $E_{ij}$ is a weight vector of weight $x_iy_j$, so the character of the left-hand side is the product in~\eqref{eq::classical::Cauchy::intro}, while $V_\lambda^{\gl_n}\otimes (V_\lambda^{\gl_m})^{\op}$ has character $s_\lambda(x)s_\lambda(y)$. Comparing characters in~\eqref{eq::Howe::intro} gives~\eqref{eq::classical::Cauchy::intro}.
	\end{itemize}
	The Cauchy identity and its $q$-deformations play a central role in combinatorics (symmetric functions, plane partitions, lattice paths), in representation theory (branching rules, dual pairs), in probability (Schur measures, last-passage percolation) and in mathematical physics (vertex models, quantum integrable systems); see, e.g.,~\cite{BC,BP,BW,Ok,OR,St,CK,FL,Las}.
    
	\smallskip
 The goal of this paper is to establish a family of generalized Cauchy identities and their underlying Howe duality, placing the staircase Cauchy identities on equal footing with the classical identity. Specifically, we pass:
\begin{itemize}[leftmargin=2em, itemsep=0pt, topsep=0pt]
    \item from the space of all matrices to a \textbf{staircase-shaped subspace},
    \item from the pair $(\mathfrak{gl}_n, \mathfrak{gl}_m)$ to the pair of \textbf{uppertriangular Borel subalgebras} $(\mathfrak{b}_n, \mathfrak{b}_m)$, and
    \item from the irreducible modules $V_\lambda$ to \textbf{Demazure submodules} $D_{\sigma\lambda}$.
\end{itemize}
Some particular cases of corresponding  {\it "non-symmetric"} versions of the Cauchy identity are known to have combinatorial and probabilistic applications~\cite{Las, AE, AGL}.

	\subsection{Staircase matrices and the main identities}
	\label{intro::ingredients}

	\nlabel{ni:stairintro}Let $\ov n=(n_1\leq n_2\leq\dots\leq n_m)$ be a weakly increasing sequence of positive integers and put $n=n_m$. The \emph{staircase shape} $\bbY_{\ov n}$ is the diagram with $m$ columns in which the $j$-th column consists of the top $n_j$ cells (see Example~\ref{ex::staircase::poset} below). A \emph{staircase matrix} of shape $\ov n$ is an $n\times m$ matrix whose entries vanish outside $\bbY_{\ov n}$; such matrices appear as outputs of Gaussian elimination. They form a subspace $\Mat_{\ov n}\subset\Mat_{n\times m}$, which is stable under the upper triangular (Borel) subalgebras $\fb_n$ acting from the left and $\fb_m$ acting from the right, but not under the full algebras $\gl_n$ and $\gl_m$; this is why Howe duality cannot be applied directly. The character of the symmetric algebra $S^\bullet(\Mat_{\ov n})$ is
	\[
	\prod_{(i,j)\in\bbY_{\ov n}}\frac{1}{1-x_iy_j},
	\]
	and this product is the left-hand side of all identities in the paper.

	\subsubsection{\textbf{Key polynomials and Demazure atoms.}}\nlabel{ni:keyintro} Since only Borel subalgebras act, Schur polynomials have to be replaced by non-symmetric analogues. For a composition $\eta\in\bZ^n_{\ge0}$ let $\eta^+$ be the partition obtained by sorting its entries, and let $D_\eta\subset V^{\gl_n}_{\eta^+}$ be the \emph{Demazure module}, that is, the $\fb_n$-submodule generated by the extremal weight vector of weight $\eta$~\cite{Dem1,Dem2}. Its character is the \emph{key polynomial} $\kappa_\eta(x)$. The \emph{van der Kallen module} ${\KK}_\eta=D_\eta\big/\sum_{\nu\prec\eta}D_\nu$ is the part of $D_\eta$ not seen by smaller Demazure modules~\cite{vdK}; its character is the \emph{Demazure atom} $a_\eta(x)$. Thus $\kappa_\eta=\sum_{\nu\preceq\eta}a_\nu$, where $\preceq$ is the Bruhat order on the rearrangements of $\eta^+$: the order generated by the relations $d<(ij)\cdot d$ for $i<j$ and $d_i>d_j$, where $(ij)\cdot d$ exchanges the entries $d_i$ and $d_j$ (the smallest element is the partition $\eta^+$ itself). Both families $\{\kappa_\eta\}$ and $\{a_\eta\}$ are bases of $\bC[x_1,\dots,x_n]$, and the Schur polynomial $s_\lambda$ is the key polynomial of the increasing rearrangement of $\lambda$. Everything is mirrored for the right action; the corresponding polynomials in the variables $y$ are denoted by $\kappa^\eta(y)$ and $a^\eta(y)$ (they are obtained from the left ones by reversing the order of the variables, see~\eqref{eq:left-right}).

\subsubsection{\textbf{Staircase corners and $DL$-dense arrays.}}\nlabel{ni:cornersintro} The right-hand sides of our identities are governed by a canonical subset $\St_{\ov n}\subset\bbY_{\ov n}$ of \emph{staircase corners}. It is obtained by peeling the diagram: take all corners of $\bbY_{\ov n}$ (the last cell of every column that is longer than the column to its left), delete their rows and columns, take the corners of what remains, and so on (Definition~\ref{def::DL::indices}). By construction no two staircase corners share a row or a column. We order the corners by the rule ``\emph{the further down and to the left, the bigger}'': $(i,j)\succeq(i',j')$ if $i\ge i'$ and $j\le j'$.

	A \emph{$DL$-dense array} (``$DL$'' stands for down--left) is an order-preserving map $A\colon\St_{\ov n}\to\bZ_{\ge0}$, viewed as a filling of $\bbY_{\ov n}$ by non-negative integers that vanishes outside the corners (Definition~\ref{def::DL::dense}). Its \emph{horizontal weight} $\hor(A)\in\bZ^n_{\ge0}$ is the vector of row sums and its \emph{vertical weight} $\vrt(A)\in\bZ^m_{\ge0}$ is the vector of column sums (Definition~\ref{def:arrays}); since every row and every column contains at most one corner, $A$ is determined by either of them (Proposition~\ref{prp::DL::hor::vrt}). We denote by $\DL_{\ov n}$ the set of all $DL$-dense arrays, and by $\DL_{\ov n}(\lambda)$ the subset of arrays whose non-zero entries are the parts of a partition $\lambda$ (necessarily $l(\lambda)\le\#\St_{\ov n}$). The set $\DL_{\ov n}(\lambda)$ carries the Bruhat order: $A\preceq_{\Bruhat}B$ if $\hor(A)\preceq_{\Bruhat}\hor(B)$ among the rearrangements of $\lambda$; by Theorem~\ref{thm::poset::Bruhat::intro} below this is equivalent to $\vrt(A)\preceq_{\Bruhat}\vrt(B)$. Let $\mu^{\DL_{\ov n}(\lambda)}$ denote the M\"obius function of this poset.
\begin{example}
	\label{rem::rectangular::intro}
	For a rectangular shape $\ov n=(n^m)$ the staircase corners $(n,1)\succ(n-1,2)\succ\dots$ form a chain along the diagonal from the lower-left corner upwards. Hence, for every $\lambda$, the set $\DL_{(n^m)}(\lambda)$ consists of a single array: the parts of $\lambda$ are placed on this diagonal in nonincreasing order from the lower-left to the upper-right (cf.~the arrays of Danilov and Koshevoy~\cite{DK1,DK2}). For example, for $\ov n=(4^6)$ and $l(\lambda)\le 4$:
	$$
	\St_{(4^6)}=
	\begin{tikzpicture}[scale=0.4,baseline=(current bounding box.center)]
		\draw[step=1cm] (0,0) grid (6,-4);
		\foreach \j in {1,2,3,4} \node at (\j-0.5,\j-4.5) {{\color{blue}$\bullet$}};
		\draw [blue,line width=1.1pt,->] (3.5,-0.5) -- (0.5,-3.5);
	\end{tikzpicture}
	\qquad\qquad
	\DL_{(4^6)}(\lambda)=\left\{
	\begin{tikzpicture}[scale=0.4,baseline=(current bounding box.center)]
		\draw[step=1cm] (0,0) grid (6,-4);
		\node at (0.5,-3.5) {{\color{blue}$\scriptstyle\lambda_1$}};
		\node at (1.5,-2.5) {{\color{blue}$\scriptstyle\lambda_2$}};
		\node at (2.5,-1.5) {{\color{blue}$\scriptstyle\lambda_3$}};
		\node at (3.5,-0.5) {{\color{blue}$\scriptstyle\lambda_4$}};
	\end{tikzpicture}
	\right\}.
	$$
\end{example}

	\begin{example}
		\label{ex::staircase::poset}
		Let $\ov n=(2\,3^3\,5^2)$, i.e.\ $(2,3,3,3,5,5)$. The first picture below is the staircase shape $\bbY_{\ov n}$. In the second, the shaded hooks show the peeling procedure and the bullets are the five staircase corners. The third is the Hasse diagram of the poset $\St_{\ov n}$ (arrows point from smaller to bigger elements).
		\begin{equation*}
			\begin{array}{ccc}
				\begin{tikzpicture}[scale=0.3,shift={(0,2)}]
					\draw[step=1cm] (0,0) grid (6,-2);
					\draw[step=1cm] (1,0) grid (6,-3);
					\draw[step=1cm] (4,0) grid (6,-5);
				\end{tikzpicture}
				&
				\begin{tikzpicture}[scale=0.3,shift={(0,2)}]
					\fill[cyan!10] (0,0) -- (0,-2) -- (1,-2)--(1,0);
					\fill[cyan!20] (0,-1) -- (0,-2) -- (6,-2)--(6,-1);
					\fill[yellow] (1,0) -- (1,-3) -- (2,-3)--(2,0);
					\fill[yellow] (1,-2) -- (1,-3) -- (6,-3)--(6,-2);
					\fill[magenta!25] (4,0) -- (4,-5) -- (5,-5)--(5,0);
					\fill[magenta!25] (4,-4) -- (4,-5) -- (6,-5)--(6,-4);
					\fill[green!30] (2,0) -- (2,-1) -- (4,-1)--(4,0);
					\fill[lime!40] (5,0) -- (5,-4) -- (6,-4)--(6,0);
					\draw[step=1cm] (0,0) grid (6,-2);
					\draw[step=1cm] (1,0) grid (6,-3);
					\draw[step=1cm] (4,0) grid (6,-5);
					\node (v21) at (0.5,-1.5) {{ \color{blue}$\bullet$ }};
					\node (v32) at (1.5,-2.5) {{ \color{blue}$\bullet$ }};
					\node (v13) at (2.5,-.5) {{ \color{blue}$\bullet$ }};
					\node (v54) at (4.5,-4.5) {{ \color{blue}$\bullet$ }};
					\node (v46) at (5.5,-3.5) {{ \color{blue}$\bullet$ }};
				\end{tikzpicture}
				&
				\begin{tikzpicture}[scale=0.3,shift={(0,2)}]
					\node (v21) at (0.5,-1.5) {{ \color{blue}$\bullet$ }};
					\node (v32) at (1.5,-2.5) {{ \color{blue}$\bullet$ }};
					\node (v13) at (2.5,-.5) {{ \color{blue}$\bullet$ }};
					\node (v54) at (4.5,-4.5) {{ \color{blue}$\bullet$ }};
					\node (v46) at (5.5,-3.5) {{ \color{blue}$\bullet$ }};
					\draw [blue,line width=1.1pt,->] (2.5,-.5) -- (.5,-1.5);
					\draw [blue,line width=1.1pt,->] (2.5,-.5) -- (1.5,-2.5);
					\draw [blue,line width=1.1pt,->] (v54) -- (v46);
				\end{tikzpicture}
				\\
				\begin{array}{c}
					\text{ Young diagram }\\
					\bbY_{(2\,3^3\,5^2)}
				\end{array}
				&
				\begin{array}{c}
					\text{ staircase corners }\\
					\text{ (a rook placement)}
				\end{array}
				&
				\begin{array}{c}
					\text{ Hasse diagram}\\
					\text{ of } \St_{(2\,3^3\,5^2)}
				\end{array}
			\end{array}
		\end{equation*}
		For $\lambda=(2^3\,1^2)$ the set $\DL_{(2\,3^3\,5^2)}(2^3\,1^2)$ consists of the following four arrays (only the entries at the corners are drawn):
		$$
        \DL_{(2\,3^3\,5^2)}(2^3\,1^2) =
        \left\{
		\grDLex{2}{2}{2}{1}{1} \ \prec_{\Bruhat} \
		\grDLex{2}{2}{1}{2}{1} \ \prec_{\Bruhat} \
		\grDLex{2}{1}{1}{2}{2} \ \prec_{\Bruhat} \
		\grDLex{1}{2}{1}{2}{2}
		\right\}.
        $$
		The picture below shows the horizontal and vertical weights of one $DL$-dense array of this shape:
		$$
		\begin{tikzcd}
			\bZ_{\geq 0}^{5}\ni (1,2,3,1,4) =
			\begin{tikzpicture}[scale=0.3,shift={(0,2)}]
				\draw[step=1cm] (0,0) grid (1,-5);
				\node (v21) at (0.5,-.8) {{ \color{blue}\tiny 1 }};
				\node (v32) at (0.5,-1.8) {{ \color{blue}\tiny 2 }};
				\node (v13) at (0.5,-2.8) {{ \color{blue} \tiny 3 }};
				\node (v54) at (0.5,-3.8) {{ \color{blue}\tiny 1 }};
				\node (v46) at (0.5,-4.8) {{ \color{blue}\tiny 4 }};
			\end{tikzpicture}
			&
			\begin{tikzpicture}[scale=0.3,shift={(0,2)}]
				\draw[step=1cm] (0,0) grid (6,-2);
				\draw[step=1cm] (1,0) grid (6,-3);
				\draw[step=1cm] (4,0) grid (6,-5);
				\node (v21) at (0.5,-1.8) {{\color{blue}{\tiny {2}}}};
				\node (v32) at (1.5,-2.8) {{\color{blue}{\tiny {3}}}};
				\node (v13) at (2.5,-.8) {{\color{blue}{\tiny {1}}}};
				\node (v54) at (4.5,-4.8) {{\color{blue}{\tiny {4}}}};
				\node (v46) at (5.5,-3.8) {{\color{blue}{\tiny {1}}}};
			\end{tikzpicture}
			\arrow[l,"\hor"'] \arrow[r,"\vrt"]
			&
			\begin{tikzpicture}[scale=0.3,shift={(0,-0.5)}]
				\draw[step=1cm] (0,0) grid (6,1);
				\node (v21) at (0.5,.3) {{ \color{blue}\tiny 2 }};
				\node (v32) at (1.5,.3) {{ \color{blue}\tiny 3 }};
				\node (v13) at (2.5,.3) {{ \color{blue} \tiny 1 }};
				\node (v54) at (4.5,.3) {{ \color{blue}\tiny 4 }};
				\node (v46) at (5.5,.3) {{ \color{blue}\tiny 1 }};
			\end{tikzpicture}
			= (2,3,1,0,4,1)\in\bZ_{\geq0}^{6}.
		\end{tikzcd}
		$$
	\end{example}

\subsubsection{\textbf{Bubble sort.}}\nlabel{ni:bsintro} The last ingredient is a poset version of the bubble-sort algorithm. Call a composition $\ov d=(d_1,\dots,d_m)\in\bZ^m_{\ge0}$ \emph{$\ov n$-admissible} if, for every $k$, at most as many of $d_1,\dots,d_k$ are non-zero as there are staircase corners in the first $k$ columns. The \emph{bubble-sort map} $\bbs_{\ov n}$ reads $\ov d$ from left to right and moves every new entry to the left by exchanging it with smaller entries, in the way prescribed by the poset $\St_{\ov n}$ (Algorithm~\ref{alg::bb:sort}), until the result is the vertical weight of a $DL$-dense array. For a rectangular shape $\St_{\ov n}$ is a chain and $\bbs_{\ov n}$ is ordinary sorting into decreasing order.
\subsubsection{\textbf{Cauchy identities for staircase matrices}}
	\begin{theoremA}[Corollary~\ref{cor::Cauchy::identity}]
		\label{thm::Cauchy::intro}
		For every staircase shape $\ov n$ we have
		\begin{align}
			\label{eq::Cauchy::bs::intro}
			\prod_{(i,j)\in \bbY_{\ov n}}\frac{1}{1-x_iy_j}&=
			\sum_{A\in\DL_{\ov n}} \kappa_{\hor(A)}(x) \cdot \Bigg(\sum_{\substack{\ov d\ \ov n\text{-admissible}\\ \bbs_{\ov n}(\ov d) = \vrt(A)}} a^{\ov d}(y)\Bigg),\\
			\label{eq::Cauchy::Moebius::intro}
			\prod_{(i,j)\in \bbY_{\ov n}}\frac{1}{1-x_i y_j}&=
			\sum_{\lambda}
			\ \sum_{\substack{A,B\in\DL_{\ov n}(\lambda)\\ A\preceq_{\Bruhat} B}} \mu^{\DL_{\ov n}(\lambda)}_{A,B}\, \kappa_{\hor(A)}(x)\, \kappa^{\vrt(B)}(y),
		\end{align}
		where $\lambda$ runs over all partitions with $l(\lambda)\le\#\St_{\ov n}$.
	\end{theoremA}

	Identity~\eqref{eq::Cauchy::bs::intro} is manifestly positive: every summand is a product of a key polynomial and a sum of Demazure atoms. Identity~\eqref{eq::Cauchy::Moebius::intro} is an alternating (inclusion--exclusion) formula; it is the identity conjectured in~\cite{FKM::Cauchy}. Its coefficients are given by the M\"obius function, which for regular $\lambda$ is described explicitly in Theorem~\ref{thm::poset::Bruhat::intro} below, and which is conjectured to take values in $\{0,\pm1\}$ for all $\lambda$ (Conjectures~\ref{conj::Shell} and~\ref{conj::Moebius}). 
    For rectangular shapes, Identity~\ref{eq::Cauchy::Moebius::intro} reduces to the classical one~\eqref{eq::classical::Cauchy::intro} since, as mentioned in Example~\ref{rem::rectangular::intro} the Bruhat poset $\DL_{(n^m)}(\lambda)$ consists of a single element.

	The identity, expressed through the ``half-bubble-sort'' map $\hb_{\ov{n}}$ instead of $\bbs_{\ov n}$, were proved in~\cite{FKM::Cauchy} by means of the Pieri rules for key polynomials from~\cite{AQ1,AQ2}. The proof given here does not use Pieri rules; on the contrary, the generalized Pieri rules can be recovered from our results, in the same way as the classical Pieri rules follow from the classical Cauchy identity.

	\subsection{Howe duality for staircase matrices}
	\label{intro::Howe}

	Theorem~\ref{thm::Cauchy::intro} is proved by computing the character of $S^\bullet(\Mat_{\ov n})$ in two ways, as in the classical case. The Howe decomposition~\eqref{eq::Howe::intro} induces on each graded piece $S^N(\Mat_{\ov n})$ a filtration $\calF^\lambda$, indexed by the partitions $\lambda\vdash N$ ordered by dominance, by $\fb_n\ttt\fb_m$-sub-bimodules (roughly, $\calF^\lambda$ collects the part of $S^N(\Mat_{\ov n})$ lying in the Howe components $V_\nu\otimes V_\nu^{\op}$ with $\nu\ge\lambda$; the precise definition is~\eqref{eq::FIltrationBimodule} in \S\ref{sec::Howe::staircase}). The associated graded pieces are described by the following theorem.

	\begin{theoremA}[Howe duality for staircase matrices; Theorem~\ref{thm::Howe::lattice}]
		\label{thm::structural::intro}
		For every partition $\lambda\vdash N$ with $l(\lambda)\leq \#\St_{\ov n}$ there is an isomorphism of $\fb_n\ttt\fb_m$-bimodules
		\begin{equation}
			\label{eq::subquot::hor::vrt::intro}
			\gr\calF^{\lambda}\left(S^{N}(\Mat_{\ov n})\right) \simeq \sum_{A\in\DL_{\ov n}(\lambda)} D_{\hor(A)}\otimes D_{\vrt(A)}^{\op} \ \subset\ V_{\lambda}^{\gl_n}\otimes (V_{\lambda}^{\gl_m})^{\op}.
		\end{equation}
	\end{theoremA}

	In other words, the summands of the staircase Howe decomposition are indexed by $DL$-dense arrays $A$, and each is a tensor product of two Demazure modules, of weights $\hor(A)$ and $\vrt(A)$. The sum in~\eqref{eq::subquot::hor::vrt::intro} is, however, \emph{not} direct: the Demazure modules overlap. To extract characters we therefore need to understand sums and intersections of Demazure modules. This is governed by the following not widely known but rather beautiful theorem.

	\begin{theoremA}[Demazure modules form a distributive lattice; Theorem~\ref{thm::Demazure::intersect}]
		\label{thm::Dem::lattice::intro}
		Let $V_\lambda$ be an irreducible finite-dimensional representation of a semisimple Lie algebra $\fg$ with Weyl group $W$ (we only need $\fg=\gl_n$). The Demazure modules $D_{w\lambda}\subset V_\lambda$, $w\in W$, generate a \emph{distributive lattice} $\Lattice_{D}(V_{\lambda})$ under sums and intersections. Equivalently, $V_\lambda$ has a basis such that every $D_{w\lambda}$ is spanned by a subset of it.
	\end{theoremA}

	Consequently, every subspace obtained from Demazure modules by iterated sums and intersections is again a sum of Demazure modules, and its character is the sum of the Demazure atoms $a_\nu$ over all $\nu$ such that $D_\nu$ is contained in it (\S\ref{sec::Demazure::distributive}). Applied to the tensor products $D_{\hor(A)}\otimes D_{\vrt(B)}^{\op}$ and combined with the combinatorics of the next subsection, this turns~\eqref{eq::subquot::hor::vrt::intro} into a genuine character formula: the minimal subquotients of the lattice are tensor products of \emph{generalized van der Kallen modules} (Theorem~\ref{thm::D::Dop::distrib}), whose characters, on each side, are given either by sums of atoms over fibres of a bubble-sort map or, equivalently, by alternating sums of key polynomials with the M\"obius function as coefficients. This gives~\eqref{eq::Cauchy::bs::intro} and~\eqref{eq::Cauchy::Moebius::intro}.

	Theorem~\ref{thm::Dem::lattice::intro} is known to experts. The existence of a basis compatible with all Demazure modules goes back to Kashiwara~\cite{Kashiwara}, who obtained it from crystal bases; Littelmann~\cite{Littleman} gave an algorithm producing such a basis (his standard monomial theory works for all symmetrizable Kac--Moody algebras); a proof based on Frobenius splitting, which we learned from M.~Brion, was extended to affine Lie algebras by Kato~\cite{Kato::Frobenius}. We give a proof along these lines in \S\ref{sec::Demazure::distributive}.

	\subsection{Combinatorics of $DL$-dense arrays}
	\label{intro::combinatorics}

	The poset $\St_{\ov n}$ has two special features. Its Hasse diagram is a forest whose roots are the minimal elements (we call such posets \emph{arborescent}), and the row index $\hor$ and the column index $\vrt$ embed $\St_{\ov n}$ into the linearly ordered sets $[1,n]$ and $[1,m]$ so that, for every corner $s$, the set of corners $\succeq s$ is sent to an interval. We call $\hor$ a \emph{consistent linearization} and $\vrt$ a \emph{consistent anti-linearization} (Definition~\ref{def::poset::linear}). The combinatorial part of the paper works in this generality. For an arborescent poset $\sS$ with an anti-linearization $v\colon\sS\hookrightarrow[1,m]$, a composition $\ov d\in\bZ^m_{\ge0}$ is called \emph{$\sS$-dominant} if $d_l=0$ for $l\notin v(\sS)$ and $d_{v(s)}\ge d_{v(t)}$ whenever $s\succ t$ (Definition~\ref{def::S::dominant}). For $\sS=\St_{\ov n}$ and $v=\vrt$ these are exactly the vertical weights of $DL$-dense arrays. Let $\bD_{\sS}(\lambda)$ be the set of $\sS$-dominant compositions whose non-zero entries are the parts of $\lambda$, with the Bruhat order inherited from the rearrangements $\bS_m\lambda$ of $\lambda$.

	\begin{theoremA}[Bruhat order on dominant compositions]
		\label{thm::poset::Bruhat::intro}
		Let $\sS$ be an arborescent poset with a consistent anti-linearization and let $\lambda$ be a partition.
		\begin{enumerate}[label=(\roman*),itemsep=1pt,topsep=2pt]
			\item $\bD_{\sS}(\lambda)$ is empty if $l(\lambda)>\#\sS$.
			\item (Theorem~\ref{thm::dominant::edge}) If $l(\lambda)\le\#\sS$, the covering relations of the Bruhat order on $\bD_{\sS}(\lambda)$ are given by exchanging two entries that form a \emph{minimal disorder} (Definition~\ref{def::disorder}).
			\item (Corollaries~\ref{cor::bbs::monotone} and~\ref{cor::bbs::monotone:1}) If $l(\lambda)\le\#\sS$, the bubble-sort map $\bbs_{\sS}$ is a monotone, order-decreasing idempotent from the set of $\sS$-admissible rearrangements of $\lambda$ onto $\bD_{\sS}(\lambda)$.
			\item (Theorem~\ref{thm::S:dom::outline}, Corollary~\ref{cor::Mobius::DL}) If $\lambda$ is \emph{regular} of length $\#\sS$ (that is, it has exactly $\#\sS$ parts, all distinct), then $\bD_{\sS}(\lambda)$ does not depend on $\lambda$; it is bounded, graded and $\EL$-shellable, and its M\"obius function is
			\[
			\mu_{a,b}=\begin{cases}(-1)^{\rk(b)-\rk(a)}, & \text{if every rearrangement strictly between $a$ and $b$ is $\sS$-dominant},\\ 0,& \text{otherwise.}\end{cases}
			\]
			\item (Equation~\eqref{eq::proj::parab::D}) There is a monotone idempotent from the regular case onto $\bD_{\sS}(\lambda)$ which accounts for the difference between regular and non-regular $\lambda$.
		\end{enumerate}
	\end{theoremA}
	
	Applied to $\sS=\St_{\ov n}$, once with the anti-linearization $\vrt$ and once with the linearization $\hor$, Theorem~\ref{thm::poset::Bruhat::intro} shows that $\hor$ and $\vrt$ identify $\DL_{\ov n}(\lambda)$ with posets of dominant compositions inside $\bS_n\lambda$ and $\bS_m\lambda$ respectively, and that the two induced Bruhat orders coincide (Corollary~\ref{cor::hor::vrt::Bruhat}); this is what makes the order on $\DL_{\ov n}(\lambda)$ used in Theorem~\ref{thm::Cauchy::intro} well defined.

	Theorem~\ref{thm::poset::Bruhat::intro} generalizes a classical special case. If $\St_{\ov n}$ is a disjoint union of chains (this happens when $\Mat_{\ov n}$ is a parabolic subalgebra of $\gl_n$ with parabolic subgroup $\bS_{\ov n}\subset\bS_n$), then $\DL_{\ov n}(\lambda)$ is the parabolic Bruhat graph and $\bbs_{\ov n}$ sends a permutation $w$ to the element of minimal length in the coset $w\bS_{\ov n}$. In the general case the poset $\DL_{\ov n}(\lambda)$ need not be graded for non-regular $\lambda$ (Example~\ref{ex::nongrad}); this is where the conjectures mentioned above enter.

	\subsection{Structure of the paper, notation}
	\label{intro::structure}

	The paper consists of a combinatorial part (\S\ref{sec::Recollection}, \S\ref{sec::Bubble::sort::all} and \S\ref{sec::Staircase}, except its last subsection~\S\ref{sec::DL::Distr}), which can be read independently of Lie theory, and a representation-theoretic part (\S\ref{sec::Demazure::all}, \S\ref{sec::DL::Distr} and \S\ref{sec::Howe::duality}).
	\begin{itemize}[itemsep=1pt,topsep=2pt]
		\item \S\ref{sec::Recollection} recalls terminology for posets (\S\ref{sec::poset::terms}), the Bruhat graph (\S\ref{sec::Bruhat}) and distributive lattices of subspaces (\S\ref{sec::Distributive}).
		\item \S\ref{sec::Bubble::sort::all} is purely combinatorial: it introduces dominant compositions of an arborescent poset (\S\ref{sec::S::dominant}), describes their Hasse diagram (\S\ref{sec::Hasse::dominant}), defines the bubble-sort algorithm (\S\ref{sec::Bubble-sort}) and proves Theorem~\ref{thm::poset::Bruhat::intro} for regular $\lambda$ (\S\ref{sec::S::Regular}); the non-regular case is discussed in \S\ref{sec::S::parabolic}.
		\item \S\ref{sec::Demazure::all} recalls Demazure modules (\S\ref{sec::Demazure::def}) and proves Theorem~\ref{thm::Dem::lattice::intro} (\S\ref{sec::Demazure::distributive}).
		\item \S\ref{sec::Staircase} defines staircase corners (\S\ref{sec::staircase::corners}), $DL$-dense arrays (\S\ref{sec::DL_dense}) and their Bruhat order (\S\ref{sec::Bruhat::DL-dense}), and builds the double Demazure lattice (\S\ref{sec::DL::Distr}).
		\item \S\ref{sec::Howe::duality} proves Theorems~\ref{thm::structural::intro} and~\ref{thm::Cauchy::intro}.
		\item Appendix~\ref{sec::Example::DL} contains pictorial examples of the posets $\bD_{\sS}(\lambda)$ and $\DL_{\ov n}(\lambda)$.
	\end{itemize}
	The paper uses a fair amount of notation. To help the reader, the \emph{Index of notation} at the end of the paper (page~\pageref{sec::notation}) lists the symbols by topic together with a short description and a reference to the place where each of them is introduced.

	\subsection*{Acknowledgements}
	We are grateful to Michel Brion for useful correspondence that leads to the main idea of Theorem~\ref{thm::Dem::lattice::intro} and Evgeny Feigin, Sergei Ivanov, Igor Makhlin and Alek Vainstein for useful conversations. The work of Ie. M. partially supported by the International Scientists Program of the Beijing Natural Sciences
	Foundation (grant number IS24002).

	\section{Recollection of (Bruhat) Posets and Distributive Lattices}
	\label{sec::Recollection}
	
	\subsection{Recollection of poset terminology}
	\label{sec::poset::terms}
	Let us recall some properties of partially ordered sets that are important for our purposes.
	We refer to~\cite{Wachs} and references therein for the introduction to the theory of posets. 
	All posets in this paper are finite and, in particular, all intervals\nlabel{ni:interval}
	$$(x,y):=\{z\colon x\prec z\prec y\}$$ 
	are finite.

	\begin{definition}
		\label{def::bounded::graded}
		A poset $(X,\preceq)$ is called 
		\begin{itemize}
			\item \emph{bounded} if there exists an
			element $\hat{1}$ greater or equal than every other element in $X$, and an element $\hat{0}$ less or equal to every other element in $X$. 
			\item \emph{graded} (also called \emph{ranked} or \emph{pure}) if all chains between two comparable elements have the same length.
		\end{itemize}    
		The length of the maximal chain from $\hat{0}$ to $x$ in a bounded graded poset is called \emph{the rank function} and is denoted by $\rk(x)$. 
	\end{definition}
	
	The \emph{Hasse graph} or \emph{Hasse diagram} $\Gamma_X$ of the poset $X$ is a directed graph whose vertices are elements of $X$ and edges represent the covering relations in the poset $X$. For the pictorial description, we orient the edges from the smaller element to the bigger one.\nlabel{ni:Hasse}
	
	\begin{definition}
		\label{def::thin}
		A graded poset $(X,\preceq)$ is called 
		\begin{itemize}
			\item \emph{thin} if all open intervals of length $2$ consist of two elements: 
			$$\forall x\prec y \ \& \ \rk(y)-\rk(x)=2 \ \Rightarrow \ (x,y):= x\prec z_1,z_2 \prec y.$$
			\item \emph{subthin} if all open intervals of length $2$ consist of at most two elements.
		\end{itemize} 
	\end{definition}

	\begin{definition}	
		\label{def::ELlabelling}
		An  $A$-\emph{labelling} of the poset $(X,\prec)$	is the map $\euE$ from the set of edges of the Hasse diagram $\Gamma_X$ to a given set $A$.
		
		The $A$-labelling is called \emph{$\EL$-labelling} (Edge-Lexicographical) if $A$ is a poset and for any pair of comparable elements $x<y\in X$ 
		there exists a unique maximal chain 
		$$\pth_{\euE}(x\prec y):=\left(x\prec z_1\prec \ldots\prec z_k\prec y\right)$$ going from $x$ to $y$ which has increasing $\euE$ labels (when reading the covering
		relations from minima to maxima)
		$$
		\euE(x\prec z_1) < \euE(z_1\prec z_2) < \ldots < \euE(z_k\prec y)
		$$
		and, moreover, this unique maximal chain is minimal for
		the lexicographic order on maximal chains (comparing the words given by the successive $\euE$ labels).
	\end{definition}
	
	\begin{definition}
		\label{def::ELshellable}
		A poset $X$ is called \emph{$\EL$-shellable} if there exists an $\EL$-labelling of its Hasse diagram.    
	\end{definition}
	
	We refer to~\cite{Wachs} and to~\cite{BB, Bjoner_Shellable} for the details of the aforementioned notions and applications to Coxeter groups.
	
	\begin{definition}
		\label{def::idempotent}
		A map $F : X \rightarrow X$ on the poset $(X,\leq)$ is called 
		\begin{itemize}
			\item \emph{decreasing} if $F(x)\leq x$ (respectively \emph{increasing} if $F(x)\geq x$);
			\item \emph{monotone} if $\forall x\leq y$ we have $F(x) \leq F(y)$;
			\item \emph{idempotent} if $\forall x\in X \  F(x)=F(F(x))$. 
		\end{itemize}
	\end{definition}
	
	\begin{lemma}
		\label{lem::poset::idempotent}
		Suppose $F:X\to X$ is a monotone idempotent. Then
		\begin{itemize}
			\item if $F$ is decreasing then 
			$\forall x\in X$ the element $F(x)$ is the supremum element among the elements $y\in F(X)$ that are less or equal to $x$:
			$$\forall x\in X \ F(x)= \sup\{y\in F(X) \colon y \leq x\}.$$
			\item Similarly, if $F$ is an increasing monotone idempotent then 
			$$\forall x\in X \ F(x)= \inf\{y\in F(X) \colon y \geq x\}.$$
		\end{itemize}
	\end{lemma}
	\begin{proof} 
		Suppose that $F$ is decreasing and $F(x')\leq x$. Then since $F$ is an order-preserving map we have:
		\[F(x')=F(F(x'))\leq F(x).\]
	\end{proof}
	In particular, we see that if $x<x'$ and $x<x''$ is a pair of different edges in the Hasse diagram of $F(X)$ then the open intervals $(x,x')$ and $(x,x'')$ of the poset $X$ do not intersect.

	\begin{definition}
		\label{def::Mobius}
		Let \( X \) be a locally finite poset. The \emph{M\"obius function} \( \mu^{X}: X \times X \to \mathbb{Z} \) is defined recursively as follows:
		\begin{itemize}
			\item For all \( x \in X \), \( \mu_{x,x} = 1 \).
			\item For all \( x, y \in X \) such that \( x < y \), the Möbius function is given by:
			\[
			\mu_{x,y} = -\sum_{z\in [x, y)} \mu_{x,z},
			\]
			where the sum is taken over all \( z \in X \) such that \( x \leq z < y \).
			\item If \( x \not\leq y \), then \( \mu_{x,y} = 0 \).
		\end{itemize}
	\end{definition}
	The main application of the M\"obius function is the following 
	M\"obius Inversion Formula
	\begin{fact}
		\label{fact::Mobius}	
		Let \( f, g: X \to \mathbb{C} \) be two functions on a locally finite poset $X$ related by the following summation:
		\[
		\forall x\in X \quad
		g(x) = \sum_{y \leq x} f(y).
		\]
		Then, \( f \) can be expressed in terms of \( g \) using the M\"obius function \( \mu^{X} \) of the poset $X$:
		\[
		\forall y\in X \quad
		f(y) = \sum_{x \leq y} \mu^{X}_{x,y} g(x).
		\]
	\end{fact}

	\begin{definition}
		\label{def::ordercomplex}
		The \emph{order complex $\Delta(X)$} of a poset $X$ is the simplicial complex whose vertices are indexed by elements of $X$ and simplices $\{x_0,\ldots,x_k\}$ are in one-to-one correspondence with chains $x_0<\ldots<x_k$ in the poset $X$. 
	\end{definition}
	
	\begin{fact}
		The M\"obius function $\mu^{X}_{x,y}$ is equal to the reduced Euler characteristic of the geometric realization of the order complex of the open interval $(x,y)$:
		$$\mu^{X}_{x,y} = \tilde\chi(|\Delta(x,y)|).$$
	\end{fact}

	The following facts follow from the seminal \emph{Quillen's Theorem A} (\cite{Quillen}, see also \cite[Th.2.3]{Walker} and \cite[Corollary 10.12 on p.1853]{Combinatorics_Book}):
	
	\begin{fact}
		If the poset $X$ has a minimal element then $\Delta(X)$ is contractible.
	\end{fact}
	\begin{corollary}\label{fact::subposet}
		Let $F : X \rightarrow X$ be an order-preserving idempotent then
		\begin{itemize}
			\item $\forall y\in F(X)$ the simplicial complex $\Delta(F^{-1}(y))$  is contractible;
			\item 
			The order complexes $\Delta(X)$  and of $\Delta(F(X))$ are homotopy equivalent.
		\end{itemize}
	\end{corollary}
	
	The subsequent fact was stated by Danaraj-Klee for simplicial complexes (\cite{Shellable_Thin}) and later on was generalized to posets by Bjorner (\cite{Bjorner_Shellable}). 
	It has numerous applications for the parabolic BGG resolution:
	\begin{fact}
		\label{fact::Shell::ball}
		If $(X,\leq)$ is a shellable subthin graded poset then the geometric realization of an ordered complex $\Delta(X)$ is homeomorphic to
		\begin{itemize}
			\item a sphere, if the poset $X$ is thin;
			\item a ball, if the poset $X$ is subthin but not thin.
		\end{itemize}
	\end{fact}
	
	\begin{corollary}
		\label{cor::Mobius::thin}
		If $X$ is a graded bounded subthin shellable poset then $\forall x<y \in X$ we have
		$$
		\mu^X_{x,y} = \begin{cases}
			(-1)^{\rk(y)-\rk(x)}, \text{ if the interval $[x,y]$ is thin},\\
			0, \text{ if the interval $[x,y]$ is subthin, but not thin.}
		\end{cases}
		$$
	\end{corollary}

	\subsection{Bruhat Order and Bruhat Graph}
	\label{sec::Bruhat}
	
	Let $W$ be the Weyl group associated with a semisimple Lie algebra $\fg$ and the root system $\Phi$.  
	We assume that the Cartan decomposition $\fg=\fn_{-}\oplus\fh\oplus\fn_{+}$ is fixed, and that the set of simple reflections $S\subset W$ corresponds to this decomposition.  
	The group $W$ is naturally equipped with the \emph{length function}, which counts the number of simple reflections in a reduced decomposition.  
	
	Moreover, for any weight $\lambda\in P$ in the weight lattice, we define:\nlabel{ni:lambdaplus}
	\begin{itemize}
		\item the \emph{dominant weight} $\lambda_+ := W\lambda\cap P_{+}$ in the $W$-orbit of $\lambda$;
		\item the \emph{Weyl subgroup} $W_J := \mathrm{Stab}(\lambda_+)\subset W$, generated by reflections that fix $\lambda_+$;
		\item the element $\sigma_\lambda\in W$ of minimal length such that $\lambda = \sigma_{\lambda}\lambda_+$, which represents the left coset $W/W_J$.
	\end{itemize}
	
	\begin{definition}
		\label{def::Bruhat}
		The \emph{Bruhat partial order} on the $W$-orbit $W\lambda$ is defined as follows.  
		For $\lambda=\sigma_\lambda \lambda_+$ and $\eta=\sigma_\eta \lambda_+$, we set  
		\[
		\lambda \preceq \eta \quad \text{if and only if} \quad \Phi_-\cap \sigma_\lambda (\Phi_+) \subset  \Phi_-\cap \sigma_\eta (\Phi_+).
		\]
		The {\rm Hasse diagram} of the poset $W\lambda\simeq W/W_{J}$ is called the \emph{parabolic Bruhat graph}.  
		If $W_J$ is trivial, it is called the \emph{Bruhat graph}.
	\end{definition}
	
	For further details, we refer to the classical textbooks on this subject~\cite{Humphreys, BB}.
	
	\begin{fact}
		\label{fact::Bruhat::graph}
		The (parabolic) Bruhat graph $W\lambda$ is {\rm bounded} and {\rm graded}, with  
		\[
		\hat{0}:=\lambda_+, \quad \hat{1}:=\lambda_{-}=\omega_0\lambda_+,
		\]
		where $\omega_0$ (respectively, $\omega_0^{J}$) is the {longest element} of the Weyl group $W$ (resp. the Weyl subgroup $W_J$).  
		The {\rm rank function} coincides with the length function:  
		\[
		\rk(\lambda) = l(\sigma_\lambda).
		\]
		
		Moreover, the Bruhat graph is {\rm thin} for regular $\lambda$ (i.e., when $W\lambda \simeq W$), while the parabolic Bruhat graph is {\rm subthin}.  
		Additionally, $W\lambda$ admits an \emph{\(\EL\)-labelling}.
	\end{fact}
	
	For this paper, we focus on the case $\fg=\gl_n$.  
	The corresponding Weyl group $W$ is the \emph{symmetric group} $\bS_n$, and the set of weights we consider consists of compositions $\bZ^{n}_{\geq 0}$.\nlabel{ni:Sn}
	
	When $W=\bS_n$, the {\it "reflections"} are {\it "transpositions"}, and the set of {\it "simple reflections"} consists of {\it "consecutive"} transpositions.
	This allows us to express the length function and the edges of the Bruhat graph in terms of compositions.  
	In particular, if $\lambda_+=(\lambda_1>\ldots>\lambda_m)$ is a {\it "strictly decreasing partition"} also called regular, then:
	\begin{gather}
		\label{eq::length::S_n}
		\text{For } \nu\in\bS_m\lambda_+ , \quad \rk(\nu) = l(\nu) := \#\{ i < j \colon \nu_i < \nu_j \}; \\
		\label{eq::Bruhat::edge}
		\begin{array}{c}
			\nu \stackrel{\prec}{\longrightarrow} (ij)\nu  \\
			\text{ is an edge in the Bruhat graph } \bS_m\eta
		\end{array}
		\ \  \Leftrightarrow \ \
		\begin{cases}
			\nu_i > \nu_j, \\ 
			{\small \forall k = i+1,\dots, j-1, \quad \nu_i\notin [\nu_j,\nu_i].}
		\end{cases}
	\end{gather}
	
	A detailed combinatorial description of the (strong) Bruhat order for $\bS_m$ can be found in~\cite[Th.2.1.5]{BB}.
	
	\begin{fact}
		The labelling of the edges of the Bruhat graph $\bS_m$ by transpositions defines an \emph{\(\EL\)-labelling} if we consider the following \emph{total order} on transpositions $\{(ij) :1\leq i<j \leq n\}$:
		\begin{equation}
			\label{eq::order::transpos}
			(il) \triangleleft (jk) \quad \stackrel{\mathsf{def}}{\Leftrightarrow} \quad
			\begin{cases}
				l < k, \\
				l = k \ \text{and} \ i > j.
			\end{cases}
		\end{equation}
	\end{fact}
	
	It is well-known that for any given parabolic subgroup $W_J$ of a given Weyl group $W$ each coset $w W_J$ is an interval in the Bruhat graph $W$ with the unique element of minimal $w^+$ and maximal $w^{-}$ length correspondingly. Moreover, the assignments $w\mapsto w^+$ (resp. $w\to w^{-}$) are monotone nonincreasing (respectively nondecreasing) idempotents acting on $W$ (see e.g.~\cite[\S2.5]{BB}). Let us adopt this observation to our case of a symmetric group.
	
	From now on, we suppose that $\lambda$ is a partition with repeating elements. Then for $\delta:=(m-1,m-2,\ldots,1,0)$ the partition $\lambda+\delta$ has no equal elements.  Consequently, $\bS_{m}(\lambda+\delta)\simeq \bS_m$. Let $\bS_{\lambda}$ be the subgroup that stabilizes partition $\lambda$. To each permutation $\sigma$ we assign $\sigma_{\lambda}^{+}$ and $\sigma_{\lambda}^{-}$ that are elements of minimal and maximal length among those permutations whose action on $\lambda$ is equal to $\sigma\lambda$. Then the assignments:
	$$\psi_{\lambda}^{+}:\bS_{m} \stackrel{\sigma\mapsto\sigma_{\lambda}^{+}}{\longrightarrow} \bS_{m}, \quad 
	\psi_{\lambda}^{-}:\bS_{m} \stackrel{\sigma\mapsto\sigma_{\lambda}^{-}}{\longrightarrow} \bS_{m}
	$$
	are known to be monotone idempotents (nonincreasing and nondecreasing correspondingly). The images of these idempotents are isomorphic to the poset $\bS_m\lambda$ and we end up with the order-preserving projection that forgets elements from $\delta$: 
	\begin{equation}
		\label{eq::proj::parabolic}
		\pi_{\lambda}:\bS_m(\lambda+\delta) \twoheadrightarrow \bS_m\lambda,
	\end{equation}
	whose left and right adjoints are called by the same letters $\psi_{\lambda}^{+}$ and $\psi_{\lambda}^{-}$ correspondingly.
	
	\subsection{Distributive Lattice of Vector Subspaces}
	\label{sec::Distributive}
	
	\begin{definition}
		\label{def::lattice}
		A \emph{lattice} 
		is a partially ordered set $(\Lattice,\le)$ in which every pair of elements $a,b\in\Lattice$ has a unique \emph{supremum} $a\vee b$ (also called the \emph{least upper bound} or \emph{join}) and a unique \emph{infimum} $a\wedge b$ (also called the \emph{greatest lower bound} or \emph{meet}). Both operations are monotone with respect to the given order: 
		$$ (a_1\le a_2)\ \& \ (b_1\le b_2) \ \Rightarrow \ (a_{1}\vee b_{1}\leq a_{2}\vee b_{2})\ \& \ ( a_{1}\wedge b_{1}\leq a_{2}\wedge b_{2}). $$
		
		A lattice $\Lattice$ is called \emph{distributive} if it satisfies the following identity: 
		$$
		\forall x,y,z\in\Lattice \qquad
		x\wedge (y\vee z) = (x\wedge y) \vee (x\wedge z).
		$$
	\end{definition}
	
	Throughout this paper, we assume that all lattices under consideration are finite.
	
	\begin{example} (\emph{Boolean lattice})
		The power set $2^{M}$ of a given set $M$ forms a distributive lattice with the following operations:
		$$
		\begin{array}{c}
			S\leq T \ \Leftrightarrow \ S\subset T, \\
			S\wedge T:= S\cap T, \quad S\vee T:= S\cup T.
		\end{array}
		$$
	\end{example}
	
	Birkhoff's theorem (\cite{Birkhoff}, see also~\cite[\S5]{Lattice}) states that any finite distributive lattice is a sublattice of a Boolean lattice.  
	The proof of this theorem is straightforward and relies on the following notion:
	
	\begin{definition}
		\label{def::joinirr}
		An element $V$ of a distributive lattice $\Lattice$ is called \emph{join-irreducible} (or $\vee$-irreducible) if, whenever $V=V_1\vee V_2$, we must have either $V=V_1$ or $V=V_2$.
		
		The set of all $\vee$-irreducible elements of $\Lattice$ is denoted by $\euJ(\Lattice)$.
	\end{definition}
	
	In particular, for any two $\vee$-irreducible elements $V_1,V_2$ in a distributive lattice $\Lattice$, the following holds:
	$$ V_1\le V_2 \ \Leftrightarrow \ V_1\vee V_2 = V_2.
	$$
	
	\begin{example}
		\label{ex::poset::lattice}
		To each finite poset $(X,\le)$, we associate a sublattice of the Boolean lattice $2^{X}$ consisting of \emph{upper sets}:
		$$\Lattice_X:= \{ S\subset X\ \colon\ x\in S \Rightarrow \forall y\ge x \ y\in S \}.$$
		The set of $\vee$-irreducible elements of $\Lattice_X$ is indexed by elements of $X$ and consists of the upper sets $X_{\ge x}:=\{y\in X \colon y\geq x\}$.
	\end{example}
	
	\begin{fact}(\cite[\S5]{Lattice})
		\label{fact::distr::poset}
		The assignments $(X,\le) \mapsto \Lattice_{X}$ and $\Lattice \mapsto \euJ(\Lattice)$ define a pair of mutually inverse bijections between the set of finite posets and the set of finite distributive lattices.
	\end{fact}
	
	However, in our applications, we are primarily interested in distributive lattices of vector subspaces.  
	The set $\Lattice(U)$ of all vector subspaces of a given vector space $U$ forms a poset under inclusion and a lattice under the operations of intersection and sum:\nlabel{ni:LU}
	$$
	\forall V_1,V_2\subset U \quad
	\begin{cases}
		V_1\leq V_2 \ \stackrel{\mathsf{def}}{\Leftrightarrow} \ V_1\subset V_2;\\
		V_1\wedge V_2:=V_1\cap V_2, \quad V_1\vee V_2:= V_1+V_2.
	\end{cases}
	$$
	
	\begin{definition}
		\label{def::Lgenerated}
		For a given collection of vector subspaces $V_1,\ldots, V_k \subset U$, we define the lattice $\Lattice(V_1,\ldots, V_k; U)$ as the smallest sublattice of $\Lattice(U)$ that contains all $V_i$.
	\end{definition}
	
	Note that while the lattice $\Lattice(U)$ may be infinite, its sublattice $\Lattice(V_1,\ldots, V_k; U)$ is always finite.
	The following proposition is a well-known fact from linear algebra (see, e.g.,~\cite[Chapter 1, Proposition 7.1]{PP::Quadratic} and references therein):
	\begin{proposition}
		\label{prp::distr::vect}
		The lattice $\Lattice(V_1,\ldots,V_k;U)\subset \Lattice(U)$ is distributive if and only if there exists a basis $\{e_1,\ldots,e_n\}$ of the ambient space $U$ such that for each $i=1,\ldots,k$, the space $V_i$ has a basis given by $V_i\cap\{e_1,\ldots,e_n\}$.
	\end{proposition}
	
	In particular, under the assumptions of Proposition~\ref{prp::distr::vect}, the mapping:
	$$
	\psi: V_{i} \mapsto V_i\cap\{e_1,\ldots,e_n\} \subset \{e_1,\ldots, e_n\}
	$$
	defines an embedding of the distributive lattice $\Lattice(V_1,\ldots,V_k;U)$ into the Boolean lattice of subsets of the basis $\{e_1,\ldots,e_n\}$,
	where the set $\psi(V_i)$ forms a basis of the space $V_i$.
	
	From now on, we consider a distributive lattice $\Lattice$ of vector subspaces of a given vector space $U$, generated by its subset of $\vee$-irreducible vector subspaces:
	$$
	\euJ(\Lattice):=\{V_1,\ldots,V_k\}.
	$$
	We assume that $(0)$ and $U$ belong to $\euJ(\Lattice)$ and that the induced poset $\euJ(\Lattice)$ is bounded:
	$$
	i \leq j \quad \stackrel{\mathsf{def}}{\Leftrightarrow} \quad V_i \subset V_j.
	$$
	Any element in the distributive lattice is a join of the $\vee$-irreducible elements, what means that any subspace $V\in\Lattice$ satisfies:
	$$
	V = \vee_{V_i\subset V} V_i = \sum_{V_i\subset V} V_i.
	$$
	The set $\euJ(\Lattice)$ defines a \emph{standard} filtration applied to both $U$ and any subspace $V$ in $\Lattice$:
	\begin{equation}
		\label{eq::filtration::distributive}
		\calF_\Lattice^{i} V:=V\wedge V_i = V\cap V_i, \quad i\le j \Rightarrow \calF_\Lattice^{i} V \subset \calF_\Lattice^j V.
	\end{equation}
	
	Let $E:=\{e_1,\ldots,e_{\dim V}\}$ be a given common basis of $\Lattice$.
	To each $\vee$-irreducible space $V_i$, we associate a subset\nlabel{ni:Ei}
	$E^i:=\{e^i_1,\dots,e^i_{r_i}\}\subset E$, consisting of basis elements that belong to $V_i$ but not to any other proper subspace $V_j\subsetneq V_i$.
	From this definition, we obtain:
	
	\begin{fact}
		\label{fact::distributive}
		\begin{itemize}[itemsep=0pt, topsep=0pt]
			\item 
			Each space is spanned by the union of basis elements:
			$V_i = \mathsf{span}\langle \cup_{V_j \subset V_i} E^j\rangle.$
			\item
			The elements from $E^i$ form a basis for the space
			\begin{equation}
				\label{eq::min::subquot}
				\KK_i:=V_i\left/\sum_{V_j \subsetneq V_i}V_j\right.
			\end{equation}
			called the \emph{minimal subquotient in $\Lattice$}.
			\item 
			The subquotient
			$$
			\calF_\Lattice^i V/ \calF_{\Lattice}^{<i} V \simeq \begin{cases}
				\KK_i, \text{ if } V_i\subset V,\\
				0, \text{ otherwise. }
			\end{cases}
			$$
			\item 
			The basis of the quotient $\bar{V}=V'/V''$ for $V',V'' \in \Lattice$ consists of the (disjoint) union of the sets $E^i$ for indices $i$ satisfying $(V_i\subset V')\ \& \ (V_i \not\subset V'')$.
		\end{itemize}
	\end{fact}
	
	For any subset $V_{i_1},\ldots, V_{i_r} \subset \Lattice_{V}$ of $\vee$-irreducible subspaces, we define:
	$$
	\euJ^{\leq i_1,\ldots,i_r}:=\{ V_j \colon (j\leq i_1)\ \& \ldots \& \ (j \leq i_r)\} \subset \euJ.
	$$
	Then we have the equality: 
	$$
	V_{i_1}\wedge \ldots \wedge V_{i_r} = \vee_{V\in\euJ^{{\leq i_1,\ldots,i_r}}} V =  \vee_{V\in \max(\euJ^{{\leq i_1,\ldots,i_r}} )}V.
	$$
	The latter follows from the fact that all $V_i$'s are $\vee$-irreducible.
	
	\begin{proposition}
		\label{prp::distibutive::resolution}	
		Each subquotient $\bar{V}:= V'/V''$ of two subspaces $V', V''$ in a distributive lattice $\Lattice$ admits a resolution by a direct sum of $\vee$-irreducible elements.
	\end{proposition}
	\begin{proof}
		We prove by induction on $i$ that the minimal quotient $\KK_i$ (defined in~\eqref{eq::min::subquot}) has a resolution by $\vee$-irreducible elements $V_j$ with $j\leq i$.
		
		The base case is clear: when $i$ is minimal, we have $\KK_i = V_i$.
		
		Suppose the statement holds for all $j<i$. Then, by gluing these resolutions, we obtain a resolution for any quotient $\bar{U}:= U_1/U_2$ with $U_1\subsetneq V_i$. Let $V_{j_1},\ldots,V_{j_r}$ be the maximal $\vee$-irreducible subspaces satisfying $V_j\leq V_i$. Then, we have the short exact sequence:
		$$
		\begin{tikzcd}
			0 \arrow[r] & V_{j_1}+\ldots+ V_{j_r} \arrow[r] & V_i \arrow[r] & \KK_i \arrow[r] & 0.
		\end{tikzcd}
		$$
		Applying inclusion-exclusion, we obtain the long exact sequence:
		$$
		\begin{tikzcd}[row sep=1em, column sep=1em]
			0\arrow[r] & \bigcap_{s=1}^{r}V_{j_s} \arrow[r] & \ldots \arrow[r] & \bigoplus_{\substack{{\mathsf{J}\subset [1r]}\\ {|\mathsf{J}| = l} }} \bigcap_{s\in \mathsf{J}} V_{j_s} \arrow[r] & \ldots \arrow[r] & \bigoplus_{s=1}^{r} V_{j_s} \arrow[r] & +_{s=1}^{r} V_{j_s} \arrow[r] & 0.
		\end{tikzcd}
		$$
		By induction, each intersection is resolved by $\vee$-irreducible elements, hence so is $\KK_i$.
	\end{proof}
	\begin{proposition}
		\label{prp::distributive::idempotent}	
		Let $F:\euJ(\Lattice) \to \euJ(\Lattice)$ be an increasing monotone idempotent on the poset $\euJ(\Lattice)$ of $\vee$-irreducible elements in a given distributive lattice $\Lattice$.
		Then the image $F(\euJ(\Lattice))$ forms a set of $\vee$-irreducible elements in the distributive sublattice $\Lattice_{F}\subset \Lattice$ generated by $F(\euJ(\Lattice))$.
		
		Moreover, for every $l\in F(\euJ(\Lattice))$, the minimal subquotients in the sublattice $\Lattice_{F}$ are given by:
		\begin{equation}
			\label{eq::quot::idempotent}
			\KK^{F}_{l}:= V_{F(i)}\left/\sum_{V_{F(j)}\subset V_{F(i)}} V_{F(j)} \right. \simeq \bigoplus_{l\colon F(l)=F(i)} V_l \left/ 
			\sum_{V_{j}\subset V_{l}} V_{j}.
			\right.
		\end{equation}
	\end{proposition}
	
	\begin{proof}
		The proof follows from a direct application of Birkhoff's bijection between posets and distributive lattices, as discussed in Example~\ref{ex::poset::lattice}.
		
		First, note that any sublattice of a distributive lattice remains distributive. Second, any $\vee$-irreducible element of the original lattice that belongs to the sublattice remains $\vee$-irreducible within it. Finally, using Lemma~\ref{lem::poset::idempotent}, we conclude that there are no additional $\vee$-irreducible elements in $\Lattice_{F}$ beyond those in the set $\{V_{F(i)} \colon V_i\in \euJ(\Lattice)\}$.
		Isomorphism~\eqref{eq::quot::idempotent} follows from a direct examination of the elements of the basis $E$. Specifically, the union $\cup_{l\in F^{-1}(F(i))} E^l$ constitutes a basis for both the left-hand and right-hand sides of~\eqref{eq::quot::idempotent}, establishing the claim.
	\end{proof}

	\section{Bubble-Sort Algorithm for Arborescent Posets with (Anti)Linearization}
	\label{sec::Bubble::sort::all}
	
	This section is purely combinatorial and independent of the rest of the paper.
	We generalize the "bubble-sort" algorithm for arborescent posets $\mathsf{S}$ with (anti)linearization and describe the Bruhat partial order and its properties on the set of monotone functions over these posets. 
	
	Later in the paper, these posets will appear as the set of staircase corners $\St_{\ov{n}}$ (Proposition~\ref{prp::poset->partition}). The Bruhat poset of monotone functions described in this section will resurface in~\S\ref{sec::Staircase}, playing a key role in defining the vertical and horizontal weights of $DL$-dense arrays (Definition~\ref{def::DL::dense}), which form the main combinatorial foundation for the generalization of Howe duality for staircase matrices.
	
	\subsection{Arborescent Posets with Consistent (Anti)Linearization}
	\label{sec::S::dominant}
	
	\begin{definition}
		\label{def::poset::linear}	
		\begin{itemize}
			\item
			A poset $(\sS,\succ)$, together with an order-reversing monomorphism $v:(\sS,\succ)\hookrightarrow ([1m],<)$ that satisfies one of the two following equivalent conditions:
			\begin{equation}
				\label{eq::poset::st}
				\begin{array}{c}
					\forall s,t,r \in \sS \colon \left[(v(s)<v(t) < v(r))\, \& \, (s\succ r)\right] \quad \Rightarrow \quad (t\succ r); \\
					\Leftrightarrow \ 
					\forall s\in \sS \ \exists a_s\leq b_s\leq m \ \colon \   \sS_{\succeq s}=v^{-1}([a_s,b_s]).
				\end{array}
			\end{equation}
			is called an {\it arborescent poset with consistent anti-linearization}.
			\item 
			A poset $(\sS,\succ)$, together with an order-preserving embedding $h:(\sS,\succ)\hookrightarrow ([1n],<)$ that satisfies one of the two following equivalent conditions:
			\begin{equation*}
				\label{eq::poset::st:h}
				\begin{array}{c}
					\forall s,t,r \in \sS \colon \left[(h(s)<h(t) < h(r))\, \& \, (s\prec r)\right] \quad \Rightarrow \quad (t\prec r); \\
					\Leftrightarrow \ 
					\forall s\in \sS \ \exists a_s\leq b_s\leq n \ \colon \   \sS_{\succeq s}=h^{-1}([a_s,b_s]).
				\end{array}
			\end{equation*}
			is called an {\it arborescent poset with consistent linearization}.
		\end{itemize}
	\end{definition}
	
	Assumption~\eqref{eq::poset::st} ensures that the Hasse diagram of $(S,\prec)$ forms a forest, where the set of minimal elements corresponds to the roots.
	Indeed, if this were not the case, there would exist three elements $s,t,r\in\mathsf{S}$ such that $s\succ t$, $s\succ r$, with $t$ and $r$ being incomparable, and $v(s)<v(t)<v(r)$. Consequently, the subset $\mathsf{S}_{\succeq r}$ would contain $s$ and, therefore, the preimage of the interval $[v(s),v(r)]$. However, $t\not\in \mathsf{S}_{\succeq r}$, leading to a contradiction.
	
	\begin{definition}
		\label{def::Sdominant}
		Given an arborescent poset $\sS$ with consistent linearization or anti-linearization $v:\sS\hookrightarrow [1m]$, a composition $\ov{d}:=(d_1,\ldots,d_m)\in\bZ_{\geq 0}^{m}$ is called \emph{$\sS$-dominant} if the following conditions hold:
		$$
		\begin{cases}
			l\notin v(\sS) \ \Rightarrow \ d_l = 0; \\
			s\succ t\in \sS \ \Rightarrow \ d_{v(s)} \geq d_{v(t)}.
		\end{cases}
		$$
	\end{definition}
	
	\begin{remark}
		\label{rk::opposite::linearization}	
		For an arborescent poset $\sS$, the order-reversing embedding
		$v:\sS\hookrightarrow [1m]$ is a consistent anti-linearization if and only if its opposite order-preserving embedding $\op\circ v:\sS\hookrightarrow [1m]$ is a consistent linearization.
		Here,  
		$$\op(l):=m-l+1$$
		is the standard order-reversing involution of the interval $[1m]$.
		
		Accordingly, a composition $\ov{d}:=(d_1,\ldots,d_m)\in\bZ_{\geq 0}^{m}$ is $\sS$-dominant for the anti-linearization $v$ if and only if the opposite composition $\ov{d}^{\op}:=(d_m,\ldots,d_1)$ is dominant for the linearization $\op\circ v$.
	\end{remark}
	
	From this point onward, we fix an arborescent poset $\sS$ with a consistent {\bf anti}-linearization $v:\sS\hookrightarrow [1m]$. Similar definitions and results hold for a consistent {\bf linearization} by considering the opposite order on the linearly ordered set $[1m]$.
	
	\begin{definition}
		\label{def::S::dominant}
		A composition $\ov{d}:=(d_1,\ldots,d_m)\in\bZ_{\geq 0}^{m}$ is called:
		\begin{itemize}
			\item {\it $\sS$-admissible} if the following inequality holds for all $k\leq m$: 
			$$\#\{j=1,\dots,k\ \colon \ d_j>0\} \leq \# \sS_{\leq k},$$
			where 
			$\sS_{\leq k}:= v^{-1}([1,k]).$
			\item 
			{\it $\sS[{l}]$-dominant} if the following implications hold:
			$$
			\begin{cases}
				l\notin v(\sS) \ \Rightarrow \ d_l = 0; \\
				l\in v(\sS) \ \Rightarrow \ \forall s \in\sS \ \colon s\succ v^{-1}(l) \text{ we have } d_{v(s)} \geq d_l;
			\end{cases}
			$$
			\item {\it $\sS[{\leq k}]$-dominant} 
			if $\ov{d}$ is $\sS[{l}]$-dominant for all $l\leq k$.
		\end{itemize}
	\end{definition}
	
	In particular, a composition $\ov{d}$ is $\sS$-dominant if and only if it is $\sS[\leq m]$-dominant.
	
	\begin{example}
		\label{ex::arbor}
		Below is a pictorial example of an arborescent poset with consistent anti-linearization. The elements of the poset $\sS$ are represented by blue dots, edges in the Hasse diagram of $\sS$ are drawn as edges, and each cell corresponds to an element of the linearly ordered set $[1m]$:
		$$
		\begin{tikzpicture}[scale=0.5, shift={(0,-1)}]
			\draw[step=1cm] (0,1) grid (9,0);
			\node (v1) at (.5,.4) {{ \color{blue}$\bullet$ }};
			\node (v2) at (1.5,.4) {{ \color{blue}$\bullet$ }};
			\node (v3) at (2.5,.4) {{ \color{blue}$\bullet$ }};
			\node (v5) at (4.5,.4) {{ \color{blue}$\bullet$ }};
			\node (v6) at (5.5,.4) {{ \color{blue}$\bullet$ }};
			\node (v7) at (6.5,.4) {{ \color{blue}$\bullet$ }};
			\node (v8) at (7.5,.4) {{ \color{blue}$\bullet$ }};
			\draw  (2.5,.6) edge[blue,line width=1.1pt,->,bend right=90] (.5,.6);
			\draw (2.5,.6) edge[blue,line width=1.1pt,->,bend right=60] (1.5,.6);
			\draw (7.5,.6) edge[blue,line width=1.1pt,->,bend right=90] (4.5,.6);
			\draw (7.5,.6) edge[blue,line width=1.1pt,->,bend right=60] (6.5,.6);
			\draw (6.5,.6) edge[blue,line width=1.1pt,->,bend right=60]  (5.5,.6);
		\end{tikzpicture} \ := \ 
		\left\{
		\begin{array}{c}
			\sS:=\{s_1,s_2,s_3, s_5,s_6,s_7,s_8\}, \\
			s_1\succ s_3, s_2\succ s_3, s_5\succ s_8, s_6\succ s_7\succ s_8, \\
			v:\sS\hookrightarrow[1 9], \text{ with }v(s_i)=i.
		\end{array}
		\right\}
		$$
		We hope that the following pictorial illustration of various compositions for the aforementioned arborescent poset $\sS$ will help in understanding Definition~\ref{def::S::dominant}:
		$$
		\begin{array}{ccl}
			\begin{tikzpicture}[xscale=0.5,yscale=0.4, shift={(0,0)}]
				\draw[step=1cm] (0,1) grid (9,0);
				\node (v1) at (.5,.4) {{ \color{blue}\small $3$ }};
				\node (v2) at (1.5,.4) {{ \color{blue}\small $4$  }};
				\node (v3) at (2.5,.4) {{ \color{blue} \small $1$  }};
				\node (v4) at (3.5,.4) {{ \color{blue} \small $0$  }};
				\node (v5) at (4.5,.4) {{ \color{blue} \small $4$  }};
				\node (v6) at (5.5,.4) {{ \color{blue} \small $6$  }};
				\node (v7) at (6.5,.4) {{ \color{blue} \small $3$  }};
				\node (v8) at (7.5,.4) {{ \color{blue} \small $2$  }};
				\node (v9) at (8.5,.4) {{ \color{blue} \small $0$  }};
				\draw  (2.5,1) edge[blue,line width=1.1pt,->,bend right=90] (.5,1);
				\draw (2.5,1) edge[blue,line width=1.1pt,->,bend right=70] (1.5,1);
				\draw (7.5,1) edge[blue,line width=1.1pt,->,bend right=90] (4.5,1);
				\draw (7.5,1) edge[blue,line width=1.1pt,->,bend right=70] (6.5,1);
				\draw (6.5,1) edge[blue,line width=1.1pt,->,bend right=70]  (5.5,1);
			\end{tikzpicture} & - & \text{ is $\sS$-dominant }\Leftrightarrow \ \text{ $\sS[\leq 9]$-dominant, }  \\
			\begin{tikzpicture}[xscale=0.5,yscale=0.4, shift={(0,-.5)}]
				\draw[step=1cm] (0,1) grid (9,0);
				\node (v1) at (.5,.4) {{ \color{blue}\small $3$ }};
				\node (v2) at (1.5,.4) {{ \color{blue}\small $4$  }};
				\node (v3) at (2.5,.4) {{ \color{blue} \small $1$  }};
				\node (v4) at (3.5,.4) {{ \color{blue} \small $0$  }};
				\node (v5) at (4.5,.4) {{ \color{blue} \small $4$  }};
				\node (v6) at (5.5,.4) {{ \color{blue} \small $2$  }};
				\node (v7) at (6.5,.4) {{ \color{red} \small $3$  }};
				\node (v8) at (7.5,.4) {{ \color{blue} \small $0$  }};
				\node (v9) at (8.5,.4) {{ \color{blue} \small $6$  }};
				\draw  (2.5,1) edge[blue,line width=1.1pt,->,bend right=90] (.5,1);
				\draw (2.5,1) edge[blue,line width=1.1pt,->,bend right=70] (1.5,1);
				\draw (7.5,1) edge[blue,line width=1.1pt,->,bend right=90] (4.5,1);
				\draw (7.5,1) edge[blue,line width=1.1pt,->,bend right=70] (6.5,1);
				\draw (6.5,1) edge[red,line width=1.1pt,->,bend right=70]  (5.5,1);
			\end{tikzpicture} 
			& - & \text{ is $\sS[\leq 6]$-dominant and $\sS$-admissible,} \\
			\begin{tikzpicture}[xscale=0.5,yscale=0.4, shift={(0,-.5)}]
				\draw[step=1cm] (0,1) grid (9,0);
				\node (v1) at (.5,.4) {{ \color{blue}\small $3$ }};
				\node (v2) at (1.5,.4) {{ \color{blue}\small $4$  }};
				\node (v3) at (2.5,.4) {{ \color{blue} \small $1$  }};
				\node (v4) at (3.5,.4) {{ \color{red} \small $8$  }};
				\node (v5) at (4.5,.4) {{ \color{blue} \small $4$  }};
				\node (v6) at (5.5,.4) {{ \color{blue} \small $6$  }};
				\node (v7) at (6.5,.4) {{ \color{blue} \small $3$  }};
				\node (v8) at (7.5,.4) {{ \color{blue} \small $2$  }};
				\node (v9) at (8.5,.4) {{ \color{blue} \small $0$  }};
				\draw  (2.5,1) edge[blue,line width=1.1pt,->,bend right=90] (.5,1);
				\draw (2.5,1) edge[blue,line width=1.1pt,->,bend right=70] (1.5,1);
				\draw (7.5,1) edge[blue,line width=1.1pt,->,bend right=90] (4.5,1);
				\draw (7.5,1) edge[blue,line width=1.1pt,->,bend right=70] (6.5,1);
				\draw (6.5,1) edge[blue,line width=1.1pt,->,bend right=70]  (5.5,1);
			\end{tikzpicture} 
			& - & \text{ is $\sS[\leq 3]$-dominant, but not $\sS$-admissible.}
		\end{array}
		$$
	\end{example}
	
	\begin{notation}
		\label{not::ADsets}
		For each partition $\lambda$, we define the following subsets of $\bS_m\lambda$, consisting of compositions whose multiset of nonzero elements equals $\lambda$: 
		\begin{center}
			$\bA_{\sS}^{v}(\lambda)$ -- the set of $\sS$-admissible compositions; \\
			$\bD_{\sS}^{v}(\lambda)$ -- the set of $\sS$-dominant compositions; \\
			$\bAD_{\sS[\leq k]}^{v}(\lambda)$ -- the set of $\sS[\leq k]$-dominant and $\sS$-admissible compositions.
		\end{center}
		We equip these sets with the poset structure induced from the Bruhat partial order on $\bS_m\lambda$.
	\end{notation}
	
	Note that every $\sS$-dominant composition is also $\sS$-admissible, leading to the following sequence of poset embeddings:
	$$\bD_{\sS}^{v}(\lambda)= \bAD_{\sS[\leq m]}^{v}(\lambda)\subset \bAD_{\sS[\leq m-1]}^{v}(\lambda) \subset \ldots \subset \bAD_{\sS[\leq 0]}^{v}(\lambda) = \bA_{\sS}^{v}(\lambda).$$
	
	For each $k=1,\dots,m$, we denote by $\tilde{\sS}_k$ the poset $\sS_{\leq k}\sqcup\{k+1,\ldots,m\}$, where the partial order is inherited from $\sS_{\leq k}$, while the remaining elements are incomparable with all others. The poset $\tilde{\sS}_k$ is evidently arborescent, and the natural map\nlabel{ni:Stilde}
	$\tilde{v}:\tilde{\sS}_k\to [1m]$, which coincides with $v$ on $\sS_{\leq k}$ and acts as the identity on the complement, defines a consistent anti-linearization. From the definitions, we immediately obtain the following isomorphism of sets:
	\begin{equation}
		\label{eq::B=AD}
		\bAD_{\sS[\leq k]}^{v}(\lambda) = \bD_{\tilde{\sS}_k}^{\tilde{v}}(\lambda) \cap \bA_{\sS}^{v}(\lambda).
	\end{equation}
	
	The main goal of this section is to demonstrate that the posets $\bD_{\sS}^{v}(\lambda)$ inherit many desirable properties of the Bruhat graph $\bS_m\lambda$. One of the key features is the generalization of the classical bubble-sort algorithm, which defines a monotone idempotent map $\bbs_{\sS}:\bA_{\sS}^{v}(\lambda)\rightarrow \bD_{\sS}^{v}(\lambda)$. This map is constructed as a composition of monotone idempotents $\bbs_{\sS}^{k}:\bAD_{\sS[<k]}^{v}(\lambda) \to \bAD_{\sS[\leq k]}^{v}(\lambda)$.
	We have successfully proven that for a regular partition $\lambda$, the posets $\bD_{\sS}^{v}(\lambda)$ are bounded, graded, subthin, and $\EL$-shellable. In contrast, for $\lambda$ with repeated parts, the poset may not be graded; however, we conjecture that it remains shellable. Let us now discuss the differences between the regular and non-regular cases.
	
	The posets $\bD_{\sS}^{v}(\lambda)$ and $\bD_{\sS}^{v}(\lambda')$ are isomorphic whenever there is an isomorphism between the ordered multisets $\lambda$ and $\lambda'$. In particular, if $\lambda$ is {\it "regular"} (i.e., all parts of $\lambda$ are distinct), then the set $\bD_{\sS}^{v}(\lambda)$ of $\sS$-dominant compositions does not depend on $\lambda$ and will be denoted simply as $\bD_{\sS}^{v}$.\nlabel{ni:regular}
	
	\begin{lemma}
		Suppose that $\sS$ is an arborescent poset with a consistent anti-linearization $v:\sS\hookrightarrow[1m]$.
		Then the monotone projection $\pi_\lambda:\bS_{m} \twoheadrightarrow \bS_m\lambda$ and its adjoint embedding $\psi_{\lambda}^{+}:\bS_m\lambda \hookrightarrow \bS_m$ (defined in~\eqref{eq::proj::parabolic}) restrict to a monotone projection and its adjoint monotone embedding on the set $\bD_{\sS}^{v}$ of $\sS$-dominant weights:
		\begin{equation}
			\label{eq::proj::parab::D}
			\begin{tikzcd}
				\bD_{\sS}^{v} = \bD_{\sS}^{v}(\lambda+\delta) \arrow[rr,"\pi_\lambda"', two heads, shift right = 2] 
				\arrow[rr,"\psi_\lambda^{+}","\perp"', hookleftarrow, shift left = 2] 
				&&
				\bD_{\sS}^{v}(\lambda). 
			\end{tikzcd}
		\end{equation}
	\end{lemma}
	\begin{proof}
		It follows directly from Definition~\ref{def::S::dominant} that if $\ov{a}\in\bS_m(\lambda+\delta)$ is $\sS$-dominant, then the corresponding composition $\pi_{\lambda}(\ov{a})$ is also $\sS$-dominant. Conversely, if $\ov{b}\in\bS_{m}\lambda$ is $\sS$-dominant, then $\psi_{\lambda}^{+}(\ov{b})$ is also $\sS$-dominant. 
	\end{proof}
	Consequently, the composition $\psi_{\lambda}^+\circ\pi_\lambda$, also denoted by $\psi_{\lambda}^{+}$, is a non-increasing monotone idempotent on $\bD_{\sS}$.
	
	\begin{example}
		\label{ex::Double::coset}
		Suppose that $\sT=\sT_1\sqcup\ldots\sqcup\sT_k$ is a disjoint union of linearly ordered sets $\sT_q=(s_{m_{q-1}+1}\succ\ldots\succ s_{m_q})$ and that $v:\sT\hookrightarrow[1m]$ is a consistent surjective anti-linearization given by:
		$$ v(s_{i}) = i.$$
		Then, for each composition $\ov{d}\in\bZ_{\geq 0}^{m}$, we define a composition $\bbs_{\sT}(\ov{d})$ by sorting the elements within each connected component of $\sS$. More precisely, $\bbs_{\sT}(\ov{d})$ is determined by the following conditions, satisfied for all $q=1,\ldots,k$:
		\begin{itemize}[itemsep=0pt, topsep=0pt]
			\item The following multisets are equal:
			$$\{d_{m_{q-1}+1},\ldots,d_{m_{q}}\}\ = \ \{\bbs_{\sT}(\ov{d})_{m_{q-1}+1},\ldots,\bbs_{\sT}(\ov{d})_{m_{q}}\}.$$  
			\item 
			The values in $\bbs_{\sT}(\ov{d})$ are arranged in non-increasing order:
			\begin{equation}
				\label{eq::parabolic::ordered}
				\bbs_{\sT}(\ov{d})_{m_{q-1}+1}\geq \ldots \geq \bbs_{\sT}(\ov{d})_{m_{q}}.
			\end{equation}
		\end{itemize}
		By construction, we have
		$\bbs_{\sT}(\ov{d})=\bbs_{\sT}(\bbs_{\sT}(\ov{d})),$
		and for any $\ov{a}, \ov{b} \in \bS_m(\lambda)$ satisfying $\ov{a} \prec_{\Bruhat} \ov{b}$, we obtain:
		$
		\bbs_{\sT}(\ov{a}) \preceq \bbs_{\sT}(\ov{b}).
		$
		This confirms that $\bbs_{\sT}$ is a monotone idempotent in the sense of Definition~\ref{def::idempotent} (see, e.g., \cite[\S2.5]{BB}).
		
		Now, suppose that 
		$\lambda=(\lambda_1^{k_1}, \ldots, \lambda_l^{k_l})$ with $\lambda_1>\ldots>\lambda_l\geq 0$. Then 
		the corresponding set of $\sT$-dominant compositions has the following structure:
		$$
		\bD_{\sT}(\lambda)\ \simeq \ \bS_{\ov{m}}\backslash \bS_m\lambda 
		\ \simeq \
		\bS_{\ov{m}}\backslash \bS_m \slash \bS_{\ov{\lambda}}, $$
		where 
		$$
		\bS_{\ov{m}}:=\bS_{m_1}\times \ldots \times \bS_{m_r-m_{r-1}}, \quad \bS_{\ov{\lambda}}:=\bS_{k_1}\times \ldots \times \bS_{k_l} = \mathsf{Stabilizer}(\lambda).
		$$
		In other words, $\bD_{\sT}(\lambda)$ can be identified with a double coset of a Weyl group modulo two parabolic Weyl subgroups:
		$$
		\begin{tikzcd}
			\bS_{m}=\bS_{m}(\lambda+\delta) 
			\arrow[rr,"\pi_\lambda"', two heads, shift right = 1.7] 
			\arrow[rr,"\psi_\lambda^{+}","\perp"', hookleftarrow, shift left = 1.7] 
			\arrow[d,"\bbs_{\sT}"]
			&& 
			\bS_m\lambda = \bS_m/\bS_{\ov{\lambda}}
			\arrow[d,"\bbs_{\sT}"]
			\\
			\bD_{\sT}^{v}=\bS_{\ov{m}}\backslash \bS_{m} 
			\arrow[rr,"\pi_\lambda"', two heads, shift right = 1.7] 
			\arrow[rr,"\psi_\lambda^{+}","\perp"', hookleftarrow, shift left = 1.7] 
			&& 
			\bD_{\sT}^{v}(\lambda) =\bS_{\ov{m}}\backslash \bS_m/\bS_{\ov{\lambda}}
		\end{tikzcd}
		$$
		Note that, as a poset, the double quotient $\bS_{\ov{m}}\backslash \bS_m/\bS_{\ov{\lambda}}$ may be non-graded, as illustrated in Example~\ref{ex::nongrad}.
		However, all covering relations in the Hasse diagram of the induced Bruhat graph are given by appropriate transpositions.
		We refer to~\cite{Stembridge, Kobayashi::double::coset} for a discussion on double quotients in the context of Coxeter groups.
	\end{example}

	\subsection{Hasse diagram of $\sS$-dominant compositions}
	\label{sec::Hasse::dominant}
	
	In this section, we assume that $(\sS,v)$ is an arborescent poset with {\bf anti-linearization} and focus exclusively on $\sS$-dominant compositions. 
	From the definition, it follows that whenever $j\notin v(\sS)$, we have $d_j=0$ for any $\sS$-dominant composition $\ov{d}$. Thus, the $j$'th entry of any composition in $\bD_{\sS}^{v}(\lambda)$ does not influence any conclusions in this section, as it is always equal to $0$.
	Therefore, without loss of generality, we assume that $v:\sS\hookrightarrow [1m]$ is surjective, making it an isomorphism.
	
	\begin{definition}
		\label{def::disorder}	
		A pair $(ij)$ of indices  is called 
		a \emph{minimal $\sS$-disorder} for an $\sS$-dominant composition $\ov{d}:=(d_1,\ldots,d_m)$ if the following conditions hold:
		\begin{itemize}
			\item $v^{-1}(i)$ and $v^{-1}(j)$ are incomparable in $\sS$.
			\item $(ij)$ forms a disorder, meaning that $i<j$ and $d_{i}>d_{j}$.
			\item 
			For all $l$ in the interval $(i,j)\subset[1m]$, the following implications hold:
			\begin{equation}
				\label{eq::S::disorder:1a}
				\left\{
				\begin{array}{ccc}
					v^{-1}(l)\prec v^{-1}(i)  & \Rightarrow & d_{l}\leq d_j,\\
					v^{-1}(l) \succ v^{-1}(j) & \Rightarrow & d_{l} \geq d_i, \\
					v^{-1}(l) \text{ is uncomparable with } v^{-1}(i) \text{ and } v^{-1}(j) & \Rightarrow & d_l\notin [d_j,d_i].
				\end{array}
				\right.
			\end{equation}
		\end{itemize}
	\end{definition}
	
	\begin{lemma}
		\label{lem::Bruhat::edge}	
		If $(ij)$ is a minimal $\sS$-disorder for an $\sS$-dominant composition $\ov{d}$, then the composition:
		$$
		(ij)\cdot\ov{d}:= (\ldots, d_{i-1},{\bf d_j},d_{i+1},\ldots,d_{j-1},{\bf d_i},d_{j+1},\ldots )
		$$
		is also $\sS$-dominant, and $\ov{d}\prec_{\Bruhat} (ij)\cdot\ov{d}$.
	\end{lemma}
	\begin{proof}
		Since each element $d_l$ for $l=i+1,\ldots,j-1$ does not belong to the open interval $(d_j,d_i)$, it follows that $\ov{d}$ is less than $(ij)\cdot\ov{d}$ in the Bruhat partial order.
		The first two implications in~\eqref{eq::S::disorder:1a} ensure that the $\sS$-dominant inequalities remain valid for $(ij)\cdot\ov{d}$.
	\end{proof}
	
	It is worth mentioning that under the assumptions of Lemma~\ref{lem::Bruhat::edge}, if all elements of the composition $\ov{d}$ are pairwise distinct, then $\ov{d}\prec (ij)\cdot\ov{d}$ is a covering relation in the Bruhat graph $\bS_m\lambda$. Indeed, by the minimality condition, there is no $l \in \{i+1,\dots,j-1\}$ such that $d_i>d_l>d_j$.
	
	\begin{lemma}
		\label{lem::order::dominant}	
		Suppose that $\lambda$ is regular and $\ov{a},\ov{b}\in \bD_{\sS}^{v}(\lambda)$ is a pair of Bruhat-comparable $\sS$-dominant compositions ($\ov{a}\prec_{\Bruhat}\ov{b}$). Then there exists a minimal $\sS$-disorder $(ij)$ for $\ov{a}$ such that $\ov{a}\prec_{\Bruhat} (ij)\ov{a} \preceq_{\Bruhat} \ov{b}$.
	\end{lemma}
	\begin{proof}
		First, observe that if $\sS$ is a disjoint union of linearly ordered sets, then the statement follows from Example~\ref{ex::Double::coset}.
		
		We proceed by induction on the number $m$ of elements in $\sS$.
		Recall that the Hasse diagram of $\sS$ is a forest of rooted trees, where smaller elements are closer to the roots. 
		If $\sS$ consists of a single tree, then $v^{-1}(m)$ is the unique minimal element of $\sS$. Thus, $a_{m}=b_{m}=\lambda_m$, where $\lambda_m$ is the smallest part of $\lambda$.
		Then we can remove the root of $\sS$ and apply induction on $m$ with the truncated anti-linearization $\sS_{<m}\to [1m-1]$ and the corresponding $\sS_{<m}$-dominant compositions $(a_1,\ldots,a_{m-1})$ and $(b_1,\ldots,b_{m-1})$.
		
		If the Hasse diagram of $\sS$ contains multiple connected components, we decompose $\sS$ according to $v$:
		\begin{equation}
			\label{eq::tree:decomp}
			\sS = \sT_1\sqcup \sT_2 \sqcup \ldots \sqcup \sT_k, \quad |\sT_q|=m_q.
		\end{equation}
		For $q<q'$ and $s\in \sT_q$, $s'\in \sT_{q'}$, we have $v(s)<v(s')$, and we define a linear order on each tree $\sT_q$ following $v$:
		$$
		t\prec t' \ \stackrel{\mathsf{def}}{\Leftrightarrow} \ \begin{cases}
			\exists q\colon t,t'\in \sT_q, \\
			v(t) > v(t').
		\end{cases}
		$$
		The poset $\sT=\sqcup_{q=1}^{k}\sT_k$ is a disjoint union of linearly ordered sets $\sT_q$ and thus satisfies Example~\ref{ex::Double::coset}. 
		Therefore, we have a monotone idempotent $\bbs_{\sT}$ that maps $\ov{d}$ to a $\sT$-dominant composition $\bbs_{\sT}(\ov{d})$, as described in~\eqref{eq::parabolic::ordered}.
		
		If $\bbs_{\sT}(\ov{a})=\bbs_{\sT}(\ov{b})$, then the multisets of values of $\ov{a}$ and $\ov{b}$ restricted to each $\sT_q$ coincide, allowing us to proceed separately on each $\sT_q$.
		
		Otherwise, there exists a transposition $(ll')$ in the parabolic Bruhat graph such that:
		$$\bbs_{\sT}(\ov{a})\prec (ll')\bbs_{\sT}(\ov{a}) \preceq \bbs_{\sT}(\ov{b}).$$
		Since all elements of $\lambda$ are distinct, there exists a unique pair $(i,j)$ such that:
		$$a_{i} = (\bbs_{\sT}(\ov{a}))_{l} \ > \ a_{j} = (\bbs_{\sT}(\ov{a}))_{l}.$$
		Note that $(ll')$ is a minimal disorder for $\sT$-dominant composition $\bbs_{\sT}(\ov{a})$ what follows that
		$$
		\forall p= m_{q-1}+1,\ldots, m_{q'} \text{ we have } \bbs_{\sT}(a)_p\notin [(\bbs_{\sT}(\ov{a}))_{l'},(\bbs_{\sT}(\ov{a}))_{l}]
		\ \Leftrightarrow\ a_p\notin [a_j,a_i].
		$$
		Consequently $(ij)$ is a minimal $\sS$-disorder for $\ov{a}$ and, moreover, $(ij)\ov{a}\preceq\ov{b}$.
	\end{proof}
	
	\begin{theorem}	
		\label{thm::dominant::edge}
		The restriction of the Bruhat partial order to the subset $\bD_{\sS}^{v}(\lambda) \subset \bS_{m}\lambda$ of $\sS$-dominant compositions is generated by the relations $\ov{a}\prec (ij) \ov{a}$, where $(ij)$ is a minimal $\sS$-disorder for $\ov{a}$.
	\end{theorem}
	\begin{proof}
		First, suppose that $\lambda$ is regular. 
		Lemma~\ref{lem::Bruhat::edge} implies that $\ov{a}$ is less than $(ij)\ov{a}$ in the Bruhat partial order. Lemma~\ref{lem::order::dominant} establishes that the comparison $\ov{a}\prec (ij)\ov{a}$ defines an edge in the Hasse diagram.
		Indeed, assume that 
		$\ov{a}\prec_{\Bruhat}\ov{b}\in\bD_{\sS}^{v}(\lambda)$. Then there exists a chain of $\sS$-dominant compositions:
		$$\ov{a}=\ov{c_0}\prec \ov{c_1} \prec \ov{c_2} \prec \ldots \preceq \ov{b}$$
		such that each $\ov{c_i}$ differs from $\ov{c_{i+1}}$ by a transposition associated with a minimal $\sS$-disorder and that $\ov{c_k}\preceq \ov{b}$. Since the Bruhat graph is finite, this chain cannot be infinite. Consequently, there exists an index $l$ such that $\ov{c_l}=\ov{b}$.
		
		If $\lambda$ has repeating elements, we use the monotone projection $\pi_\lambda$ and its adjoint $\psi_{\lambda}^+$ from~\eqref{eq::proj::parab::D}. Specifically, consider an edge $\ov{a}\prec \ov{b}$ in the Hasse diagram of $\bD_{\sS}^{v}(\lambda)$. Then $\psi_{\lambda}^+(\ov{a})\prec\psi_{\lambda}^+(\ov{b})$ in $\bD_{\sS}$, and the latter must be connected by a chain of edges in the Hasse diagram of $\bD_{\sS}$:
		$$
		\psi_{\lambda}^+(\ov{a})=\ov{c_0}\prec\ldots\prec\ov{c_{k-1}} \prec\ov{c_{k}}\prec\ldots\prec \ov{c_l}=\psi_{\lambda}^+(\ov{b}),
		$$
		such that for all $i\neq k$, we have $\pi_{\lambda}(\ov{c_{i-1}})=\pi_{\lambda}(\ov{c_i})$.
		Consequently, $\ov{a}$ differs from $\ov{b}$ by a transposition in a disorder, and the corresponding disorder must be minimal since it arises from a minimal disorder in $\bD_{\sS}$.
	\end{proof}

	\subsection{Bubble-sort}
	\label{sec::Bubble-sort}
	
	In this section, we continue working with an arborescent poset $\sS$ equipped with a consistent order-reversing anti-linearization $v:\sS\to[1m]$.
	The {\bf{bubble-sort algorithm}} described below is an inductive procedure for reordering the elements of an $\sS$-admissible composition $\ov{d}$ into an $\sS$-dominant composition $\bbs_{\sS}(\ov{d})$:
	\begin{equation}
		\label{eq::bbs::comp}
		\bbs_{\sS}:= \bbs_{\sS}^{m}\circ \ldots\circ \bbs_{\sS}^1.
	\end{equation}
	In each step $k$ (denoted by $\bbs_{\sS}^{k}$), only the first $k$ elements of the composition are rearranged. Thus, we explain the procedure for compositions of length $k$.
	The $k$'th step $\bbs_{\sS}^k$ follows this recursive procedure to ensure $\sS[{k}]$-dominance:
	
	\begin{algorithm}
		\label{alg::bb:sort}
		Suppose that the composition $\ov{d}=(d_1,\ldots,d_k)$ is $\sS[<k]$-dominant and $\sS$-admissible but not $\sS[k]$-dominant.
		Then:
		\begin{itemize}
			\item
			Find the index $j$ as the maximal element of $[1m]$ with respect to the linear order $\leq$, among the set of maximal elements of $\sS$ with respect to the partial order $\prec$, whose weight is less than $d_k$. That is,
			\begin{equation}
				\label{eq::bubl::index}
				j:= \max_{\leq}\left\{v\left(\max_{\preceq}\left\{s\in \sS_{< k} \ \colon  \ d_{v(s)} < d_k\right\}\right)\right\}.
			\end{equation}
			\item Repeat this procedure for the composition obtained by swapping $d_j$ and $d_k$ in $\ov{d}$:
			$$(jk)\cdot \ov{d}:=(d_1,\ldots, d_{j-1},{\bf d_k},d_{j+1},\ldots,d_{k-1},{\bf d_j})$$ 
			until the composition becomes $\sS[k]$-dominant.
		\end{itemize}
	\end{algorithm}

	\begin{example}
		Let us show an example of bubble-sort process $\bbs_{\sS}$ applied to an $\sS$-admissible composition 
		$\begin{tikzpicture}[xscale=0.5,yscale=0.4, shift={(0,-.5)}]
			\draw[step=1cm] (0,1) grid (9,0);
			\node (v1) at (.5,.4) {{ \color{blue}\small $1$ }};
			\node (v2) at (1.5,.4) {{ \color{blue}\small $0$  }};
			\node (v3) at (2.5,.4) {{ \color{blue} \small $3$  }};
			\node (v4) at (3.5,.4) {{ \color{blue} \small $0$  }};
			\node (v5) at (4.5,.4) {{ \color{blue} \small $2$  }};
			\node (v6) at (5.5,.4) {{ \color{blue} \small $6$  }};
			\node (v7) at (6.5,.4) {{ \color{blue} \small $2$  }};
			\node (v8) at (7.5,.4) {{ \color{blue} \small $4$  }};
			\node (v9) at (8.5,.4) {{ \color{blue} \small $5$  }};
			\draw  (2.5,1) edge[blue,line width=1.1pt,->,bend right=90] (.5,1);
			\draw (2.5,1) edge[blue,line width=1.1pt,->,bend right=70] (1.5,1);
			\draw (7.5,1) edge[blue,line width=1.1pt,->,bend right=90] (4.5,1);
			\draw (7.5,1) edge[blue,line width=1.1pt,->,bend right=70] (6.5,1);
			\draw (6.5,1) edge[blue,line width=1.1pt,->,bend right=70]  (5.5,1);
		\end{tikzpicture}$ for the arborescent poset $\sS$ with antilinearisation from Example~\ref{ex::arbor}.
		As suggested in~\ref{alg::bb:sort} we apply $\bbs_{\sS}^{k}$ on the truncated compositions. We draw in "red" the numbers $d_s$ such that $s$ is one of the maximums in $\sS_{<k}$ with $d_{v(s)}<d_k$ and we round the maximum of them with respect to the standard linear order.
		\begin{gather*}
			\bbs_{\sS}^{1}=\bbs_{\sS}^{2}=\Id; \quad \bbs_{\sS}^{3}: \
			\begin{tikzpicture}[xscale=0.5,yscale=0.4, shift={(0,-.5)}]
				\draw[step=1cm] (0,1) grid (9,0);
				\node (v1) at (.5,.4) {{ \color{red}\small $1$ }};
				\node[ext] (v2) at (1.5,.4) {{ \color{red} \small $0$  }};
				\node (v3) at (2.5,.4) {{ \color{blue} \small $3$  }};
				\draw  (2.5,1) edge[blue,line width=1.1pt,->,bend right=90] (.5,1);
				\draw (2.5,1) edge[blue,line width=1.1pt,->,bend right=70] (1.5,1);
			\end{tikzpicture} 
			\stackrel{(23)}{\rightarrow}
			\begin{tikzpicture}[xscale=0.5,yscale=0.4, shift={(0,-.5)}]
				\draw[step=1cm] (0,1) grid (9,0);
				\node (v1) at (.5,.4) {{ \color{blue}\small $1$ }};
				\node (v2) at (1.5,.4) {{ \color{blue} \small $3$  }};
				\node (v3) at (2.5,.4) {{ \color{blue} \small $0$  }};
				\draw  (2.5,1) edge[blue,line width=1.1pt,->,bend right=90] (.5,1);
				\draw (2.5,1) edge[blue,line width=1.1pt,->,bend right=70] (1.5,1);
			\end{tikzpicture}, \\
			\bbs_{\sS}^{4}=\bbs_{\sS}^{5}=\bbs_{\sS}^{6}=\bbs_{\sS}^{7}=\Id; \\
			\bbs_{\sS}^{8}:
			\begin{tikzpicture}[xscale=0.5,yscale=0.4, shift={(0,-.5)}]
				\draw[step=1cm] (0,1) grid (9,0);
				\node (v1) at (.5,.4) {{ \color{red}\small $1$ }};
				\node (v2) at (1.5,.4) {{ \color{red}\small $3$  }};
				\node (v3) at (2.5,.4) {{ \color{blue} \small $0$  }};
				\node (v4) at (3.5,.4) {{ \color{blue} \small $0$  }};
				\node (v5) at (4.5,.4) {{ \color{red} \small $2$  }};
				\node (v6) at (5.5,.4) {{ \color{blue} \small $6$  }};
				\node[ext] (v7) at (6.5,.4) {{ \color{red} \small $2$  }};
				\node (v8) at (7.5,.4) {{ \color{blue} \small $4$  }};
				\draw  (2.5,1) edge[blue,line width=1.1pt,->,bend right=90] (.5,1);
				\draw (2.5,1) edge[blue,line width=1.1pt,->,bend right=70] (1.5,1);
				\draw (7.5,1) edge[blue,line width=1.1pt,->,bend right=90] (4.5,1);
				\draw (7.5,1) edge[blue,line width=1.1pt,->,bend right=70] (6.5,1);
				\draw (6.5,1) edge[blue,line width=1.1pt,->,bend right=70]  (5.5,1);
			\end{tikzpicture}
			\stackrel{(78)}{\rightarrow}
			\begin{tikzpicture}[xscale=0.5,yscale=0.4, shift={(0,-.5)}]
				\draw[step=1cm] (0,1) grid (9,0);
				\node (v1) at (.5,.4) {{ \color{blue}\small $1$ }};
				\node (v2) at (1.5,.4) {{ \color{blue}\small $3$  }};
				\node (v3) at (2.5,.4) {{ \color{blue} \small $0$  }};
				\node (v4) at (3.5,.4) {{ \color{blue} \small $0$  }};
				\node (v5) at (4.5,.4) {{ \color{blue} \small $2$  }};
				\node (v6) at (5.5,.4) {{ \color{blue} \small $6$  }};
				\node (v7) at (6.5,.4) {{ \color{blue} \small $4$  }};
				\node (v8) at (7.5,.4) {{ \color{blue} \small $2$  }};
				\draw  (2.5,1) edge[blue,line width=1.1pt,->,bend right=90] (.5,1);
				\draw (2.5,1) edge[blue,line width=1.1pt,->,bend right=70] (1.5,1);
				\draw (7.5,1) edge[blue,line width=1.1pt,->,bend right=90] (4.5,1);
				\draw (7.5,1) edge[blue,line width=1.1pt,->,bend right=70] (6.5,1);
				\draw (6.5,1) edge[blue,line width=1.1pt,->,bend right=70]  (5.5,1);
			\end{tikzpicture};
			\\
			\bbs_{\sS}^{9}:
			\left\{
			\begin{array}{c}
				\begin{tikzpicture}[xscale=0.5,yscale=0.4, shift={(0,-.5)}]
					\draw[step=1cm] (0,1) grid (9,0);
					\node (v1) at (.5,.4) {{ \color{red}\small $1$ }};
					\node (v2) at (1.5,.4) {{ \color{red}\small $3$  }};
					\node (v3) at (2.5,.4) {{ \color{blue} \small $0$  }};
					\node (v4) at (3.5,.4) {{ \color{blue} \small $0$  }};
					\node (v5) at (4.5,.4) {{ \color{red} \small $2$  }};
					\node (v6) at (5.5,.4) {{ \color{blue} \small $6$  }};
					\node[ext] (v7) at (6.5,.4) {{ \color{red} \small $4$  }};
					\node (v8) at (7.5,.4) {{ \color{blue} \small $2$  }};
					\node (v9) at (8.5,.4) {{ \color{blue} \small $5$  }};
					\draw  (2.5,1) edge[blue,line width=1.1pt,->,bend right=90] (.5,1);
					\draw (2.5,1) edge[blue,line width=1.1pt,->,bend right=70] (1.5,1);
					\draw (7.5,1) edge[blue,line width=1.1pt,->,bend right=90] (4.5,1);
					\draw (7.5,1) edge[blue,line width=1.1pt,->,bend right=70] (6.5,1);
					\draw (6.5,1) edge[blue,line width=1.1pt,->,bend right=70]  (5.5,1);
				\end{tikzpicture}
				\stackrel{(79)}{\rightarrow}
				\begin{tikzpicture}[xscale=0.5,yscale=0.4, shift={(0,-.5)}]
					\draw[step=1cm] (0,1) grid (9,0);
					\node (v1) at (.5,.4) {{ \color{red}\small $1$ }};
					\node (v2) at (1.5,.4) {{ \color{red}\small $3$  }};
					\node (v3) at (2.5,.4) {{ \color{blue} \small $0$  }};
					\node (v4) at (3.5,.4) {{ \color{blue} \small $0$  }};
					\node[ext] (v5) at (4.5,.4) {{ \color{red} \small $2$  }};
					\node (v6) at (5.5,.4) {{ \color{blue} \small $6$  }};
					\node (v7) at (6.5,.4) {{ \color{blue} \small $5$  }};
					\node (v8) at (7.5,.4) {{ \color{blue} \small $2$  }};
					\node (v9) at (8.5,.4) {{ \color{blue} \small $4$  }};
					\draw  (2.5,1) edge[blue,line width=1.1pt,->,bend right=90] (.5,1);
					\draw (2.5,1) edge[blue,line width=1.1pt,->,bend right=70] (1.5,1);
					\draw (7.5,1) edge[blue,line width=1.1pt,->,bend right=90] (4.5,1);
					\draw (7.5,1) edge[blue,line width=1.1pt,->,bend right=70] (6.5,1);
					\draw (6.5,1) edge[blue,line width=1.1pt,->,bend right=70]  (5.5,1);
				\end{tikzpicture}
				\stackrel{(69)}{\rightarrow}\\
				\stackrel{(69)}{\rightarrow}
				\begin{tikzpicture}[xscale=0.5,yscale=0.4, shift={(0,-.5)}]
					\draw[step=1cm] (0,1) grid (9,0);
					\node[ext] (v1) at (.5,.4) {{ \color{red}\small $1$ }};
					\node (v2) at (1.5,.4) {{ \color{blue}\small $3$  }};
					\node (v3) at (2.5,.4) {{ \color{blue} \small $0$  }};
					\node (v4) at (3.5,.4) {{ \color{blue} \small $0$  }};
					\node (v5) at (4.5,.4) {{ \color{blue} \small $4$  }};
					\node (v6) at (5.5,.4) {{ \color{blue} \small $6$  }};
					\node (v7) at (6.5,.4) {{ \color{blue} \small $5$  }};
					\node (v8) at (7.5,.4) {{ \color{blue} \small $2$  }};
					\node (v9) at (8.5,.4) {{ \color{blue} \small $2$  }};
					\draw  (2.5,1) edge[blue,line width=1.1pt,->,bend right=90] (.5,1);
					\draw (2.5,1) edge[blue,line width=1.1pt,->,bend right=70] (1.5,1);
					\draw (7.5,1) edge[blue,line width=1.1pt,->,bend right=90] (4.5,1);
					\draw (7.5,1) edge[blue,line width=1.1pt,->,bend right=70] (6.5,1);
					\draw (6.5,1) edge[blue,line width=1.1pt,->,bend right=70]  (5.5,1);
				\end{tikzpicture} 
				\stackrel{(19)}{\rightarrow}
				\begin{tikzpicture}[xscale=0.5,yscale=0.4, shift={(0,-.5)}]
					\draw[step=1cm] (0,1) grid (9,0);
					\node (v1) at (.5,.4) {{ \color{blue}\small $2$ }};
					\node (v2) at (1.5,.4) {{ \color{blue}\small $3$  }};
					\node[ext] (v3) at (2.5,.4) {{ \color{red} \small $0$  }};
					\node (v4) at (3.5,.4) {{ \color{blue} \small $0$  }};
					\node (v5) at (4.5,.4) {{ \color{blue} \small $4$  }};
					\node (v6) at (5.5,.4) {{ \color{blue} \small $6$  }};
					\node (v7) at (6.5,.4) {{ \color{blue} \small $5$  }};
					\node (v8) at (7.5,.4) {{ \color{blue} \small $2$  }};
					\node (v9) at (8.5,.4) {{ \color{blue} \small $1$  }};
					\draw  (2.5,1) edge[blue,line width=1.1pt,->,bend right=90] (.5,1);
					\draw (2.5,1) edge[blue,line width=1.1pt,->,bend right=70] (1.5,1);
					\draw (7.5,1) edge[blue,line width=1.1pt,->,bend right=90] (4.5,1);
					\draw (7.5,1) edge[blue,line width=1.1pt,->,bend right=70] (6.5,1);
					\draw (6.5,1) edge[blue,line width=1.1pt,->,bend right=70]  (5.5,1);
				\end{tikzpicture} 
				\stackrel{(39)}{\rightarrow} \\
				\stackrel{(39)}{\rightarrow} 
				\begin{tikzpicture}[xscale=0.5,yscale=0.4, shift={(0,-.5)}]
					\draw[step=1cm] (0,1) grid (9,0);
					\node (v1) at (.5,.4) {{ \color{blue}\small $2$ }};
					\node (v2) at (1.5,.4) {{ \color{blue}\small $3$  }};
					\node (v3) at (2.5,.4) {{ \color{blue} \small $1$  }};
					\node (v4) at (3.5,.4) {{ \color{blue} \small $0$  }};
					\node (v5) at (4.5,.4) {{ \color{blue} \small $4$  }};
					\node (v6) at (5.5,.4) {{ \color{blue} \small $6$  }};
					\node (v7) at (6.5,.4) {{ \color{blue} \small $5$  }};
					\node (v8) at (7.5,.4) {{ \color{blue} \small $2$  }};
					\node (v9) at (8.5,.4) {{ \color{blue} \small $0$  }};
					\draw  (2.5,1) edge[blue,line width=1.1pt,->,bend right=90] (.5,1);
					\draw (2.5,1) edge[blue,line width=1.1pt,->,bend right=70] (1.5,1);
					\draw (7.5,1) edge[blue,line width=1.1pt,->,bend right=90] (4.5,1);
					\draw (7.5,1) edge[blue,line width=1.1pt,->,bend right=70] (6.5,1);
					\draw (6.5,1) edge[blue,line width=1.1pt,->,bend right=70]  (5.5,1);
				\end{tikzpicture} 
			\end{array}
			\right.
		\end{gather*}
	\end{example}

	\begin{lemma}
		The \emph{bubble-sort} operation $\bbs_{\sS}^k$ is well-defined, and the index $j$
		obtained by assumption~\eqref{eq::bubl::index} is chosen such that $(jk)$ is the minimal $\sS_{<k}$-disorder for the composition $\ov{d}$.
	\end{lemma}
	\begin{proof}
		Let us analyze the details of Algorithm~\ref{alg::bb:sort}.
		Suppose that $\ov{d}$ is not $\sS[k]$-dominant.
		If $k\notin v(\sS)$, then $\sS_{<k} = \sS_{\leq k}$. Since $\ov{d}$ is $\sS$-admissible, there exists $s\in\sS_{<k}$ such that $d_{v(s)}=0$. 
		If $k=v(t)$ for some $t\in \sS$, then since $\ov{d}$ is not $\sS_{\leq k}$-dominant, there exists an element $s\in\sS_{<k}$ such that $s\succ t$ and $d_{v(s)}<d_{v(t)}=d_k.$
		In both cases, we find an element $s\in\sS_{<k}$ such that $d_{v(s)}<d_k$.
		Consequently, the subset 
		\[
		\left\{s\in \sS_{< k} \ \colon  \ d_{v(s)} < d_k\right\}\subset \sS
		\]
		is nonempty, ensuring that the index $j$ defined by~\eqref{eq::bubl::index} is well-defined. Moreover, for all $l= j+1,\ldots, k-1$, we have either $l\notin v(\sS)$, in which case $d_l=0$, or $l\in v(\sS)$ and $d_l\leq d_j<d_k$. Therefore, by the description~\eqref{eq::Bruhat::edge}, we conclude that the pair $(jk)$ is a minimal $\sS_{<k}$-disorder for $\ov{d}$.
	\end{proof}
	
	\begin{corollary}
		\label{cor::bbs::idem}
		The assignment $\ov{d}\mapsto \bbs_{\sS}^{k}$ is a nonincreasing idempotent on the set $\bAD_{\sS[<k]}^{v}(\lambda)$ of $\sS_{<k}$-dominant $\sS$-admissible compositions, whose image is the set $\bAD_{\sS[\leq k]}^{v}(\lambda)$ of $\sS[k]$-dominant compositions:
		\begin{align}
			\label{eq::bbs::idemp}
			\forall \ov{d}\in \bAD_{\sS[<k]}^{v}(\lambda) 	&\qquad
			\ov{d}\succeq_{\Bruhat} \bbs_{\sS}^k(\ov{d}) = \bbs_{\sS}^k(\bbs_{\sS}^k(\ov{d}))\in\bAD_{\sS[\leq k]}^{v}(\lambda).  
		\end{align}
	\end{corollary}
	\begin{proof}
		The map $\bbs_{\sS}^k$ is nonincreasing~\eqref{eq::bbs::idemp} because at each step we do the transposition in a minimal disorder and it is idempotent that fixes the set of $\sS[k]$-dominant compositions.    
	\end{proof}

	\subsection{The Desired Properties of the Posets $\bD_{\sS}(\lambda)$}
	\label{sec::S::Dom}
	\subsubsection{{\bf Case of Regular $\lambda$}}
	\label{sec::S::Regular}
	
	Let us analyze in detail the properties of $\bD_{\sS}^{v}(\lambda)$ for regular $\lambda$ (i.e., when all elements of the partition are distinct). In particular, we will sometimes omit the notation $\lambda$ since, for regular $\lambda$, the structure of $\bD_{\sS}^{v}(\lambda)$ does not depend on it.
	We may also assume, without loss of generality, that $v$ is surjective and thus an isomorphism. Consequently, we have $\bA_{\sS}(\lambda) = \bS_m\lambda \simeq \bS_m$, since every composition $\ov{d} \in \bS_m$ is $\sS$-admissible. Moreover, the minimal element $\lambda_{+}$ is always $\sS$-dominant.
	
	\begin{proposition}
		\label{prp::Hasse::Ds}
		The Hasse diagram of $\bD_{\sS}^{v}$ is a subgraph of the Bruhat graph $\bS_m$. In particular, $\bD_{\sS}^{v}$ is graded, and the rank function in $\bD_{\sS}^{v}$ is given by:
		$$\rk_{\bD_{\sS}^{v}}(\ov{d}):=\#\{s, t\in \sS \ \colon \ (s\not\succ t) \ \& \ (v(s)<v(t)) \ \& \ (d_{v(s)} < d_{v(t)})\}.$$
	\end{proposition}
	\begin{proof}
		If all $\lambda_i$'s are distinct, then the non-strict inequalities in~\eqref{eq::S::disorder:1a} become strict. Consequently, a minimal $\sS$-disorder $(ij)$ for $\ov{a}$ defines an edge $\ov{a} \to (ij)\ov{a}$ in the Bruhat graph $\bS_m$. Thus, the rank function counts the number of (dis)orders among incomparable elements in $\sS$.    
	\end{proof}
	
	\begin{corollary}
		\label{cor::dominant::subthin}	
		The poset $\bD_{\sS}^{v}$ is subthin.
	\end{corollary}
	\begin{proof}
		Consider two compositions $\ov{a},\ov{b}\in\bD_{\sS}^{v}(\lambda)$ such that the rank difference between $\ov{a}$ and $\ov{b}$ is exactly $2$.
		Since the Hasse diagram of $\bD_{\sS}^{v}$ is a subgraph of $\bS_m$, we know that the open interval $(\ov{a},\ov{b})^{\bS_m}$ in $\bS_m$ consists of exactly two elements (since $\bS_m$ is thin). By assumption, at least one of these elements must belong to $\bD_{\sS}^{v}$, which implies that $\bD_{\sS}^{v}$ is subthin.
	\end{proof}

	\begin{proposition}
		\label{lem::bbs::idemp}	
		The bubble-sort operator $\bbs_{\sS}^k: \bAD_{\sS[<k]}^{v} \to \bAD_{\sS[\leq k]}^{v}$ is non-increasing, monotone, and idempotent (in the sense of Definition~\ref{def::idempotent}).
	\end{proposition}
	
	\begin{proof}
		By Corollary~\ref{cor::bbs::idem}, it has already been established that $\bbs_{\sS}^k$ is a non-increasing idempotent. We now prove the monotonicity property by induction on the Bruhat partial order.
		According to Theorem~\ref{thm::dominant::edge}, the edges in the Hasse diagram of $\sS_{<k}$-dominant compositions correspond to minimal disorders of the poset $\tilde{\sS}_{k-1}$, as discussed in~\eqref{eq::B=AD}. Let $(il)\ov{a} \succeq_{\Bruhat} \ov{a}$ be an edge in the Bruhat subgraph of $\sS$-admissible, $\sS_{<k}$-dominant compositions. To establish monotonicity, it suffices to show that $\bbs_{\sS}^{k}((il)\ov{a}) \succeq_{\Bruhat} \bbs_{\sS}^{k}(\ov{a})$ in the Hasse diagram of $\sS_{\leq k}$-dominant compositions.

		Recall that the bubble-sort operator $\bbs_{\sS}^{k}(\ov{a})$ acts on $\ov{a}$ from the left by a sequence of transpositions $(j_1 k) \dots (j_l k)$ with $j_{\ldot} < k$. We analyze the difference between $\bbs_{\sS}^k(\ov{a})$ and $\bbs_{\sS}^k((il)\ov{a})$ by considering each transposition $(j k)$ in the expansion of $\bbs_{\sS}^k$ consecutively.
		
		Without loss of generality, suppose that for a given composition, both $(il)$ and $(jk)$ are minimal $\sS_{<k}$-disorders. The case where these transpositions are equal is trivial; thus, we assume that $(il)\ov{a}$ and $(jk)\ov{a}$ form a pair of incomparable $\sS_{<k}$-dominant compositions. We restrict the action of $\bbs_{\sS}^{k}$ to the subset of indices $\mathsf{J} := \{i, l, j, k\}$.
		In this restriction, the partial order $(\sS_{<k}, \prec)$ induced on $\mathsf{J}$ forms an arborescent poset of size $3$ or $4$. All arborescent posets of size at most $4$ and their corresponding posets of dominant weights are classified in Appendix~\ref{sec::Example::DL}. (Notably, there are only two arborescent posets whose dominant weights do not coincide with parabolic Bruhat graphs.) In each of these cases, the monotonicity of the bubble-sort operation can be verified directly, showing that $\bbs_{\sS}^k((il)\ov{a})$ is either equal to $\bbs_{\sS}^k(\ov{a})$ or is separated by one or two edges in the Bruhat order. Since the local order for the remaining elements remains unchanged, the disorders for the restriction of $\ov{a}$ to the subset of indices $\mathsf{J}$ remain disorders for the full composition $\ov{a}$.
	\end{proof}

	\begin{corollary}
		\label{cor::bbs::monotone}
		The bubble-sort process $\bbs_{\sS}:\bS_m\to\bD_{\sS}^{v}$ is a nonincreasing, monotone idempotent on $\bS_m$.
	\end{corollary}
	\begin{proof}
		By Definition~\eqref{eq::bbs::comp}, $\bbs_{\sS}$ is constructed as a sequential composition of nonincreasing, monotone idempotents $\bbs_{\sS}^{k}$ for $k=1,\ldots,m$, where each $\bbs_{\sS}^{k}$ is applied to the image of the previous step $\bbs_{\sS}^{k-1}$.
	\end{proof}
	
	\begin{corollary}
		\label{cor::DS::bound}
		The poset $\bD_{\sS}^{v}$ is bounded.
	\end{corollary}
	\begin{proof}
		Since the Bruhat graph $\bS_m$ is bounded, its image under a monotone idempotent must also be bounded.
	\end{proof}

	Suppose we are given an $\EL$-labelling $\euE_{\bS_m}$ of the Bruhat graph $\bS_m$, as described in~\eqref{eq::order::transpos}.  
	Then, for each pair of comparable elements $\ov{a}\prec_{\Bruhat}\ov{b}$, we assign the path:
	\[
	\pth_{\bS_m}(\ov{a}\to\ov{b}):= \left(\ov{a} \prec \ov{y_1} \prec  \ldots \prec \ov{y_{k-1}} \prec \ov{b}\right),
	\]
	such that 
	\[
	\euE_{\bS_m}(\ov{a} \prec \ov{y_1}) < \euE_{\bS_m}(\ov{y_1} \prec \ov{y_2})< \ldots < \euE_{\bS_m}(\ov{y_{k-1}} \prec \ov{b}).
	\]
	This path is assumed to be unique and lexicographically minimal due to the $\EL$-property.
	
	\begin{lemma}
		\label{lm::EL::S}	
		Suppose that  
		$\ov{a}\prec\ov{b}\prec\ov{c}$ is a sequence of three adjacent elements in the Hasse diagram of $\bD_{\sS}^{v}(\lambda)$.  
		Then the maximal chain $\pth_{\bS_m}(\ov{a}\to\ov{c})$ in the Bruhat graph $\bS_m\lambda$ with increasing $\euE$-labellings contains an $\sS$-dominant composition $\ov{d}$ such that $\ov{a}\prec \ov{d} \prec \ov{c}$. Moreover, if $\ov{d}\neq \ov{b}$, then the induced $\euE_{\bS_m}$-labeling along the alternative path is decreasing:
		\[
		\euE_{\bS_m}(\ov{a}\prec\ov{b}) > \euE_{\bS_m}(\ov{b}\prec\ov{c}).
		\]
	\end{lemma}
	\begin{proof}
		Since the edges in the Hasse diagram of $\bD_{\sS}^{v}$ correspond to transpositions, we can restrict the problem to the cases where $\#\sS\leq 4$ and analyze the combinatorics of the $\euE_{\bS_m}$-labeling in $\bS_3$ and $\bS_4$.  
		Note that assumption~\eqref{eq::poset::st} plays a crucial role in this argument.
	\end{proof}
	
	\begin{proposition}
		\label{prp::EL::S::domninant}
		The induced $\euE_{\bS_m}$-labeling of the Hasse diagram of $\bD_{\sS}^{v}$ is an $\EL$-labeling.
	\end{proposition}
	\begin{proof}
		Let $\ov{a}\prec\ov{b}$ be a pair of comparable $\sS$-dominant compositions. Suppose that the length of the maximal chain between $\ov{a}$ and $\ov{b}$ in $\bD_{\sS}$ is $k$.  
		By Proposition~\ref{prp::Hasse::Ds}, the distance between $\ov{a}$ and $\ov{b}$ in $\bS_{m}$ is also $k$.  
		We prove by induction on $k$ that the path $\pth_{\bS_m}(\ov{a}\to \ov{b})$ with the minimal $\euE_{\bS_m}$-labeling consists of $\sS$-dominant compositions.  
		
		The base case ($k=1$) is trivial.  
		For the induction step, assume that the claim holds for chains of length less than $k$.  
		By Lemma~\ref{lm::EL::S}, starting with any maximal chain from $\ov{a}$ to $\ov{b}$ that passes through $\sS$-dominant elements, if the labeling is not increasing, then there exists a subsequence of three adjacent elements in the chain where the labeling is decreasing.  
		Applying Lemma~\ref{lm::EL::S}, we can replace this subsequence with another maximal chain in $\bD_{\sS}$ from $\ov{a}$ to $\ov{b}$ with a lexicographically smaller labeling.  
		This completes the induction step.
	\end{proof}
	
	Let us summarize in one theorem the key properties obtained for the set of $\sS$-dominant compositions in the case of regular $\lambda$.
	
	\begin{theorem}
		\label{thm::poset::AD}
		\label{thm::S:dom::outline}
		The following properties hold for the poset $(\bD_{\sS}^{v},\prec_{\Bruhat})$ of $\sS$-dominant compositions:
		\begin{itemize}[itemsep=0pt,topsep=0pt]
			\item The covering relations are given by transpositions associated with minimal disorders.
			\item The poset $\bD_{\sS}^{v}$ is bounded, graded, subthin, and $\EL$-shellable.
			\item The bubble-sort map $\bbs_{\sS}:\bS_m\twoheadrightarrow \bD_{\sS}^{v}$ is a monotone, nonincreasing idempotent.
		\end{itemize}
	\end{theorem}
	\begin{proof}
		The covering relations were described in Theorem~\ref{thm::dominant::edge}. Corollary~\ref{cor::DS::bound} ensures that $\bD_{\sS}$ is bounded, while Corollary~\ref{cor::dominant::subthin} establishes the subthin property. The $\EL$-shellability is proved in Proposition~\ref{prp::EL::S::domninant}.  
		Finally, from Corollary~\ref{cor::bbs::monotone}, we know that $\bbs_{\sS}$ is a nonincreasing, monotone idempotent.
	\end{proof}
	
	\begin{corollary}
		\label{cor::Mobius::DL}
		If $\lambda$ is regular, then the following identity holds for the M\"obius function on the poset $\bD_{\sS}=\bD_{\sS}(\lambda)$:
		\[
		\mu^{\bD_{\sS}}_{\ov{a},\ov{b}} = \begin{cases}
			(-1)^{\rk(\ov{b})-\rk(\ov{a})}, & \text{if the interval } (\ov{a},\ov{b})^{\bS_m} = (\ov{a},\ov{b})^{\bD_{\sS}^{v}}, \\
			0, & \text{otherwise.}
		\end{cases}
		\]
		In other words, the M\"obius function is equal to $\pm1$ if every intermediate composition $\ov{c}\in\bS_m\lambda$ satisfying $\ov{a}\prec_{\Bruhat} \ov{c} \prec_{\Bruhat}\ov{b}$ is $\sS$-dominant, and it is zero otherwise.
	\end{corollary}
	\begin{proof}
		The poset $\bD_{\sS}(\lambda)$ is graded, $\EL$-shellable, and subthin, which implies that each interval in this poset is either a ball or a sphere (Fact~\ref{fact::Shell::ball}). As mentioned in Corollary~\ref{cor::Mobius::thin}, computing the M\"obius function requires determining when an interval is not thin. This happens precisely when no elements are lost when restricting the interval $(\ov{a},\ov{b})^{\bS_m\lambda} \subset \bS_m\lambda$ to the subposet $\bD_{\sS}^{v}(\lambda)$ of $\sS$-dominant weights.
	\end{proof}

	\subsubsection{{\bf Parabolic case}}
	\label{sec::S::parabolic}
	\label{sec::Parabolic::bbs}
	
	From now on we suppose that partition $\lambda=(\lambda_1^{k_1}, \ldots, \lambda_l^{k_l})$ may have repeating elements and the anti-liniarization $v:\sS\hookrightarrow [1m]$ may be nonsurjective.
	
	\begin{lemma}
		\label{lem::Admissible::poset}	
		The set $\bA_{\sS}(\lambda)$ of $\bS$-admissible compositions is the interval $[\lambda_+^{\sS},\lambda_{-}]$ in the Bruhat poset $\bS_m\lambda$ if $l(\lambda)\leq \#\sS$. On the other hand, if $l(\lambda)> \#\sS$, then $\bA_{\sS}(\lambda) =\emptyset$.
		
		Here, $\lambda_+^{\sS}$ denotes the unique $\sS$-dominant composition whose $l$'th nonzero element is equal to $\lambda_l$, while the remaining elements are zero.
	\end{lemma}
	\begin{proof}
		First, notice that the total number of nonzero elements in an $\sS$-admissible composition cannot exceed the size of $\sS$. Thus, any $\sS$-admissible composition belongs to $\bS_m\lambda$ with $l(\lambda)\leq k$. 
		Second, while decreasing the Bruhat order of an $\sS$-admissible composition $\nu\in\bS_m\lambda$, one can attempt to shift all zero elements to the right while preserving $\sS$-admissibility. What remains is determining the position of the zeros. Ultimately, this reduces to analyzing the Bruhat graph $\bS_k\lambda$ where $k=\# \sS$. The latter has a unique minimal element, $\lambda_+$.
	\end{proof}
	
	\begin{proposition}
		\label{prp::parab::idem}
		For any partition $\lambda$ with repeating elements, the idempotent $\bbs_{\sS}^k$ commutes with the parabolic projection $\pi_\lambda$ and its adjoint embedding $\psi_{\lambda}^{+}$:
		\begin{equation}
			\label{eq::idemp::square}
			\begin{tikzcd}
				\bAD_{\sS[<k]}=\bAD_{\sS[<k]}(\lambda+\delta) 
				\arrow[rr,"\pi_\lambda"', two heads, shift right = 1.7] 
				\arrow[rr,"\psi_\lambda^{+}","\perp"', hookleftarrow, shift left = 1.7] 
				\arrow[d,"\bbs_{\sS}^{k}"]
				&& 
				\bAD_{\sS[<k]}(\lambda)
				\arrow[d,"\bbs_{\sS}^{k}"]
				\\
				\bAD_{\sS[\leq k]}=\bAD_{\sS[\leq k]}(\lambda+\delta)
				\arrow[rr,"\pi_\lambda"', two heads, shift right = 1.7] 
				\arrow[rr,"\psi_\lambda^{+}","\perp"', hookleftarrow, shift left = 1.7] 
				&& 
				\bAD_{\sS[\leq k]}(\lambda)
			\end{tikzcd}
		\end{equation}       
	\end{proposition}
	\begin{proof}
		The proof follows from a straightforward comparison of these two idempotents based on the following simple observation:
		\[
		(\lambda_i\geq\lambda_j) \Rightarrow (\lambda+\delta)_i >(\lambda+\delta)_j.
		\]
		This ensures that while performing bubble-sort in $\bAD_{\sS[<k]}$, we first reorder the elements of the composition $\ov{d}$ that become equal in $\bAD_{\sS[<k]}(\lambda)$ and do other reorderings afterward.
	\end{proof}
	
	\begin{corollary}
		\label{cor::bbs::monotone:1}
		The bubble-sort idempotent $\bbs_{\sS}^{k}$ is monotone for all $\lambda$.
	\end{corollary}
	\begin{proof}
		The idempotent $\bbs_{\sS}^{k}:\bAD_{\sS[<k]}(\lambda)\twoheadrightarrow \bAD_{\sS[\leq k]}$ is the composition of three monotone maps $\pi_{\lambda}\circ \bbs_{\sS}^{k} \circ \psi_{\lambda}^{+}$ and is consequently also monotone.
	\end{proof}
	
	\begin{corollary}
		The poset $\bD_{\sS}^{v}(\lambda)$ is bounded.
	\end{corollary}
	\begin{proof}
		The poset $\bA_{\sS}^{v}$ is bounded (since it is an interval in the Bruhat graph), which implies that its image under the monotone idempotent is also bounded.
	\end{proof}
	
	Unfortunately, not all properties of the Bruhat graph carry over to the parabolic case, as seen in the following example.
	
	\begin{example}
		\label{ex::nongrad}
		The poset $\bS_{2}\backslash \bS_4 \slash \bS_2$ is not graded:
		$$
		\begin{array}{c}
			\sS:=
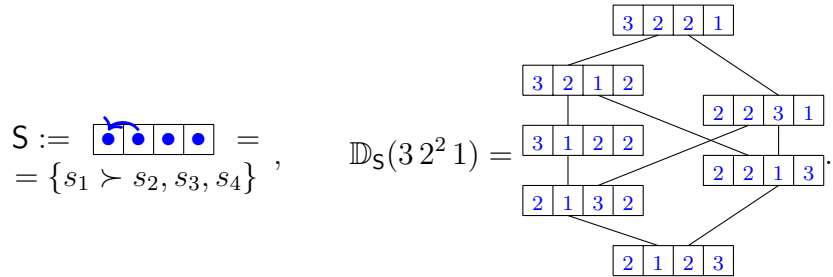
			\begin{tikzpicture}[scale=0.4, shift={(0,-.5)}]
				\draw[step=1cm] (0,0) grid (4,1);
				\node (v1) at (.5,0.5) {{ \color{blue}$\bullet$ }};
				\node (v2) at (1.5,0.5) {{ \color{blue}$\bullet$ }};
				\node (v3) at (2.5,0.5) {{ \color{blue}$\bullet$ }};
				\node (v4) at (3.5,0.5) {{ \color{blue}$\bullet$ }}; 
				\draw  (1.5,.8) edge[blue,line width=1.1pt,->,bend right=60] (.5,.8);
			\end{tikzpicture} = \\
			=\{s_1\succ s_2,s_3,s_4\}
		\end{array},
		\qquad
		\bD_{\sS}(3\, 2^2\, 1)=
		\begin{tikzpicture}[scale=0.4, shift={(0,-4)}]
			\draw[step=1cm] (-2,0) grid (2,1);
			\node (v01) at (-1.5,0.4) {\tiny \color{blue} 2};
			\node (v02) at (-.5,0.4) {\tiny \color{blue} 1};
			\node (v03) at (.5,0.4) {\tiny \color{blue} 2};
			\node (v04) at (1.5,0.4) {\tiny \color{blue} 3};
			\draw[step=1cm] (-5,2) grid (-1,3);
			\node (v01) at (-1.5-3,0.4+2) {\tiny \color{blue} 2};
			\node (v02) at (-.5-3,0.4+2) {\tiny \color{blue} 1};
			\node (v03) at (.5-3,0.4+2) {\tiny \color{blue} 3};
			\node (v04) at (1.5-3,0.4+2) {\tiny \color{blue} 2};
			\draw[step=1cm] (5,3) grid (1,4);
			\node (v01) at (-1.5+3,0.4+3) {\tiny \color{blue} 2};
			\node (v02) at (-.5+3,0.4+3) {\tiny \color{blue} 2};
			\node (v03) at (.5+3,0.4+3) {\tiny \color{blue} 1};
			\node (v04) at (1.5+3,0.4+3) {\tiny \color{blue} 3};
			\draw (-.5,1) -- (-3.5,2);
			\draw (.5,1) -- (3.5,3);
			\draw[step=1cm] (-5,4) grid (-1,5);
			\node (v01) at (-1.5-3,0.4+4) {\tiny \color{blue} 3};
			\node (v02) at (-.5-3,0.4+4) {\tiny \color{blue} 1};
			\node (v03) at (.5-3,0.4+4) {\tiny \color{blue} 2};
			\node (v04) at (1.5-3,0.4+4) {\tiny \color{blue} 2};
			\draw[step=1cm] (5,5) grid (1,6);
			\node (v01) at (-1.5+3,0.4+5) {\tiny \color{blue} 2};
			\node (v02) at (-.5+3,0.4+5) {\tiny \color{blue} 2};
			\node (v03) at (.5+3,0.4+5) {\tiny \color{blue} 3};
			\node (v04) at (1.5+3,0.4+5) {\tiny \color{blue} 1};
			\draw (-3.5,3) -- (-3.5,4);
			\draw (3.5,4) -- (3.5,5);
			\draw (-2.5,3) -- (2.5,5);
			\draw[step=1cm] (-5,6) grid (-1,7);
			\node (v01) at (-1.5-3,0.4+6) {\tiny \color{blue} 3};
			\node (v02) at (-.5-3,0.4+6) {\tiny \color{blue} 2};
			\node (v03) at (.5-3,0.4+6) {\tiny \color{blue} 1};
			\node (v04) at (1.5-3,0.4+6) {\tiny \color{blue} 2};
			\draw[step=1cm] (-2,8) grid (2,9);
			\node (v01) at (-1.5,0.4+8) {\tiny \color{blue} 3};
			\node (v02) at (-.5,0.4+8) {\tiny \color{blue} 2};
			\node (v03) at (.5,0.4+8) {\tiny \color{blue} 2};
			\node (v04) at (1.5,0.4+8) {\tiny \color{blue} 1};
			\draw (-3.5,5) -- (-3.5,6);
			\draw (-3.5,7) -- (-.5,8);
			\draw (3.5,6) -- (.5,8);
			\draw (2.5,4) -- (-2.5,6);
		\end{tikzpicture}.
		$$
	\end{example}

	The following conjectures have been verified in several specific cases. Notably, Conjecture~\ref{conj::Shell} implies Conjecture~\ref{conj::Moebius}, which is essential for determining coefficients in the Cauchy identity~\eqref{eq::Cauchy::Moebius} for staircase matrices.
	
	\begin{conjecture}
		\label{conj::Shell}
		For any partition $\lambda$ with repeating elements, the poset $\bD_{\sS}(\lambda)$ is shellable.
	\end{conjecture}
	
	\begin{conjecture}
		\label{conj::Moebius}
		The M\"obius function $\mu^{\bD_{\sS}(\lambda)}_{\ttt,\ttt}$ on the poset $\bD_{\sS}(\lambda)$ takes values in the set $\{-1,0,1\}$.
	\end{conjecture}

	\subsubsection{{\bf Arborescent Poset with Consistent Linearization}}
	\label{sec::linearization}
	
	We conclude this section by considering the opposite case of an arborescent poset $\sT$ enhanced with an consistent {\bf linearization} (Definition~\ref{def::poset::linear}).
	
	\begin{theorem}
		\label{thm::bbs::op}	
		Suppose that $(\sT,\prec)$ is an arborescent poset with an order-preserving consistent linearization $h:\sT\hookrightarrow [1n]$. Then, if $l(\lambda)\leq \# \sT$, the following properties hold for the set $\bD_{\sT}^{h}(\lambda)$ of $\sT$-dominant compositions:
		\begin{itemize}
			\item The set $\bD_{\sT}^{h}(\lambda)$ is empty if $l(\lambda)>\#\sT$.
			\item If $l(\lambda)\leq\#\sT$, then:
			\begin{itemize}
				\item The covering relations in the Bruhat partial order on the set $\bD_{\sT}^{h}(\lambda)$ of $\sT$-dominant compositions are given by transpositions in minimal disorders.
				\item There exists a nondecreasing monotone projection $\bbs_{\sT}^{\op}:[\lambda_+,\lambda_-^{\sT}] \to  \bD_{\sT}^{h}(\lambda)$ from the interval in the Bruhat graph $\bS_n\lambda$ to the set of $\sT$-dominant compositions.
				\item The poset $\bD_{\sT}^{h}(\lambda)$ is bounded.        
			\end{itemize}
			\item If $l(\lambda)=\#\sT$, then the poset $\bD_{\sT}^{h}$ is graded, subthin, and $\EL$-shellable.
		\end{itemize}
	\end{theorem}
	\begin{proof}
		This follows from Remark~\ref{rk::opposite::linearization} and the results already established for the consistent order-reversing anti-linearization of $\sT$.
	\end{proof}

	\section{Distributive lattice of Demazure submodules}
	\label{sec::Demazure::all}
	
	\subsection{Demazure Modules}
	\label{sec::Demazure::def}
	
	Let $\fg$ be a semi-simple (or reductive) Lie algebra, and let $\fb=\fb^{+},\fb^{-}$ be its positive and negative Borel subalgebras, with $\fh=\fb^{+}\cap\fb^{-}$ being the Cartan subalgebra.\nlabel{ni:Lie}
	Let $\Phi=\Phi_{+}\sqcup\Phi_{-}$ denote the corresponding root system, and let $P$ be the weight lattice of $\fg$ with $P_{+}$ as the subset of dominant weights. 
	The Weyl group $W$ acts naturally on the weight lattice $P$ and  for any given weight $\nu\in P$, there exists a unique element $\sigma_\nu\in W$ of minimal length such that $\nu=\sigma_{\nu}(\lambda)$ for some $\lambda\in P_+$. 
	
	Let $V_{\lambda}$ be the irreducible finite-dimensional $\fg$-module with highest weight $\lambda$. The $\nu$-weight subspace of $V_{\lambda}$ is one-dimensional, and we denote by $v_{\nu}$ a generator of this space (in particular, $v_{\nu}\in V_{\lambda}$ is an extremal weight vector).\nlabel{ni:Vlambda}
	M.~Demazure introduced (\cite{Dem1}) the submodule\nlabel{ni:Demazure}
	$$D_{\nu}:=\U(\fb^+)v_{\nu}\subset V_{\lambda},$$ 
	defined as the submodule generated by the extremal vector $v_{\nu}$, and established the following key observation:
	
	\begin{fact}
		For all $\eta,\nu\in W\lambda_+$, the following equivalence holds:
		$$\nu\prec_{\Bruhat}\eta \ {\Longleftrightarrow} \ D_{\nu} \subset D_{\eta}.$$
	\end{fact}
	
	These modules were later named \emph{Demazure modules} and were described in terms of generators and relations:
	
	\begin{theorem}[\cite{Joseph}]
		\label{thm::Joseph}
		For any $\nu\in P$, the defining relations for the Demazure $\fb$-module $D_\nu$ are given by:
		\begin{equation}
			\label{thmJoseph}
			e_\alpha^{\max\{-\langle \alpha^\vee,\lambda \rangle,0\}+1}v_\nu=0, \quad \forall \alpha \in \Phi_+,
		\end{equation}
		where $e_\alpha\in\fb_n$ is the Chevalley generator corresponding to the root $\alpha$.
	\end{theorem}
	
	We will also need the \emph{van der Kallen modules} ${\KK}_\lambda$ introduced by van der Kallen in~\cite{vdK}, which are $\fb$-modules defined as:\nlabel{ni:vdK}
	\[
	{\KK}_\nu = D_\nu\Bigl/\Bigr.\sum_{\eta\prec\nu} D_\eta = 
	D_\nu\Bigl/\Bigr.\sum_{\substack{\eta\in W\lambda\\ D_{\eta}\subsetneq D_{\lambda}}} D_\eta.
	\]
	and shown to be cyclic $\fb$-modules defined by the following relations:
	\begin{equation}
		e_\alpha^{\max(1,\max\{-\langle \alpha^\vee,\nu \rangle,0\})}v_\nu=0, \quad \forall \alpha \in \Phi_+.
	\end{equation}
	
	\begin{remark}
		The main example for us is $\fg=\mgl_n$, where $\fb^+=\fb_n^+$ and $\fb^-=\fb_n^{-}$ are the Lie subalgebras of upper-triangular and lower-triangular matrices, respectively, and $\fh_n=\fb^{+}_n\cap\fb^{-}_n$ is the diagonal Cartan subalgebra.\nlabel{ni:gln}
		The space $\fh_n$ is spanned by the diagonal matrix units, and its dual basis is denoted by $\{\varepsilon_i \colon i=1,\ldots,n\}$.
		In this case, $P_+\subset P$ consists of weights $\lambda = \sum_{i=1}^n \lambda_i \varepsilon_i$ such that $\lambda_i\geq \lambda_{i+1}$ for all $i$. 
		Thus, weights can be identified with compositions, and dominant weights are in bijection with partitions of length at most $n$. 
		The Weyl group $W$ is isomorphic to the symmetric group $\bS_n$. 
		
		The weight $\lambda = \sum_{i=1}^n \lambda_i \varepsilon_i$ can be considered as a weight of $\mgl_m$ for any $m \geq n$. In this setting, the corresponding composition is given by $(\lambda_1,\dots,\lambda_n,0,\dots,0)$.
	\end{remark}
	
	For an $\fh_n$-module $M$ and $\eta\in P$, let $M(\eta)\subset V$ be the
	weight $\eta$ subspace. We define the character as follows:
	\begin{equation}
		\label{eq::character}
		\ch_{\fh_n} M := \sum_{\eta=(\eta_1,\dots,\eta_n)} x_1^{\eta_1}\dots x_n^{\eta_n} \dim M(\eta). 
	\end{equation}
	In particular, for a partition $\lambda_+$, the character of the corresponding irreducible highest weight $\mgl_n$ module $V_{\lambda_+}$ is given by the Schur function $s_{\lambda_+}(x_1,\dots,x_n)$.
	The Demazure modules $D_\nu$ and the van der Kallen modules ${\KK}_\nu$ are labeled by compositions (i.e., arbitrary weights $\nu\in P$).
	One has\nlabel{ni:key}
	\[
	\kappa_\nu(x_1,\dots,x_n)= \ch_{\fh} D_\nu(x_1,\dots,x_n), \qquad
	a_\nu(x_1,\dots,x_n)=\ch_{\fh} {\KK}_\nu(x_1,\dots,x_n),
	\]
	where $a_\nu(x)$ and $\kappa_\nu(x)$ are the Demazure atoms and key polynomials (see \cite{Al,P} and references therein). 
	
	In this paper, we deal with left and right modules over the Borel subalgebra $\fb_m\subset \gl_m$.
	To emphasize the distinction, we use the upper index $\op$ for right modules.\nlabel{ni:op}
	Note that the right $\fb_m$-action decreases the weight, whereas the left action increases it. In particular, the right $\gl_m$-module $V_{\lambda}^{\op}$, is generated by a cyclic vector $v_{\lambda}$ as a $\fb_m$-module.
	
	\begin{notation}
		\label{not::rightDemazure}
		For each weight $\nu$, along with its corresponding dominant weight $\lambda_+$ and antidominant weight $\lambda_-=w_0\lambda_+$, we define the right Demazure $\fb_m$-module $D_\nu^{\op}$ as the submodule of the right $\fg_m$-module $V_{\lambda_+}^{\op}$ generated by the extremal vector of right weight $\nu$.
	\end{notation}
	
	Note that we have an opposite inclusion for opposite (right) Demazure $\fb_n^+$-modules:
	$$\nu \preceq_{\Bruhat} \eta  \ {\Longleftrightarrow} \ D_{\nu}^{\op} \supset D_{\eta}^{\op}.
	$$
	In particular, the opposite van der Kallen modules are described as follows:\nlabel{ni:Kop}
	$$
	{\KK}_{\nu}^{\op} = D_{\nu}^{\op}\left/ \sum_{\eta\succ_{\Bruhat}\nu} D_{\eta}^{\op}\right..
	$$
	
	Moreover, the Cartan subalgebra $\fh_m\subset\fb_{m}$ acts on the opposite (right) Demazure and van der Kallen modules, yielding the opposite key polynomials and opposite Demazure atoms:\nlabel{ni:keyop}
	\[
	\kappa^{\la}(y_1,\dots,y_m)=\ch_{\fh_m} D_\la^{\op},\ 
	a^{\la}(y_1,\dots,y_m)=\ch_{\fh_m} {\KK}_\la^{\op}.
	\]
	Directly from the definition we have 
	\begin{equation}\label{eq:left-right}
		\kappa^{\la}(y_1,\dots,y_m)=\kappa_{w_0\la}(y_m,\dots,y_1),\quad a^{\la}(y_1,\dots,y_m)=a_{w_0\la}(y_m,\dots,y_1).
	\end{equation}
	Here, $w_0$ is the longest element in the Weyl group. In the case of the symmetric group, it reverses the composition:\nlabel{ni:w0}
	$$w_0(\lambda_1,\ldots,\lambda_n)=(\lambda_n,\ldots,\lambda_1).$$
	Note that both $\kappa^{\la}(y_1,\dots,y_m)$ and $a^{\la}(y_1,\dots,y_m)$ contain the term $y^\la$, which serves as the leading term in the Bruhat order.
	
	\begin{example}
		Let $\lambda=\lambda_+$ be a partition (i.e., a dominant composition), and let
		$\lambda_-=w_0\lambda_+$ be the corresponding antidominant composition.
		Then we have the following isomorphisms:
		$$D_{\lambda_{-}} \simeq V_{\lambda_+}, \quad \dim(D_{\lambda_+} )= \dim( D_{\lambda_{-}}^{\op})=1, \quad D_{\lambda_{+}}^{\op} \simeq V_{\lambda_{+}}. $$
		Consequently, we obtain the following equalities for the corresponding polynomials:
		$$
		\kappa_{\lambda_{-}}(x) = s_{\lambda}(x) = \kappa^{\lambda_+}(x), \quad \kappa_{\lambda_+}(x) = a_{\lambda_+}(x) = x^{\lambda_+}; \quad \kappa^{\lambda_-}(x) = a^{\lambda_-}(x) = x^{\lambda_-}.
		$$
	\end{example}

	\subsection{Demazure Modules Form a Distributive Lattice}
	\label{sec::Demazure::distributive}
	In this section, we provide a proof of Theorem~\ref{thm::Demazure::intersect} below. 
	The statement of the theorem is known to experts: the existence of a basis of $V_\lambda$ compatible with all Demazure modules goes back to Kashiwara~\cite{Kashiwara}, and P.~Littelmann gave a constructive proof (an algorithm producing such a basis) valid for all symmetrizable Kac--Moody algebras~\cite{Littleman}.
	The proof below is based on Frobenius splitting; we are grateful to Michel Brion for explaining to us its key idea, and we refer to S.~Kato~\cite[Corollary 2.22]{Kato::Frobenius} for the extension of Frobenius splitting to affine Lie algebras.
	
	Let $G$ be a simple Lie group, and let $V_\la$ be its irreducible finite-dimensional highest weight representation.
	Demazure modules naturally arise in the geometry of the flag variety $F=G/B$. 
	In particular, it is known that there exists a natural line bundle $\cL_\lambda$ associated with $\la\in P_+$ such that\nlabel{ni:linebundle}
	$$ H^{0}(F;\cL_{\lambda}) \simeq V_\lambda^*, \quad H^{0}(X_{w};\cL_{\lambda}) \simeq D_{w\lambda}^*.
	$$
	Here, $X_w\subset F$ is the Schubert variety corresponding to $w$. (See, e.g.,~\cite{Kum} for an introduction to the geometry of flag varieties).
	Let $D_{w\lambda}^{\perp}\subset H^{0}(F;\cL_{\lambda})$ denote the annihilator of $D_{w\lambda}$:
	$$
	D_{w\lambda}^{\perp}:= \ker( H^0(F,\cL_{\lambda}) \twoheadrightarrow H^0(X_w,\cL_{\lambda})) \simeq \ker \left(V_\lambda^* \twoheadrightarrow D_{w\lambda}^*\right).
	$$
	
	\begin{theorem}
		\label{thm::Demazure::intersect}
		The subspaces $\{D_{w\lambda} \colon w\in W\}$ of the integrable representation $V_{\lambda}$ generate a distributive lattice, denoted $\Lattice_{D}(V_{\lambda})$.
		Equivalently, the dual lattice generated by $\{D_{w\lambda}^{\perp}\colon w\in W\}\subset \Lattice(V_{\lambda}^*)$ is distributive.
		
		Moreover, Demazure modules constitute a set of $\vee$-irreducible elements of this lattice $\Lattice_{D}(V_{\lambda})$.
	\end{theorem}
	\begin{proof}
		From the definition of the Bruhat order, we know that for any dominant weight $\lambda$ and any pair of comparable elements in the Weyl group $W$, we have:
		$$
		w_1\prec_{\Bruhat} w_2 \in W \ \Rightarrow \ D_{w_1\lambda} \subset D_{w_2\lambda} \ \Leftrightarrow \ D_{w_1\lambda}^{\perp} \supset D_{w_2\lambda}^{\perp}.
		$$
		Consequently, for any collection of pairwise uncomparable elements $w_1,\ldots,w_k$ in the Weyl group $W$, we obtain the inclusion
		\begin{equation}\label{conjgeom}
			D_{{w_1\lambda}}^{\perp} + \ldots + D_{{w_k\lambda}}^{\perp}\ \subset \ \bigcap_{w \colon \forall i \ w\preceq w_i} D_{w\lambda}^{\perp}.
		\end{equation}
		We now show that this inclusion is actually an equality.
		In other words, we demonstrate that if a section $s\in H^0(F,\cL_{\lambda})$ vanishes when restricted to the intersection of 
		Schubert varieties $\cap_{i=1}^{k}X_{w_k}$, then it can be written as a sum $s=s_1+\ldots+s_k$ of $k$ sections, where each $s_i$ vanishes on $X_{w_i}$.
		Denote by $I_{Y}$ the ideal sheaf of the structure sheaf of $F$ consisting of functions vanishing on a subvariety $Y\subset F$. In particular, for Schubert varieties $X_{w}$, we have\nlabel{ni:ideal}
		$$
		D_{w\lambda}^{\perp}\simeq H^{0}(F;\cL_{\lambda}\otimes I_{X_{w}}).$$
		Since the schematic intersection of Schubert varieties is reduced, we obtain
		\begin{gather*}
			{ I_{X_{w_1}\cap\ldots\cap X_{w_k}} = I_{X_{w_1}}+\ldots+I_{X_{w_k}}; } \\
			{I_{X_{w_1}\cup\ldots\cup X_{w_k}} = I_{X_{w_1}}\cap\ldots\cap I_{X_{w_k}} }.
		\end{gather*}
		Thus, we have an exact sequence of ideal sheaves (originating from the inclusion-exclusion principle):
		$$
		\begin{tikzcd}
			0 \arrow[r]  &
			I_{\bigcup_{i=1}^{k}X_{w_i}} \arrow[r] &
			\oplus_{s=1}^{k} I_{\bigcup_{\substack{{i=1}\\ {i\neq s}}}^{k} X_{w_i}} \arrow[r] & \ldots \arrow[r] & \oplus_{s=1}^{k} I_{X_{w_s}} \arrow[r,two heads] & I_{\bigcap_{i=1}^{k} X_{w_i}} \arrow[r] & 0.
		\end{tikzcd} 
		$$
		Due to the Frobenius splitting of the flag variety $F$, we know that for any subset $\{u_1,\ldots, u_m\}\subset W$, the higher cohomology vanishes (see \cite[Theorems 1.4.8 and 2.3.1]{Brion_Kumar::Frobenius}): 
		$$H^{>0}(F;\cL\otimes I_{\cup_{i=1}^{m}X_{u_i}}) = 0.$$
		This leads to an exact sequence of sections of the line bundle $\cL$:
		\begin{multline*}
			0 \to
			H^0(F;\cL\otimes I_{\bigcup_{i=1}^{k}X_{w_i}}) \to
			\oplus_{s=1}^{k} H^{0}(F;\cL\otimes I_{\bigcup_{\substack{{i=1}\\ {i\neq s}}}^{k} X_{w_i}}) \to \ldots \to \\
			\to \oplus_{s=1}^{k} H^{0}(F;\cL\otimes I_{X_{w_s}}) \to H^{0}(F; \cL\otimes I_{\bigcap_{i=1}^{k} X_{w_i}} ) \to 0.
		\end{multline*}
		Thus, we obtain the inclusion-exclusion formula for the intersection of Demazure modules and a surjective map
		$$
		\oplus_{i=1}^{k} D_{w_i\lambda}^{\perp}\twoheadrightarrow H^{0}(F;\cL\otimes I_{\bigcap_{i=1}^{k} X_{w_i}}) = \bigcap_{w\colon\forall i w\preceq w_i} D_{w\lambda}^{\perp}.
		$$
		From this, it follows that the inclusion~\eqref{conjgeom} is actually an equality, establishing that the lattice of submodules $D_{w\lambda}^{\perp}$ of the integrable module $V(\lambda)^{*}$ is distributive.
		
		It remains to show that Demazure submodules are $\vee$-irreducible elements of the lattice $\Lattice_{D}(V_\lambda)$.
		This follows from the observation:
		$$\forall U\in \Lattice_{D}(V_\lambda), \quad D_{w\lambda}\subset U	\Leftrightarrow v_{w\lambda}\in U.$$
	\end{proof}
	
	It is worth mentioning that all subspaces of the lattice $\Lattice_{D}(V_\lambda)$ are $\fb$-submodules of $V_{\lambda}$ and, in particular, $\fh$-modules.
	Consequently, each quotient $U_1/U_2$ has a well-defined character $\ch$ (see~\eqref{eq::character}).
	Let us outline the corollaries of~\S\ref{sec::Distributive}, which apply naturally to our situation:
	
	\begin{corollary}
		\label{cor::XU}
		Let $U\in \Lattice_{D}(V_\lambda)$, and let $X_U\subset W\lambda$ be the subset of the parabolic Bruhat graph consisting of those $w\lambda$ such that 
		$$U\supset D_{w\lambda} \ \Leftrightarrow \ \text{the extremal vector } v_{w\lambda} \in U.$$
		Then we have:
		\begin{itemize}[itemsep=0pt,topsep=0pt]
			\item $U\simeq \sum_{\eta\in X_U} D_{\eta}$;
			\item For the filtration $\calF_{\eta}U:= U\cap D_{\eta}$, the associated graded space has the following description:
			$$\calF_{\eta} U/\sum_{\nu\prec \eta}\calF_{\nu} U \simeq \begin{cases}
				\KK_{\eta}, & \text{if } \eta\in X_U; \\
				0, & \text{if } \eta\notin X_U.
			\end{cases} $$
			where $\KK_{\eta}$ is the van der Kallen module;
			\item The character of $U$ is equal to the sum of Demazure atoms:
			$$\ch(U) = \sum_{\eta\in X_U} a_{\eta}(x).$$
			\item The character of $U$ is a linear combination of key polynomials $\{\kappa_{\nu}\colon \nu\in W\lambda\}$ with integer coefficients.
		\end{itemize}
	\end{corollary}
	
	\begin{proof}
		The van der Kallen module ${\KK}_{w \lambda}$ coincides with the one described in~\eqref{eq::min::subquot}. Each van der Kallen module contains a unique nonzero image of an extremal vector from $V_{\lambda}$. Thus, the first three items of the corollary follow directly from the general theory outlined in Fact~\ref{fact::distributive}. The last item can be deduced either from Proposition~\ref{prp::distibutive::resolution} or from the M\"obius inversion formula (Fact~\ref{fact::Mobius}).
	\end{proof}
	
	\subsection{Sublattice Indexed by $\sS$-Dominant Weights}
	Let us now apply the combinatorial results from~\S\ref{sec::Bubble::sort::all} and the findings of the preceding section~\S\ref{sec::Demazure::distributive} to the specific case of $\fg=\gl_n$, $W=\bS_n$, and an arborescent poset with (anti)linearization.
	
	\begin{corollary}
		\label{cor::distr::dominant}	
		Suppose that $(\sT,h)$ is an arborescent poset with linearization (Definition~\ref{def::poset::linear}) and that $\lambda$ is a partition of length at most $\#\sT$.
		Then the set 
		$$\{D_{\ov{a}}\subset V_{\lambda}\ \colon \  \ov{a}\in\bD_{\sT}(\lambda)\}$$
		forms the collection of $\vee$-irreducible elements of the distributive sublattice $\Lattice_{D}^{\sT}(\lambda)$ of $\Lattice_{D}(\lambda)$ generated by these elements. 
		
		Moreover, the minimal subquotients:
		\begin{equation}
			\label{eq::vdK::general}
			\KK_{\sT,\ov{a}} := D_{\ov{a}}\left/ \sum_{\ov{c}\in\bD_{\sT}(\lambda) \colon \ov{c}\prec \ov{a}} D_{\ov{c}} \right.
		\end{equation}
		admit a filtration induced by $\Lattice_{D}(\lambda)$, indexed by the elements of the Bruhat graph $W\lambda$, such that the associated graded module satisfies:
		\begin{equation}
			\label{eq::vdK::a}
			\gr^{\calF}\KK_{\sT,\ov{a}} \simeq \bigoplus_{\substack{{\ov{b}\in\bS_{n}\lambda} \colon {\bbs_{\sT}^{\op}(\ov{b})= \ov{a} }}} \KK_{\ov{b}}. 
		\end{equation} 
		The modules ${\KK}_{\sT,\ov{a}}$ are referred to as \emph{generalized van der Kallen modules}.
	\end{corollary}	
	\begin{proof}
		Recall that the opposite bubble-sort operation $\bbs_{\sT}^{\op}$ defines an increasing monotone idempotent on the interval of Bruhat graph, whose image consists of $\sT$-antidominant compositions (Theorem~\ref{thm::bbs::op}). 
		Consequently, the assumptions of Proposition~\ref{prp::distributive::idempotent} are satisfied, and it remains to carefully apply these results in our setting.
	\end{proof}
	
	\begin{corollary}
		\label{cor::min::char}	
		We have the following two equivalent descriptions of the $\fh_n$-character of the minimal subquotients:	
		\begin{equation}
			\label{eq::gen::vdK::character}
			\ch_{\fh_n}({\KK}_{\sT,\ov{a}})(x) = 
			\sum_{\substack{\ov{c} \in \bD_{\sT}(\lambda)\\  \ov{c}\preceq_{\Bruhat} \ov{a}}} \mu^{\bD_{\sT}(\lambda)}_{\ov{c},\ov{a}} \kappa_{\ov{c}}(x)
			= \sum_{\substack{\ov{b} \colon {\bbs_{\sT}^{\op}(\ov{b})= \ov{a} }}} a_{\ov{b}}(x).
		\end{equation}
		Here, $\mu^{\bD_{\sT}(\lambda)}_{\ov{c},\ov{a}}$ is the M\"obius function on the poset $\bD_{\sT}(\lambda)$ discussed in Subsection~\ref{sec::poset::terms}, $\kappa_{\ov{a}}(x)$ is the key polynomial, and $a_{\ov{a}}(x)$ is the Demazure atom associated with the composition $\ov{a}$.
	\end{corollary}
	\begin{proof}
		The first equality follows from the description~\eqref{eq::vdK::general} of the generalized van der Kallen module, while the second follows from the associated graded structure given in~\eqref{eq::vdK::a}.
	\end{proof}
	
	Let us also formulate a similar statement for the right Demazure modules: 
	\begin{corollary}
		\label{cor::distr::dominant::op}	
		Suppose that $(\sS,v)$ is an arborescent poset with a linearization and that $\lambda$ is a partition of length at most $\#\sS$.
		Then the right Demazure submodules 
		$$\{D_{\ov{a}}^{\op}\subset V_{\lambda}^{\op} \simeq V_{-w_0\lambda}\ \colon \  \ov{a}\in \bS_{n}\lambda\}$$
		form the set of $\vee$-irreducible elements of the distributive lattice $\Lattice_{D}^{\op}(\lambda)$. The subset of Demazure submodules indexed by $\sS$-dominant compositions constitutes a subset of $\vee$-irreducible elements of the distributive sublattice generated by them:
		$$\Lattice_{D}^{\sS}(\lambda)^{\op}:= \left\langle D_{\ov{a}}^{\op} \ \colon \ \ov{a}\in\bD_{\sS}(\lambda)\right\rangle \ \subset \ \left\langle D_{\ov{a}}^{\op} \ \colon \ \ov{a}\in\bS_n\lambda\right\rangle =: \Lattice_{D}(\lambda)^{\op}. $$ 
		
		The minimal subquotients:
		$${\KK}_{\sS,\ov{a}}^{\op} := D_{\ov{a}}^{\op}\left/ \sum_{\eta\in\bD_{\sS}(\lambda) \colon \eta\succ \ov{a}} D_{\eta}^{\op} \right.$$
		have the following two descriptions of their characters:
		\begin{equation}
			\label{eq::vdK::op::char}
			\ch_{\fh_m}\left({\KK}_{\sS,\ov{a}}^{\op}\right) = 
			\sum_{\substack{\ov{c} \in \bD_{\sS}(\lambda)\\  \ov{c}\succeq_{\Bruhat} \ov{a}}} \mu^{\bD_{\sS}(\lambda)}_{\ov{c},\ov{a}}\, \kappa^{\ov{c}}(y)  
			= \sum_{\ov{a}\colon \bbs_{\sS}(\ov{a})= \ov{a} } a^{\ov{a}}(y).
		\end{equation}
	\end{corollary}
	
	It is worth mentioning that, on one hand, any subspace $D\in\Lattice_{D}^{\sT}(\lambda)$ admits a resolution by a direct sum of Demazure modules (Proposition~\ref{prp::distibutive::resolution}). 
	Consequently,  whenever we are able to show that the M\"obius function $\mu^{\bD_{\sT}(\lambda)}_{\ov{c},\ov{a}}$ takes values in $\{0,1,-1\}$ (Corollary~\ref{cor::Mobius::DL} and Conjecture~\ref{conj::Moebius}) we will be able to extract an explicit description of this resolution for the minimal subquotients ${\KK}_{\sT,\ov{a}}$, generalizing the classical BGG resolution.

	\section{Staircase shapes, corners and $DL$-dense arrays}
	\label{sec::Staircase}
	In the first two subsections,~\S\ref{sec::Young::St}--\S\ref{sec::staircase::corners}, we recall the main combinatorial objects and notations introduced in~\cite{FKM::Cauchy}. In~\S\ref{sec::Bruhat::DL-dense}, we develop new combinatorial structures, relate them to those introduced in~\S\ref{sec::Bubble::sort::all}, and apply them in~\S\ref{sec::DL::Distr}.

	\subsection{Young Diagrams in Staircase Shape}
	\label{sec::Young::St}
	
	The main focus of this note is the combinatorics surrounding staircase matrices, which arise as an output of the Gaussian elimination process. The shapes of these staircase matrices are indexed by Young diagrams, drawn in a reversed form.
	
	Let us fix a positive integer $m$ and a sequence of integers $\ov{n}:=(n_1,\ldots,n_m)$ satisfying\nlabel{ni:nbar}
	$$0<n_1\leq n_2 \leq \ldots \leq n_m.$$
	The corresponding Young diagram (or \emph{staircase shape}) is defined as\nlabel{ni:Ynbar}
	\[
	\bbY_{\overline{n}} = \{(i,j): 1\le j\le m,\ 1\le i\le n_j\}\subset \bZ_{>0}^2.
	\]
	This diagram is visualized as a collection of $m$ columns of heights $n_1,\dots,n_m$, arranged from left to right.
	In particular, the first index $i$ corresponds to the row, while the second index $j$ denotes the column number.
	For example, the cells in the top row have the form $(1,j)$ for 
	$1\le j\le m$.
	
	\begin{example}
		Here is an example of a Young diagram $\bbY_{\ov{n}}$ 
		associated with $\overline n =(1^3 3^3 4^2)$: 
		\[\bbY_{(1^3\,3^3\,4^2)} = \begin{tikzpicture}[scale=0.4]
			\draw[step=1cm] (-1,1) grid (7,0);
			\draw[step=1cm] (2,0) grid (7,-1);
			\draw[step=1cm] (2,-1) grid (7,-2);
			\draw[step=1cm] (5,-2) grid (7,-3);
		\end{tikzpicture}. \]
	\end{example}

	The reason for rotating the visualization of the Young diagram is that our primary focus is on staircase matrices, which we want to be acted upon by upper-triangular matrices from both the left and the right.
	
	\begin{notation}
		\label{not::Mat}
		The subspace of rectangular 
		$n_m\times m$ matrices consisting of matrices with zero entries outside the cells belonging to 
		the Young diagram $\bbY_{\ov{n}}$ is called \emph{the set of staircase matrices of shape $\bbY_{\ov{n}}$} and is denoted by $\Mat_{\ov{n}}$.
	\end{notation}
	
	\subsection{Poset of Staircase Corners}
	\label{sec::staircase::corners}
	
	In this subsection, we associate each partition $\ov{n}$ with a subset of indices $\St_{\ov{n}}\subset \bbY_{\ov{n}}$, called the \emph{staircase corners}. We define a partial order on this set and describe some fundamental properties of the resulting poset.
	
	For a cell $(i,j)\in \bbY_{\overline{n}}$, we define the partition $\overline{n^{ij}}=(n^{ij}_1\le\ldots\le n^{ij}_{m-1})$ obtained from $\overline{n}$ by removing the $i$-th row and the $j$-th column:
	\begin{equation}
		\label{eq::erased::diagram}
		n^{ij}_k:=\begin{cases}
			n_k, \text{ if } k< j \ \& \ n_k < i, \\
			n_k-1, \text{ if } k < j \ \& \ n_k \geq i, \\
			n_{k+1}-1, \text{ if } k\geq j. 
		\end{cases}
	\end{equation}
	
	There is a natural bijection between the cells of $\bbY_{\ov{n}}$ with the $i$-th row and $j$-th column removed and the cells of $\bbY_{\ov{n^{ij}}}$, denoted by $\pi_{i,j}$:
	\begin{equation}
		\label{eq::erased::pi}    
		\pi_{ij}(s,t):=\begin{cases}
			(s,t) \text{ if } s<i\ \& \ j<t,\\
			(s-1,t) \text{ if } s>i\ \& \ j<t,\\
			(s,t-1) \text{ if } s<i\ \& \ j>t,\\
			(s-1,t-1) \text{ if } s>i\ \& \ j>t.
		\end{cases}
	\end{equation}
	
	\begin{definition}
		\label{def::DL::indices}
		The subset $\St_{\ov{n}}\subset \bbY_{\overline{n}}$ of \emph{staircase corners} is defined inductively by the following properties:
		\begin{itemize}
			\item In each row and each column, at most one cell (staircase corner) belongs to $\St_{\ov{n}}$.
			\item If $n_j>n_{j-1}$, then the corner cell $(n_j,j)$ of the Young diagram $\bbY_{\ov{n}}$ belongs to $\St_{\ov{n}}$.
			\item 
			For any $(i,j)\in\St_{\ov{n}}$, another cell $(s,t)$ belongs to $\St_{\ov{n}}\setminus\{(i,j)\}$ if and only if $\pi_{ij}(s,t)\in \St_{\ov{n^{ij}}}$.
			In other words, for any staircase corner $(i,j)\in\St_{\ov{n}}$, any other staircase corner is also a staircase corner of the diagram obtained by removing the $i$-th row and $j$-th column.
		\end{itemize} 
	\end{definition}
	
	\begin{remark}
		\label{rem::row::deletion}\label{rem::row::erase} 
		Suppose that the $j$'th column of the Young diagram $\bbY_{\ov{n}}$
		does not contain a staircase corner. Then by erasing the $j$-th column we get a bijection between $\St_{\ov{n}}$ and $\St_{\ov{n}_j''}$, where $\ov{n}_j'':=(n_1\leq \ldots\leq \widehat{n_j}\leq\ldots\leq n_m).$
		Similarly, if the $i$'th row of the Young diagram $\bbY_{\ov{n}}$ does not contain a staircase corner and $\ov{n}'_i$ is the partition then there is a bijection between $\St_{\ov{n}}$ and $\St_{\ov{n}'_i}$ where $\ov{n}'_i$ -- is a partition with $i$'th row ommited.
	\end{remark}

	It was shown in~\cite[Lemma 1.24]{FKM::Cauchy} that for any partition $\ov{n}$, the set $\St_{\ov{n}}$ of staircase corners is well-defined. Moreover, $\St_{\ov{n}}$ forms a "rook placement" in the Young diagram $\bbY_{\overline{n}}$, meaning that for any cell $(i,j)\in \St_{\overline{n}}$, there are no other staircase corners in the same row or column.
	Evgeny Smirnov explained to us in private communication that the "rook placement" of staircase corners $\St_{\ov{n}}$ is maximal under an appropriate order.
	
	We equip the set $\St_{\overline{n}}$ of staircase corners with the following partial order:
	\begin{equation}
		\label{eq::DL::order}
		(i,j) \succeq (i',j')\ \stackrel{\mathsf{def}}{\Leftrightarrow} \ \begin{cases}
			i\geq i', \\
			j\leq j'.
		\end{cases} 
	\end{equation}
	In other words, an element is enlarged by moving in the down-left direction.
	
	Below, we provide a few examples of the set $\St_{\ov{n}}$ of staircase corners along with the Hasse diagram of the corresponding partial order.
	We represent the elements of $\St_{\ov{n}}$ as blue dots and draw edges in the Hasse diagram oriented from smaller to larger elements.
	The first examples include rectangular shapes, triangular shapes, and a small mixture of both:
	\begin{equation}
		\label{pic::rectangular::corners}
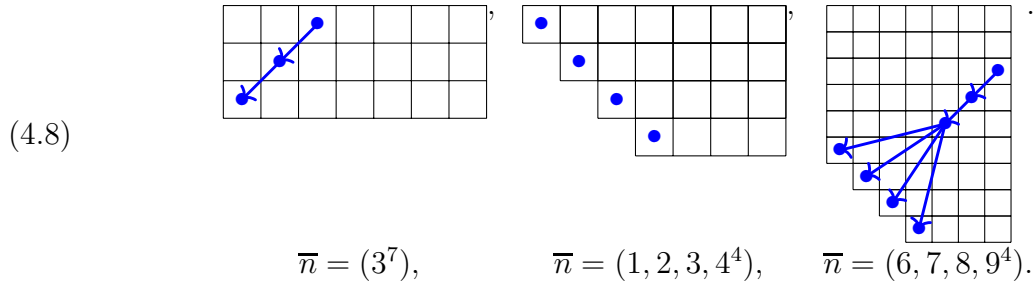
		\begin{array}{ccc}
			{\begin{tikzpicture}[scale=0.5]
					\draw[step=1cm] (0,-3) grid (7,0);
					\node (v00) at (.5,-2.5) { {{\color{blue}$\bullet$}}};
					\node (v10) at  (.5+1,-1.5) {{ \color{blue}$\bullet$ }};
					\node (v20) at (.5+2,-0.5) {{ \color{blue}$\bullet$ }};
					\draw [blue,line width=1.1pt,->] (.5+2,-0.5) -- (.5+1,-1.5);
					\draw [blue,line width=1.1pt,->] (.5+1,-1.5) -- (.5,-2.5);
				\end{tikzpicture},
			}
			&
			{
				\begin{tikzpicture}[scale=0.5]
					\draw[step=1cm] (0,-1) grid (7,0);
					\draw[step=1cm] (1,-2) grid (7,0);
					\draw[step=1cm] (2,-3) grid (7,0);
					\draw[step=1cm] (3,-4) grid (7,0);
					\node (v00) at (.5,-0.5) { {{\color{blue}$\bullet$}}};
					\node (v10) at  (.5+1,-1.5) {{ \color{blue}$\bullet$ }};
					\node (v20) at (.5+2,-2.5) {{ \color{blue}$\bullet$ }};
					\node (v20) at (.5+3,-3.5) {{ \color{blue}$\bullet$ }};
			\end{tikzpicture} },
			& 
			{
				\begin{tikzpicture}[scale=0.35]
					\draw[step=1cm] (0,-1) grid (7,0);
					\draw[step=1cm] (0,-2) grid (7,-1);
					\draw[step=1cm] (0,-3) grid (7,-2);
					\draw[step=1cm] (0,-4) grid (7,-3);
					\draw[step=1cm] (0,-5) grid (7,-4);
					\draw[step=1cm] (0,-6) grid (7,-5);
					\draw[step=1cm] (1,-7) grid (7,-6);
					\draw[step=1cm] (2,-8) grid (7,-7);
					\draw[step=1cm] (3,-9) grid (7,-8); 
					\node (v0) at (.5,-5.5) { {{\color{blue}$\bullet$}}};
					\node (v1) at  (.5+1,-6.5) {{ \color{blue}$\bullet$ }};
					\node (v2) at (.5+2,-7.5) {{ \color{blue}$\bullet$ }};
					\node (v3) at (.5+3,-8.5) {{ \color{blue}$\bullet$ }};
					\node (v4) at (.5+4,-4.5) {{ \color{blue}$\bullet$ }};
					\node (v5) at (.5+5,-3.5) {{ \color{blue}$\bullet$ }};
					\node (v6) at (.5+6,-2.5) {{ \color{blue}$\bullet$ }};
					%
					%
					\draw [blue,line width=1.1pt,->] (.5+6,-2.5) -- (.5+5,-3.5);
					\draw [blue,line width=1.1pt,->] (.5+5,-3.5) -- (.5+4,-4.5);
					\draw [blue,line width=1.1pt,<-] (.5,-5.5) -- (.5+4,-4.5);
					\draw [blue,line width=1.1pt,<-] (.5+1,-6.5) -- (.5+4,-4.5);
					\draw [blue,line width=1.1pt,<-] (.5+2,-7.5) -- (.5+4,-4.5);
					\draw [blue,line width=1.1pt,<-] (.5+3,-8.5) -- (.5+4,-4.5);
				\end{tikzpicture}. 
			}
			\\
			\ov{n}=(3^7), & \ov{n}=(1,2,3,4^4), & \ov{n}=(6,7,8,9^4).
		\end{array}
	\end{equation}
	For a larger example {\small $\ov{n}=(3^4\,5\,9^2\,13^{10}\,16^2)$}, we attempt to visualize the algorithm for identifying the staircase corners. Specifically, we begin by coloring in lime the cells that belong to the rows and columns of the corner cells of the Young diagram. Next, we examine the remaining cells and mark in yellow the hooks associated with the corners of the Young diagram obtained after removing the lime cells. We continue this process, using a new color at each step, and mark the centers of the identified staircase corner cells with blue dots. Finally, we overlay a Hasse diagram on top of the visualization.
	\[
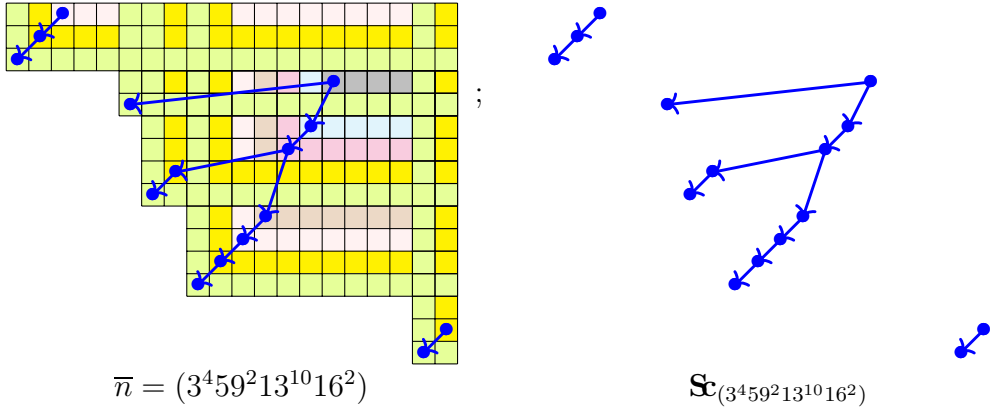
	\begin{array}{ccc}
		\begin{tikzpicture}[scale=0.3,shift={(0,2)}]
			\fill[lightgray] (14,-2) -- (14,-1) -- (20,-1)--(20,-2);
			\fill[cyan!10] (13,-4) -- (13,2) -- (14,2)--(14,-4);
			\fill[cyan!10] (13,-4) -- (13,-3) -- (20,-3)--(20,-4);
			\fill[magenta!25] (12,-5) -- (12,2) -- (13,2)--(13,-5);
			\fill[magenta!25] (12,-5) -- (12,-4) -- (20,-4)--(20,-5);
			\fill[brown!30] (11,-8) -- (11,2) -- (12,2)--(12,-8);
			\fill[brown!30] (11,-8) -- (11,-7) -- (20,-7)--(20,-8);
			\fill[pink!20] (2,1) -- (2,2) -- (3,2)--(3,1);
			\fill[pink!20] (2,1) -- (2,2) -- (20,2)--(20,1);
			\fill[pink!20] (10,-9) -- (10,2) -- (11,2)--(11,-9);
			\fill[pink!20] (10,-9) -- (10,-8) -- (20,-8)--(20,-9);
			\fill[yellow] (19,-13) -- (19,2) -- (20,2)--(20,-13);
			\fill[yellow] (19,-13) -- (19,-12) -- (20,-12)--(20,-13);
			\fill[yellow] (9,-10) -- (9,2) -- (10,2)--(10,-10);
			\fill[yellow] (9,-10) -- (9,-9) -- (20,-9)--(20,-10);
			\fill[yellow] (7,-6) -- (7,2) -- (8,2)--(8,-6);
			\fill[yellow] (7,-6) -- (7,-5) -- (20,-5)--(20,-6);
			\fill[yellow] (1,0) -- (1,2) -- (2,2)--(2,0);
			\fill[yellow] (1,0) -- (1,1) -- (20,1)--(20,0);
			\fill[lime!40] (6,-7) -- (6,2) -- (7,2)--(7,-7);
			\fill[lime!40] (7,-7) -- (7,-6) -- (20,-6)--(20,-7);
			\fill[lime!40] (5,-3) -- (5,2) -- (6,2)--(6,-3);
			\fill[lime!40] (7,-3) -- (7,-2) -- (20,-2)--(20,-3);
			\fill[lime!40] (0,-1) -- (0,2) -- (1,2)--(1,-1);
			\fill[lime!40] (1,-1) -- (1,-0) -- (20,-0)--(20,-1);
			\fill[lime!40] (8,-11) -- (8,2) -- (9,2)--(9,-11);
			\fill[lime!40] (9,-11) -- (9,-10) -- (20,-10)--(20,-11);
			\fill[lime!40] (18,-14) -- (18,2) -- (19,2)--(19,-14);
			\fill[lime!40] (19,-14) -- (19,-13) -- (20,-13)--(20,-14);
			\draw[step=1cm] (0,-1) grid (20,2);
			\draw[step=1cm] (5,-3) grid (20,-1);
			\draw[step=1cm] (6,-7) grid (20,-1);
			\draw[step=1cm] (8,-11) grid (20,-1);
			\draw[step=1cm] (18,-14) grid (20,-1);
			\node (v00) at (.5,-0.5) { {{\color{blue}$\bullet$}}};
			\node (v10) at  (.5+1,0.5) {{ \color{blue}$\bullet$ }};
			\node (v20) at (.5+2,1.5) {{ \color{blue}$\bullet$ }};
			\node (v00) at (5.5,-2.5) { {{\color{blue}$\bullet$}}};
			\node (v10) at  (5.5+9,-1.5) {{ \color{blue}$\bullet$ }};
			\node (v00) at (6.5,-6.5) { {{\color{blue}$\bullet$}}};
			\node (v10) at  (6.5+1,-5.5) {{ \color{blue}$\bullet$ }};
			\node (v00) at (8.5,-10.5) { {{\color{blue}$\bullet$}}};
			\node (v10) at  (8.5+1,-9.5) {{ \color{blue}$\bullet$ }};
			\node (v00) at (8.5+2,-8.5) { {{\color{blue}$\bullet$}}};
			\node (v10) at  (8.5+3,-7.5) {{ \color{blue}$\bullet$ }};
			\node (v00) at (12.5,-4.5) { {{\color{blue}$\bullet$}}};
			\node (v10) at  (13.5,-3.5) {{ \color{blue}$\bullet$ }};
			\node (v10) at  (18.5,-13.5) {{ \color{blue}$\bullet$ }};
			\node (v10) at  (19.5,-12.5) {{ \color{blue}$\bullet$ }};
			\node (v00) at (.5,-0.5) { {{\color{blue}$\bullet$}}};
			\node (v10) at  (.5+1,0.5) {{ \color{blue}$\bullet$ }};
			\node (v20) at (.5+2,1.5) {{ \color{blue}$\bullet$ }};
			\draw[blue,line width=1.1pt,<-] (.5,-0.5) -- (.5+1,0.5);
			\draw[blue,line width=1.1pt,<-] (.5+1,0.5) -- (.5+2,1.5);
			\node (v00) at (5.5,-2.5) { {{\color{blue}$\bullet$}}};
			\node (v10) at  (5.5+9,-1.5) {{ \color{blue}$\bullet$ }};
			
			\draw[blue,line width=1.1pt,<-] (5.5,-2.5) -- (5.5+9,-1.5);

			\node (v00) at (6.5,-6.5) { {{\color{blue}$\bullet$}}};
			\node (v10) at  (6.5+1,-5.5) {{ \color{blue}$\bullet$ }};
			
			\draw[blue,line width=1.1pt,<-] (6.5,-6.5) -- (6.5+1,-5.5);

			\node (v00) at (8.5,-10.5) { {{\color{blue}$\bullet$}}};
			\node (v10) at  (8.5+1,-9.5) {{ \color{blue}$\bullet$ }};
			
			\draw[blue,line width=1.1pt,<-] (8.5,-10.5) -- (8.5+1,-9.5);
			\draw[blue,line width=1.1pt,<-] (8.5+1,-9.5)-- (8.5+2,-8.5);
			\draw[blue,line width=1.1pt,<-] (8.5+2,-8.5)--(8.5+3,-7.5);
			\draw[blue,line width=1.1pt,<-] (8.5+3,-7.5)-- (12.5,-4.5); 
			\draw[blue,line width=1.1pt,<-] (6.5+1,-5.5) -- (12.5,-4.5);

			\node (v00) at (8.5+2,-8.5) { {{\color{blue}$\bullet$}}};
			\node (v10) at  (8.5+3,-7.5) {{ \color{blue}$\bullet$ }};
			
			\node (v00) at (12.5,-4.5) { {{\color{blue}$\bullet$}}};
			\node (v10) at  (13.5,-3.5) {{ \color{blue}$\bullet$ }};
			
			\draw[blue,line width=1.1pt,<-] (12.5,-4.5) -- (13.5,-3.5);
			\draw[blue,line width=1.1pt,<-] (13.5,-3.5) -- (5.5+9,-1.5);
			
			\node (v10) at  (18.5,-13.5) {{ \color{blue}$\bullet$ }};
			\node (v10) at  (19.5,-12.5) {{ \color{blue}$\bullet$ }};
			\draw[blue,line width=1.1pt,->] (19.5,-12.5) -- (18.5,-13.5);
		\end{tikzpicture};
		& &
		\begin{tikzpicture}[scale=0.3,shift={(0,2)}]
			\node (v00) at (.5,-0.5) { {{\color{blue}$\bullet$}}};
			\node (v10) at  (.5+1,0.5) {{ \color{blue}$\bullet$ }};
			\node (v20) at (.5+2,1.5) {{ \color{blue}$\bullet$ }};
			\node (v00) at (5.5,-2.5) { {{\color{blue}$\bullet$}}};
			\node (v10) at  (5.5+9,-1.5) {{ \color{blue}$\bullet$ }};
			\node (v00) at (6.5,-6.5) { {{\color{blue}$\bullet$}}};
			\node (v10) at  (6.5+1,-5.5) {{ \color{blue}$\bullet$ }};
			\node (v00) at (8.5,-10.5) { {{\color{blue}$\bullet$}}};
			\node (v10) at  (8.5+1,-9.5) {{ \color{blue}$\bullet$ }};
			\node (v00) at (8.5+2,-8.5) { {{\color{blue}$\bullet$}}};
			\node (v10) at  (8.5+3,-7.5) {{ \color{blue}$\bullet$ }};
			\node (v00) at (12.5,-4.5) { {{\color{blue}$\bullet$}}};
			\node (v10) at  (13.5,-3.5) {{ \color{blue}$\bullet$ }};
			\node (v10) at  (18.5,-13.5) {{ \color{blue}$\bullet$ }};
			\node (v10) at  (19.5,-12.5) {{ \color{blue}$\bullet$ }};
			\node (v00) at (.5,-0.5) { {{\color{blue}$\bullet$}}};
			\node (v10) at  (.5+1,0.5) {{ \color{blue}$\bullet$ }};
			\node (v20) at (.5+2,1.5) {{ \color{blue}$\bullet$ }};
			\draw[blue,line width=1.1pt,<-] (.5,-0.5) -- (.5+1,0.5);
			\draw[blue,line width=1.1pt,<-] (.5+1,0.5) -- (.5+2,1.5);
			\node (v00) at (5.5,-2.5) { {{\color{blue}$\bullet$}}};
			\node (v10) at  (5.5+9,-1.5) {{ \color{blue}$\bullet$ }};
			
			\draw[blue,line width=1.1pt,<-] (5.5,-2.5) -- (5.5+9,-1.5);

			\node (v00) at (6.5,-6.5) { {{\color{blue}$\bullet$}}};
			\node (v10) at  (6.5+1,-5.5) {{ \color{blue}$\bullet$ }};
			
			\draw[blue,line width=1.1pt,<-] (6.5,-6.5) -- (6.5+1,-5.5);

			\node (v00) at (8.5,-10.5) { {{\color{blue}$\bullet$}}};
			\node (v10) at  (8.5+1,-9.5) {{ \color{blue}$\bullet$ }};
			
			\draw[blue,line width=1.1pt,<-] (8.5,-10.5) -- (8.5+1,-9.5);
			\draw[blue,line width=1.1pt,<-] (8.5+1,-9.5)-- (8.5+2,-8.5);
			\draw[blue,line width=1.1pt,<-] (8.5+2,-8.5)--(8.5+3,-7.5);
			\draw[blue,line width=1.1pt,<-] (8.5+3,-7.5)-- (12.5,-4.5); 
			\draw[blue,line width=1.1pt,<-] (6.5+1,-5.5) -- (12.5,-4.5);

			\node (v00) at (8.5+2,-8.5) { {{\color{blue}$\bullet$}}};
			\node (v10) at  (8.5+3,-7.5) {{ \color{blue}$\bullet$ }};
			
			\node (v00) at (12.5,-4.5) { {{\color{blue}$\bullet$}}};
			\node (v10) at  (13.5,-3.5) {{ \color{blue}$\bullet$ }};
			
			\draw[blue,line width=1.1pt,<-] (12.5,-4.5) -- (13.5,-3.5);
			\draw[blue,line width=1.1pt,<-] (13.5,-3.5) -- (5.5+9,-1.5);
			
			\node (v10) at  (18.5,-13.5) {{ \color{blue}$\bullet$ }};
			\node (v10) at  (19.5,-12.5) {{ \color{blue}$\bullet$ }};
			\draw[blue,line width=1.1pt,->] (19.5,-12.5) -- (18.5,-13.5);
		\end{tikzpicture}.
		\\
		\ov{n}={{(3^5\, 5\, 9^2\, 13^{10}\, 16^2)}} & & \St_{{\small(3^5\, 5\, 9^2\, 13^{10}\, 16^2)}} 
	\end{array}
	\]
	
	Let us introduce special notations for the  row index and the column number of a cell of the Young diagram $\bbY_{\ov{n}}$:\nlabel{ni:horvrtcell}
	$$
	\hor((i,j)):= i, \quad \vrt((i,j)):=j.
	$$
	Since each row and each column contains at most one staircase corner, it follows that the row index and the column number define a linearization $\hor:\St_{\overline n}\hookrightarrow [1,n_m]$ and an anti-linearization $\vrt:\St_{\overline n}\hookrightarrow [1,m]$:
	$$
	(i,j)\prec (i',j') \in \St_{\ov{n}} \Leftrightarrow
	\begin{cases}
		\hor((i,j))< \hor((i',j')), \\
		\vrt((i,j))> \vrt ((i',j')).
	\end{cases}    
	$$
	
	The following lemma (proven in~\cite[Lemma 1.24]{FKM::Cauchy}) describes the fundamental properties of the poset $(\St_{\ov{n}},\prec)$ of staircase corners.
	\begin{lemma}
		\label{lem::DL::posets}
		\label{rem::ST::linearizations}
		\begin{enumerate}
			\item
			\label{item::DL::poset}
			The Hasse diagram of the poset $(\St_{\ov{n}},\preceq)$ forms a forest, with smaller elements positioned closer to the root of a tree. In other words, $\St_{\ov{n}}$ is an arborescent poset, where the map $\hor:\St_{\ov{n}}\hookrightarrow [1,n_m]$ provides a consistent linearization, and the map $\vrt:\St_{\ov{n}}\hookrightarrow [1,m]$ provides a consistent anti-linearization.
			\item
			\label{item::DL::linear}
			For any $s\in \St_{\ov{n}}$, the subset $\St_{\ov{n}}\{{\succeq s}\}$ of elements greater than or equal to $s$ forms an interval with respect to the linear order on $\vrt(\St_{\ov{n}})$:
			$$
			\forall s\in \St_{\ov{n}}\quad  \exists j_s\in [1,m] \  \colon \  \vrt(\St_{\ov{n}}\{\succeq s \}) = [j_s,\vrt(s)]. 
			$$
			Similarly, for the map $\hor$:
			$$
			\forall s\in \St_{\ov{n}} \quad \exists i_s\in [1,n_m] \  \colon \  \hor(\St_{\ov{n}}\{\succeq s \}) = [\hor(s),i_s]. 
			$$
			\item \label{item::DL::uncom}
			For any two incomparable elements $s,t\in \St_{\ov{n}}$, the vertical and horizontal comparisons are equivalent:
			$$\vrt(s)<\vrt(t) \in [1, m] \phantom{\frac{1}{2}} {\Leftrightarrow} \phantom{\frac{1}{2}} \hor(s)<\hor(t) \in [1, n_m].
			$$
		\end{enumerate}
	\end{lemma}
	
	\begin{example}
		In the following pictorial description, we consider the Young diagram $\bbY_{(2\ 3^3\ 5\ 7^4)}$. The bullets denote the elements of the arborescent poset $\St_{\ov{n}}$ and edges represent covering relations in this poset.
		$$
		
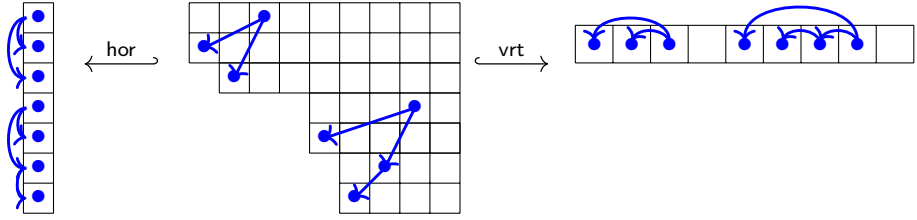
\begin{tikzcd}
			\begin{tikzpicture}[scale=0.4]
				\draw[step=1cm] (0,2) grid (1,-5);    
				\node (v1) at (.5,.3) {{ \color{blue}$\bullet$ }};
				\node (v2) at (.5,-.7) {{ \color{blue}$\bullet$ }};
				\node (v3) at (.5,1.3) {{ \color{blue}$\bullet$ }};
				\node (v5) at (.5,-2.7) {{ \color{blue}$\bullet$ }};
				\node (v6) at (.5,-4.7) {{ \color{blue}$\bullet$ }};
				\node (v7) at (.5,-3.7) {{ \color{blue}$\bullet$ }};
				\node (v8) at (.5,-1.7) {{ \color{blue}$\bullet$ }};
				\draw  (.1,1.5) edge[blue,line width=1.1pt,->,bend right=90] (.1,-.5);
				\draw (.1,1.5) edge[blue,line width=1.1pt,->,bend right=70] 
				(.1,.5);
				\draw (.1,-1.5) edge[blue,line width=1.1pt,->,bend right=70] (.1,-2.5);
				\draw (.1,-1.5) edge[blue,line width=1.1pt,->,bend right=90] (.1,-3.5);
				\draw (.1,-3.5) edge[blue,line width=1.1pt,->,bend right=90]  (.1,-4.5);
			\end{tikzpicture}
			\arrow[r,hookleftarrow,"\hor"]
			&    
			\begin{tikzpicture}[scale=0.4]
				\draw[step=1cm] (0,2) grid (9,0);
				\draw[step=1cm] (1,0) grid (9,-1);
				\draw[step=1cm] (4,-1) grid (9,-3);
				\draw[step=1cm] (5,-2) grid (9,-5);
				\node (v1) at (.5,.3) {{ \color{blue}$\bullet$ }};
				\node (v2) at (1.5,-.7) {{ \color{blue}$\bullet$ }};
				\node (v3) at (2.5,1.3) {{ \color{blue}$\bullet$ }};
				\node (v5) at (4.5,-2.7) {{ \color{blue}$\bullet$ }};
				\node (v6) at (5.5,-4.7) {{ \color{blue}$\bullet$ }};
				\node (v7) at (6.5,-3.7) {{ \color{blue}$\bullet$ }};
				\node (v8) at (7.5,-1.7) {{ \color{blue}$\bullet$ }};
				\draw [blue,line width=1.1pt,->] (2.5,1.5) -- (.5,.5);
				\draw [blue,line width=1.1pt,->] (2.5,1.5) -- (1.5,-.5);
				\draw [blue,line width=1.1pt,->] (7.5,-1.5) -- (4.5,-2.5);
				\draw [blue,line width=1.1pt,->] (7.5,-1.5) -- (6.5,-3.5);
				\draw [blue,line width=1.1pt,->] (6.5,-3.5) -- (5.5,-4.5);
			\end{tikzpicture}
			\arrow[r,hookrightarrow,"\vrt"]
			&
			\begin{tikzpicture}[scale=0.5]
				\draw[step=1cm] (0,1) grid (9,0);
				\node (v1) at (.5,.3) {{ \color{blue}$\bullet$ }};
				\node (v2) at (1.5,.3) {{ \color{blue}$\bullet$ }};
				\node (v3) at (2.5,.3) {{ \color{blue}$\bullet$ }};
				\node (v5) at (4.5,.3) {{ \color{blue}$\bullet$ }};
				\node (v6) at (5.5,.3) {{ \color{blue}$\bullet$ }};
				\node (v7) at (6.5,.3) {{ \color{blue}$\bullet$ }};
				\node (v8) at (7.5,.3) {{ \color{blue}$\bullet$ }};
				\draw  (2.5,.6) edge[blue,line width=1.1pt,->,bend right=90] (.5,.6);
				\draw (2.5,.6) edge[blue,line width=1.1pt,->,bend right=60] (1.5,.6);
				\draw (7.5,.6) edge[blue,line width=1.1pt,->,bend right=90] (4.5,.6);
				\draw (7.5,.6) edge[blue,line width=1.1pt,->,bend right=60] (6.5,.6);
				\draw (6.5,.6) edge[blue,line width=1.1pt,->,bend right=60]  (5.5,.6);
			\end{tikzpicture}
		\end{tikzcd}
		$$    
	\end{example}
	
	The following proposition describes the inverse procedure. Specifically, we seek a necessary and sufficient condition for a poset equipped with a consistent (anti)linearization to correspond to a staircase shape $\bbY_{\ov{n}}$.
	
	\begin{proposition}
		\label{prp::poset->partition}
		For any given arborescent poset $(\mathsf{S},\prec)$ with a consistent order-preserving linearization 
		$h:(\mathsf{S},\prec)\to ([1, n_m],<)$ satisfying the following properties: 
		$$
		\begin{cases}
			\forall s\in \mathsf{S}, \ \exists a(s), b(s)  \text{ such that } h(\mathsf{S}_{\succeq s})=[a(s), b(s)] \subset [1, n_m], \\
			h^{-1}(n_m)\neq \emptyset,
		\end{cases}
		$$
		there exists a partition $\ov{n}:=(n_1\le\ldots\le n_m)$ and an order-preserving bijection $\psi:\mathsf{S}\stackrel{\simeq}{\rightarrow} \St_{\ov{n}}$, such that $h = \hor\circ\psi$.
		
		Similarly, for any given arborescent poset $(\mathsf{S},\prec)$ with a consistent antilinear injection $v:\mathsf{S}\to [1, n_m]$
		that maps the subsets $\mathsf{S}_{\succeq s}$ to subintervals of $[1, m]$, there exists a partition $\ov{n}:=(n_1\le\ldots\le n_m)$ and an order-preserving bijection $\psi:\mathsf{S}\stackrel{\simeq}{\rightarrow} \St_{\ov{n}}$, such that $v = \vrt\circ\psi$ (provided that $v^{-1}(1)\neq\emptyset$).
	\end{proposition}
	
	\begin{proof}
		It suffices to prove the second statement since the first follows by transposing the Young diagram $\bbY_{\ov{n}}$. 
		
		The partition $\ov{n}$ is constructed as follows:
		$$
		\begin{cases}
			\text{If }\exists s \text{ such that } v(s)=k, \text{ then } n_k:=\#\{t\in \mathsf{S} \colon v(t)< k \} + \#\{t\in \mathsf{S} \colon t\preceq s\}; \\
			\text{If } v^{-1}(k)=\emptyset, \text{ then } n_k:=\#\{t\in \mathsf{S} \colon v(t)< k \}.
		\end{cases}
		$$
	\end{proof}
	
	Note that every arborescent poset $\sS$ with an isomorphic (anti)linearization admits a realization as a subset of staircase corners for an appropriate $\ov{n}$. However, if the (anti)linearization $v:\sS\hookrightarrow [1,m]$ is not surjective, it is necessary that each connected component of the Hasse graph of $\sS$ maps onto a closed interval of $[1,m]$.

	\subsection{DL-dense arrays}
	\label{sec::DL_dense}
	
	The $DL$-dense arrays defined below play a central role in this paper as well as in our previous work~\cite{FKM::Cauchy}. 
	First, we recall the definition, introduce the Bruhat partial order on this set and establish several key properties of this partial order in Section~\ref{sec::Bruhat::DL-dense}. 
	
	\begin{definition}\label{def:arrays}
		\begin{itemize}
			\item 
			A map $A:\bbY_{\ov{n}}\to \bZ_{\geq 0}$ that assigns a nonnegative integer $A_{i,j}$ to each cell
			$(i,j)\in \bbY_{\ov{n}}$ is called an \emph{array} of shape $\bbY_{\ov{n}}$.
			\item The total sum $|A|:=\sum_{(i,j)\in\bbY_{\ov{n}}} A_{i,j}$ is called the \emph{degree} of the array $A$.
			\item The sum of the entries in each row is called the \emph{horizontal weight} of the array:
			$$\hor(A):=\left(\sum_{j=1}^{m}A_{1j}, \ldots, \sum_{j=1}^{m} A_{n_mj} \right) \in \bZ^{n_m}.$$
			\item 
			The \emph{vertical weight} of an array is defined as
			$$\vrt(A):=\left(\sum_{i=1}^{n_1} A_{i1}, \ldots, \sum_{i=1}^{n_m}A_{im} \right)\in \bZ^{m},$$    
			i.e., it is the collection of column sums.
		\end{itemize}      
	\end{definition}
	
	\begin{definition}
		\label{def::DL::dense}
		An array $A$ of shape $\bbY_{\ov{n}}$ is called \emph{$DL$-dense} if 
		$$
		\begin{cases}
			(i,j)\notin \St_{\ov{n}} \ \Rightarrow \ A_{i,j}=0; \\
			(i,j)\prec (i',j')\in \St_{\ov{n}} \ \Rightarrow \ A_{i,j}\leq A_{i',j'}.
		\end{cases}
		$$
	\end{definition}
	
	\begin{remark}
		The set of $DL$-dense arrays is in one-to-one correspondence with the set of order-preserving $\bZ_{\geq0}$-valued functions on the poset $(\St_{\ov{n}},\prec)$.
	\end{remark}
	
	\begin{dfn}
		\label{def::DL::dense::set}
		For each partition $\lambda=\{\lambda_1\geq\ldots\geq \lambda_k\geq 0\}$ (whose length is at most $\#\St_{\ov{n}}$), we denote by $\DL_{\ov{n}}({\lambda})$ the set of $DL$-dense arrays $A$ such that the multiset $\{A_{s}\ \colon\ s\in\St_{\ov{n}}\}$ coincides
		with the multiset $\{\la_i\}_{i=1}^k$.
	\end{dfn}
	
	\begin{remark}
		\label{rem::DL::decomp}
		We have the following decomposition of the set of all $DL$-dense arrays of shape $\bbY_{\ov{n}}$ with respect to degree and partition:	
		$$\DL_{\ov{n}}=\bigsqcup_{N\geq 0}\bigsqcup_{\substack{{\lambda\vdash N}\\ {l(\lambda)\leq \#\St_{\ov{n}}}}}\DL_{\ov{n}}(\lambda).$$
	\end{remark}
	
	\begin{example}
		Here is an example of a $DL$-dense array of shape $\bbY_{(2,3,3,4)}$:
		$$
		
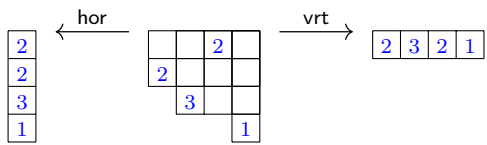
\begin{tikzcd}
			\gridV{2}{2}{3}{1}
			& \gridDL{2}{3}2{1}
			\arrow[l,"\hor"'] \arrow[r,"\vrt"]
			& \gridH2321
		\end{tikzcd}
		$$    
	\end{example}
	
	\begin{remark}
		The term \emph{array} was introduced by Danilov and Koshevoy for rectangular shapes (see, e.g.,~\cite{DK1, DK2}). They consider arrays as collections of balls placed in the cells of the diagram \( \bbY_{\ov{n}} \). One of the key features of their construction is the elegant bi-crystal structure, where balls can move up and down but only within consecutive columns or rows. The $D$-dense property means that moving a ball downward is not allowed, while the $L$-dense property prohibits moving any ball to the left.
		
		The main difference between our approach and the one proposed in~\cite{DK1} is that the crystal structure alone is insufficient for our purposes. Instead, we must consider the action of all roots, not just the simple ones, as in the case of crystals.
		
		Nevertheless, it is possible to define $D$-dense and $L$-dense arrays separately, describe the combinatorial densification maps, and relate this process to the bubble-sort map discussed in~\S\ref{sec::Bubble-sort}. We will explore these ideas in a separate paper.
	\end{remark}

	\begin{proposition}
		\label{prp::DL::hor::vrt}	
		The horizontal (vertical) weight of a $DL$-dense array is an $\St_{\ov{n}}$-dominant composition of length $n_m$ (resp. $m$), and the maps $\hor$ and $\vrt$ define bijections:
		$$
		\begin{tikzcd}
			\bS_{n_m}\lambda \arrow[r,hookleftarrow] &  \bD_{\St_{\ov{n}}}^{\hor}(\lambda) & \arrow[l,"\simeq"',"\hor"] \DL_{\ov{n}}(\lambda) \arrow[r,"\simeq", "\vrt"'] & \bD_{\St_{\ov{n}}}^{\vrt}(\lambda) \arrow[r,hook] & \bS_{m}\lambda.
		\end{tikzcd}	
		$$
		Here, $\bD_{\St_{\ov{n}}}^{\hor}(\lambda)$ denotes the set of $\St_{\ov{n}}$-dominant compositions associated with the arborescent poset $\St_{\ov{n}}$ equipped with the consistent linearization $\hor:\St_{\ov{n}}\hookrightarrow [1, n_m]$.
		Similarly, $\bD_{\St_{\ov{n}}}^{\vrt}(\lambda)$ consists of $\St_{\ov{n}}$-dominant compositions associated with the consistent antilinearization $\vrt:\St_{\ov{n}}\hookrightarrow [1, m]$.
	\end{proposition}
	
	\begin{proof}
		Since each row and each column of the Young diagram contains at most one staircase corner, and as noted in 
		Remark~\ref{rem::ST::linearizations}, we know that $\vrt$ defines an order-reversing dominant antilinearization of $\St_{\ov{n}}$, while $\hor$ defines an order-preserving dominant linearization.
		On the other hand, each row and each column of the Young diagram contains at most one staircase corner, which implies that a $DL$-dense array $A$ is uniquely determined by its horizontal (or vertical) weight. Furthermore, from the definition, we see that $A$ is $DL$-dense if and only if $\hor(A)$ is $\hor(\St_{\ov{n}})$-dominant (respectively, $\vrt(A)$ is $\vrt(\St_{\ov{n}})$-dominant).
	\end{proof}

	\subsection{Bruhat order on $\DL_{\ov{n}}(\lambda)$}
	\label{sec::Bruhat::DL-dense}
	
	Let $\bbY_{\ov{n}}$ be a given staircase shape, let $\St_{\ov{n}}$ be the corresponding poset of staircase corners (defined in~\S\ref{sec::staircase::corners}), and let $\lambda$ be a partition whose length does not exceed the size of $\St_{\ov{n}}$. 
	The goal of this section is to define and describe the basic properties of the Bruhat partial order on the subset $\DL_{\ov{n}}(\lambda)$ of $DL$-dense arrays, where the multiset of nonzero elements coincides with $\lambda$.
	
	First, recalling Remark~\ref{rem::row::deletion}, we can restrict our attention to the case where each row and each column of $\bbY_{\ov{n}}$ contains a staircase corner. This means that $n_m = m = \#\St_{\ov{n}}$, and we will work under this assumption for the remainder of Section~\ref{sec::Bruhat::DL-dense}.
	
	\begin{definition}
		\label{def::min::DL::dis}
		A pair of staircase corners $(i j)\ttt(i'j')\in \St_{\ov{n}}\times\St_{\ov{n}}$ is called 
		a \emph{minimal $DL$-disorder for a $DL$-dense array $A$} iff 
		\begin{itemize}
			\item staircase corners $(ij)$ and $(i'j')$ are uncomparable in $\St_{\ov{n}}$:
			$$ (i <i') \text{ and } (j<j'),$$
			\item $(ij)\ttt(i'j')$ is a disorder: $A_{ij}>A_{i'j'},$
			\item 
			the following implications hold for all staircase corners $(kl)\in\St_{\ov{n}}$:
			\begin{gather}
				\label{eq::DL::disorder:1a}
				\begin{array}{ccc}
					\left[ 
					\begin{array}{ccc}
						(k>i) & \& & (l<j) \\
						(k>i') & \& & (l<j')
					\end{array}
					\right. & \Rightarrow & A_{kl}\geq A_{ij}, 
				\end{array}
				\\
				\label{eq::DL::disorder:1b}
				\begin{array}{ccc}
					\left[ 
					\begin{array}{ccc}
						(k<i) & \& & (l>j) \\
						(k<i') & \& & (l>j')
					\end{array}
					\right. & \Rightarrow & A_{kl}\leq A_{i'j'}.
				\end{array} \\
				\label{eq::DL::disorder:2}
				(i<k<i') \ \& \ (j<l<j') \ \Rightarrow \ A_{kl}\notin [A_{i'j'},A_{ij}].
			\end{gather}
		\end{itemize}
	\end{definition}
	Let us give a pictorial description of the necessary inequalities for a minimal $DL$-disorder:
	\begin{equation*}
		\begin{array}{l}
			\bullet\text{\small A minimal disorder $(ij)\ttt(i'j')$ is drawn in rose,} \\   
			\bullet\text{\small Assumptions~\eqref{eq::DL::disorder:1a} are drawn in green,}\\
			\bullet\text{\small Assumptions~\eqref{eq::DL::disorder:1b} are blue,}\\
			\bullet\text{\small Assumption~\eqref{eq::DL::disorder:2} is drawn in orange.}
		\end{array}
		\
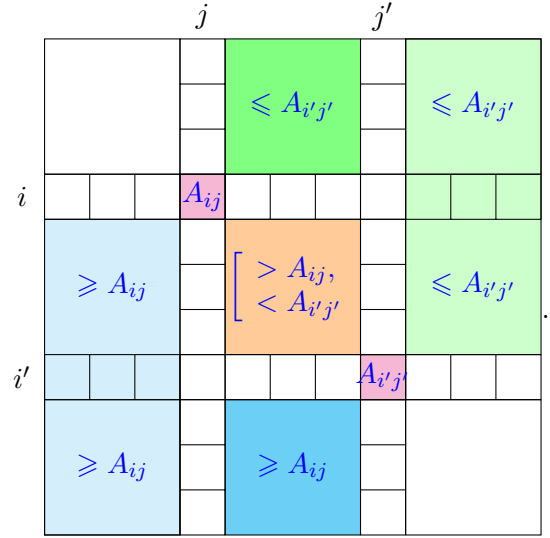
		\begin{tikzpicture}[scale=0.6]
			\filldraw[fill=cyan!15] (-5,-5) rectangle (-2,2);
			\filldraw[fill=cyan!50] (-1,-5) rectangle (2,-2);
			\filldraw[fill=green!50] (-1,3) rectangle (2,6);
			\filldraw[fill=green!20] (3,-1) rectangle (6,6);
			\filldraw[fill=orange!40] (-1,-1) rectangle (2,2);
			\filldraw[fill=magenta!35] (-2,2) rectangle (-1,3);
			\filldraw[fill=magenta!35] (2,-2) rectangle (3,-1);
			\draw[step=1cm] (-2,-5) grid (-1,6);
			\draw[step=1cm] (-5,-2) grid (6,-1);
			\draw[step=1cm] (-5,2) grid (6,3);
			\draw[step=1cm] (2,-5) grid (3,6);
			\draw (-5,3) rectangle (-2,6);
			\draw (3,3) rectangle (6,6);
			\draw (3,-5) rectangle (6,-2);
			\draw (-5,-5) rectangle (-2,-2);
			\node (a) at (-1.5,2.5) {{\cbl{\small {$A_{ij}$}}}};
			\node (b) at (2.5,-1.5) {{\cbl{\small {$A_{i'j'}$}}}};
			\node (u21) at (-3.5,0.5) {{\cbl{\small {$\geq A_{ij}$}}}};
			\node (u31) at (-3.5,-3.5) {{\cbl{\small {$\geq A_{ij}$}}}};
			\node (u12) at (0.5,4.5) {{\cbl{\small {$\leq A_{i'j'}$}}}};
			\node (u12) at (4.5,4.5) {{\cbl{\small {$\leq A_{i'j'}$}}}};
			\node (u32) at (0.5,-3.5) {{\cbl{\small {$\geq A_{ij}$}}}};
			\node (u23) at (4.5,0.5) {{\cbl{\small {$\leq A_{i'j'}$}}}};
			\node (u22) at (0.5,0.5) {{\cbl{\small 
						{$\left[\begin{array}{l} >A_{ij}, \\ <A_{i'j'}\end{array}\right.$}}}};
			\node (r1) at (-5.5,2.5) {{\small {$i$}}};
			\node (r2) at (-5.5,-1.5) {{\small {$i'$}}};
			\node (c1) at (-1.5,6.5) {{\small {$j$}}};
			\node (c2) at (2.5,6.5) {{\small {$j'$}}};
		\end{tikzpicture}.
	\end{equation*}
	
	\begin{remark}
		The conditions that must be verified for $(ij)\ttt(i'j')$ to be a minimal $DL$-disorder are only those represented in dark green $(k<i)\&(j<l<j')$ and dark blue $(k>i')\&(j<l<j')$, as well as the inequalities in the orange square $(i<k<i')\&(j<l<j')$. The remaining blue and green inequalities follow directly from the $DL$-dense property of $A$.
		However, we suggest keeping all the assumptions to maintain the symmetry of both the conditions and the diagram.
	\end{remark}
	
	Here is a pair of examples of $DL$-disorders for a $DL$-dense array from $\DL_{(2 3^2 4)}(2^2 1^2)$:
	$$
	\begin{array}{cc}
		{
			\begin{tikzpicture}[scale=0.4]
				\draw[step=1cm] (0,0) grid (4,-2);
				\draw[step=1cm] (1,0) grid (4,-3);
				\draw[step=1cm] (3,0) grid (4,-4);
				\node (v21) at (0.5,-1.5) {{\cbl{\small {2}}}};
				\node[fill=magenta!35] (v32) at (1.5,-2.5) {{\cbl{\small {2}}}};
				\node (v13) at (2.5,-.5) {{\cbl{\small {1}}}};
				\node[fill=magenta!35] (v44) at (3.5,-3.5) {{\cbl{\small {1}}}};
			\end{tikzpicture}
		}
		&
		{
			\begin{tikzpicture}[scale=0.4]
				\draw[step=1cm] (0,0) grid (4,-2);
				\draw[step=1cm] (1,0) grid (4,-3);
				\draw[step=1cm] (3,0) grid (4,-4);
				\node[fill=brown!30] (v21) at (0.5,-1.5) {{\cbl{\small {2}}}};
				\node[fill=gray] (v32) at (1.5,-2.5) {{\cbl{\small {2}}}};
				\node (v13) at (2.5,-.5) {{\cbl{\small {1}}}};
				\node[fill=brown!30] (v44) at (3.5,-3.5) {{\cbl{\small {1}}}};
			\end{tikzpicture}
		} 
		\\
		\text{(32)-(44) is a minimal $DL$-disorder,} &
		\text{(21)-(44) is a non-minimal $DL$-disorder.} 
	\end{array}
	$$

	\begin{lemma}
		\label{lem::dis::vrt::hor}	
		The following conditions are equivalent for any $DL$-dense array $A$:
		\begin{itemize}[itemsep=0pt]
			\item $(ij)\ttt(i'j')$ is a minimal $DL$-disorder for $A$;
			\item $(i i')$ is a minimal $\hor(\St_{\ov{n}})$-disorder for $\hor(A)$;
			\item $(j j')$ is a minimal $\vrt(\St_{\ov{n}})$-disorder for $\vrt(A)$.	
		\end{itemize}
	\end{lemma}
	\begin{proof}
		This follows from a direct comparison of inequalities~\eqref{eq::DL::disorder:1a}--\eqref{eq::DL::disorder:2} and Definition~\ref{def::disorder}, which are visualized in the Figure above.
	\end{proof}

	\begin{corollary}
		\label{cor::hor::vrt::Bruhat}
		For any two $DL$-dense arrays $A$ and $B$, the following inequalities are equivalent:
		$$
		\hor(A) \prec_{\Bruhat} \hor(B) \ \Leftrightarrow  \  \vrt(A) \prec_{\Bruhat} \vrt(B).
		$$
		In other words, the horizontal and vertical Bruhat orders on $\DL_{\ov{n}}(\lambda)$ coincide.		
	\end{corollary}
	\begin{proof}
		Recall that $\vrt$ defines an isomorphism between $\DL_{\ov{n}}(\lambda)$ and $\bD_{\St_{\ov{n}}}^{\vrt}(\lambda)$. The covering relations in the latter poset are given by minimal $\vrt(\St_{\ov{n}})$-disorders for $\St_{\ov{n}}$-dominant weights (Theorem~\ref{thm::dominant::edge}).
		Similarly, the covering relations for the horizontal weights are also given by minimal $\hor(\St_{\ov{n}})$-disorders, which are known to be the same covering relations thanks to Lemma~\ref{lem::dis::vrt::hor}.
	\end{proof}	
	
	In particular, it follows that the set $\DL_{\ov{n}}(\lambda)$ admits a canonical Bruhat partial order, and the corresponding poset is bounded, graded, subthin, and $\EL$-shellable if $\lambda$ is regular.
	
	\subsection{Double Demazure Distributive Lattice}
	\label{sec::DL::Distr}
	The combinatorial structures defined in the preceding sections allow us to construct a distributive lattice of Demazure submodules. This will later be used in the description of the main subject of this paper -- Howe duality for staircase matrices.
	
	\begin{theorem}
		\label{thm::D::Dop::distrib}	
		For each partition $\lambda$ of length at most $\#\St_{\ov{n}}$, 
		the tensor product of $\fb_{n}\ttt\fb_{m}$-Demazure submodules 
		$$\{D_{\hor(A)}\otimes D_{\vrt(B)}^{\op} \ \colon A,B\in\DL_{\ov{n}}(\lambda)\}$$ 
		in the product of two integrable representations $V_\lambda\otimes V_{\lambda}^{\op}$ forms the set of $\vee$-irreducible elements of a distributive lattice, which we denote by $\Lattice_{D\times D^{\op}}^{{\ov{n}}}(\lambda)$.
	\end{theorem}
	\begin{proof}
		Thanks to Proposition~\ref{prp::DL::hor::vrt}, we know that the poset $\DL_{\ov{n}}(\lambda)$ is isomorphic to the poset of $\St_{\ov{n}}$-dominant weights $\bD_{\St_{\ov{n}}}^{\hor}(\lambda)$. Consequently, by Corollary~\ref{cor::distr::dominant}, the Demazure submodules $\{D_{\hor(A)} \colon A\in\DL_{\ov{n}}(\lambda)\}$ form the set of $\vee$-irreducible elements of a distributive lattice $\Lattice_{D}^{\St_{\ov{n}}}(\lambda)$ of left $\fb_{n}$-submodules in the $\gl_n$-module $V_{\lambda}$.
		Similarly, by Corollary~\ref{cor::distr::dominant::op}, the opposite Demazure submodules $\{D_{\vrt(B)}^{\op}\colon B\in\DL_{\ov{n}}(\lambda)\}$ form the set of $\vee$-irreducible elements of the distributive lattice $\Lattice_{D^{\op}}^{{\ov{n}}}(\lambda)$ of right $\fb_m$-submodules in the right $\gl_m$-module $V_{\lambda}^{\op}$.
		
		A lattice of vector spaces is distributive if and only if it has a common basis. 
		Therefore, if the distributive lattice $\Lattice$ has a common basis $E$ and the distributive lattice $\Lattice'$ has a common basis $E'$, then the Cartesian product $E\times E'$ defines a common basis for the tensor product of lattices $\Lattice\otimes \Lattice'$. This ensures that the tensor product remains distributive. Moreover, the tensor product of $\vee$-irreducible elements is also $\vee$-irreducible.
	\end{proof}
	
	Recall that for each $\vee$-irreducible subspace $D$ in a distributive lattice $\Lattice$ of vector spaces, one assigns a minimal subquotient ${\KK}_D$ (see~\eqref{eq::min::subquot}). By direct inspection, we observe that the minimal subquotients in the tensor product of distributive lattices correspond to the tensor product of minimal subquotients. 
	
	Consequently, the minimal subquotients in the lattice $\Lattice_{D\times D^{\op}}^{{\ov{n}}}(\lambda)$ are indexed by pairs of $DL$-dense arrays and are isomorphic to the product of the generalized van der Kallen modules introduced in~\eqref{eq::vdK::general}:\nlabel{ni:KAB}
	$$
	{\KK}_{A,B}^{\ov{n}}:= {\KK}_{\St_{\ov{n}},\hor(A)} \otimes {\KK}_{\St_{\ov{n}},\vrt(B)}^{\op}.
	$$
	In particular, its $\fh_{n}\ttt\fh_{m}$-character has the following description based on the formulas~\eqref{eq::gen::vdK::character} and~\eqref{eq::vdK::op::char}:
	\begin{multline}
		\label{eq::vdK::D:Dop}
		\ch_{\fh_{n}\ttt\fh_{m}}({\KK}_{A,B}^{\ov{n}}) = \ch_{\fh_n}\left({\KK}_{\St_{\ov{n}},\hor(A)}\right) \cdot \ch_{\fh_m}\left({\KK}_{\St_{\ov{n}},\vrt(B)}^{\op}\right) = \\
		=\left(\sum_{C\preceq A}\mu^{\DL_{\ov{n}}(\lambda)}_{C,A} \kappa_{\hor(C)}(x) \right)
		\left( \sum_{C\succeq B}\mu^{\DL_{\ov{n}}(\lambda)}_{B,C} \kappa^{\vrt(C)}(y) \right).
	\end{multline}
	Note that the summation in the first term runs over all $DL$-dense arrays less than or equal to $A$, while the summation in the second term ranges over $DL$-dense arrays greater than or equal to $B$.

	\section{Howe Duality for Staircase Matrices}
	\label{sec::Howe::duality}
	
	This paper is motivated by the combinatorics of various Cauchy-type identities for staircase matrices $\Mat_{\ov{n}}$ that we introduced in~\cite{FKM::Cauchy}. All these identities are based on the description of the associated graded components of different filtrations of $S^N(\Mat_{\ov{n}})$. These filtrations originate from the highest weight category structures of left $\fb_{n_m}$-modules and right $\fb_{m}$-modules. In this paper, we propose considering weaker filtrations, which are still sufficient for computing the characters.
	
	\subsection{The Bi-Module $\Mat_{\ov{n}}$ of Staircase Matrices}
	\label{sec::bimodule}
	
	In this section, we fix a collection of integers $\overline{n}:=(n_1\leq \ldots \leq n_m)$ with $n_1>0$, and we denote the largest column $n_m$ by $n$.
	In particular, the Young diagram $\bbY_{\overline{n}}$ is a subdiagram of a rectangular Young diagram $\bbY_{n^m}$. The space of staircase matrices $\Mat_{\overline{n}}$ consists of linear functions $A:\bbY_{\overline{n}}\to\Bbbk$. In other words, $\Mat_{\overline{n}}$ is the subspace of
	rectangular $n\times m$ matrices whose entries vanish outside $\bbY_{\overline{n}}$. 
	
	The space $\Mat_{\overline{n}}$ is acted upon from the left by the Borel subalgebra $\fb_n$ and from the right by the Borel subalgebra $\fb_m$ of upper triangular matrices (via left and right multiplication).\nlabel{ni:bimod}
	These actions commute, yielding a bi-module structure. In what follows, we are particularly interested in the bi-module $S(\Mat_{\overline{n}})$, the symmetric algebra of $\Mat_{\ov{n}}$.
	
	For each $(i,j)\in\Mat_{\ov{n}}$, we denote by $v_{ij}$ the matrix unit that sends the cell $(i,j)$ to $1$ and all other elements to zero. The set $\{v_{ij}\colon (i,j)\in\bbY_{\ov{n}}\}$ forms a basis of $\Mat_{\ov{n}}$.\nlabel{ni:vij}
	To each array $A$ of shape $\bbY_{\ov{n}}$, we associate the monomial\nlabel{ni:vA}
	$$v^{A}:=\prod_{i,j} v_{ij}^{A_{ij}}.$$ 
	The set $\{v^A\}$, where $A$ ranges over all arrays of degree $N$, forms a basis of the $N$'th symmetric power $S^{N}(\Mat_{\overline{n}})$.
	
	\begin{remark}
		The left action of the matrix unit $E_{ij}\subset \fb_n$ ($i\leq j\leq n$) and the right action of the matrix unit $E_{ij}\subset \fb_m$ on matrix units $\{v_{ij}\colon (ij)\in\bbY_{\ov{n}}\}$ and arrays can be summarized as follows:
		\begin{gather*}
			E_{i j} v_{a b} = \delta_{j,a} v_{ib},\qquad v_{ab} E_{i j} = \delta_{b,i} v_{aj}; \\
			E_{ii}v^A=(\hor(A))_i v^A,\quad v^A E_{ii}=(\vrt(A))_i v^A.
		\end{gather*}
	\end{remark}
	
	\begin{definition}        
		\label{def::weightsub}
		The left $\fh_{n}$-weight  $\nu$-subspace 
		${}_\nu S(\Mat_{\overline{n}})$ of $S(\Mat_{\overline{n}})$  is spanned by the elements $v^A$ such that $\hor(A)=\nu$.
		Similarly, the right $\fh_{m}$-weight  ${\overline d}$-subspace $ S(\Mat_{\overline{n}})_{\overline d}$ of $S(\Mat_{\overline{n}})$  is spanned by the elements $v^A$ such that $\vrt(A)={\overline d}$.
		The notation  ${_\nu S}(\Mat_{\overline{n}})_{\overline d}$ is used for the left-right weight $(\nu,{\overline d})$ subspace of $S(\Mat_{\overline{n}})$.
	\end{definition}
	
	Note that the left $\fb_{n}$-action on the weight subspace ${_\nu S}(\Mat_{\overline{n}})_{\overline d}$ increases the weight $\nu$ while keeping $\ov{d}$ unchanged, whereas the right $\fb_m$-action decreases the weight $\ov{d}$ while keeping $\nu$ unchanged.
	
	The following proposition was proved in~\cite[Prp.4.3]{FKM::Cauchy} based on a direct inspection of the $\fb_{n}\ttt\fb_m$-action:
	\begin{proposition}\label{prop:BimoduleGenerators}
		The monomials $v^A$ where $A$ belongs to $\bigcup_{\lambda\vdash N}\DL_{\ov{n}}(\lambda)$, the set of $DL$-dense arrays of total weight $N$, generate the $\fb_{n_m}\de\fb_{m}$-bimodule $S^N(\Mat_{\overline{n}})$.
	\end{proposition}
	
	\begin{example}
		\label{ex::rectangular}	
		For a rectangular Young diagram $\bbY_{n^m}$, the set of staircase corners $\St_{n^m}$ forms a totally ordered set of size $\min(m,n)$ (see Picture~\eqref{pic::rectangular::corners}). Consequently, the set $\DL_{n^m}(\lambda)$ consists of a single element, which we denote by 
		$$
		A_\lambda:= 
		{\begin{tikzpicture}[scale=0.5]
				\draw[step=1cm] (0,-3) grid (7,2);
				\node (v00) at (.5,-2.5) { {{\small $\lambda_1$}}};
				\node (v10) at  (.5+1,-1.5) {{\small $\lambda_2$ }};
				\node (v20) at (.5+2,-0.5) {{ \small $\lambda_3$}};
				\node (v30) at (.5+3,0.5) {{ \small $\lambda_4$}};
				\node (v40) at (.5+4,1.5) {{ \small $\lambda_5$}};		
			\end{tikzpicture},
		}
		$$ 
		that has nonzero elements on the secondary diagonal starting from the bottom-left corner. 	
	\end{example}
	
	\subsection{Howe Duality and Cauchy-Type Identities}
	\label{sec::Howe::staircase}
	
	The famous Howe duality introduced by R.Howe in \cite{Ho} (see also \cite{Goodman::Wallach}) states the following isomorphism of $\mgl_{n}\ttt\mgl_m$-bimodules:
	\begin{equation}
		\label{eq::Howe::duality}
		S^N(\Mat_{n\times m}) \simeq 
		\bigoplus_{\substack{{\lambda\vdash N} \\ {l(\lambda)\leq\min(n,m)}}} V_{\lambda}^{\mgl_{n}} \otimes (V^{\mgl_m}_\lambda)^{\op}.
	\end{equation}
	Here, $V^{\gl_n}_\lambda$ denotes the integrable $\mgl_n$-representation with highest weight $\lambda$.\nlabel{ni:Vgln}
	We are interested in the action of the Borel subalgebras $\fb_n\ttt\fb_m$ of upper-triangular matrices rather than the full matrix Lie algebras.
	In particular, the maximal Demazure $\fb$-submodules of the Borel subalgebra coincide with the corresponding irreducible finite-dimensional $\gl_n$-module:
	$$
	V^{\gl_n}_\lambda\simeq D_{\lambda_{-}}, \quad (V^{\gl_n}_{\lambda})^{\op} \simeq D_{\lambda_{+}}^{\op}.
	$$
	Moreover, the action of upper-triangular matrices defines a standard partial order on the set of dominant weights (partitions):
	
	\begin{definition}
		\label{def::partition::order}
		We say that partitions $\la\ge \nu$ (with $\lambda,\nu\vdash N$) if and only if
		$$\forall r=1,\ldots, n \ \ \la_1 + \dots + \la_r\ge \nu_1 + \dots + \nu_r.$$ 
	\end{definition}
	
	Although the explicit formulas for the generators of the summands on the right-hand side of~\eqref{eq::Howe::duality} are somewhat intricate, it is straightforward to identify representatives with respect to the following left (${}^{\lambda}\calF$) and right ($\calF^{\lambda}$) filtrations indexed by partitions $\lambda\vdash N$:
	\begin{equation}
		\label{eq::FIltrationBimodule}
		\begin{array}{c}	
			\mathcal{F}^{\lambda}\left(S^N\left(\Mat_{n\times m}\right)\right) :=\left({\bigoplus_{\nu \ge \lambda }}{S^{N}\left(\Mat_{n\times  m}\right)_{\nu}}\right)\U(\fb_{m});
			\\
			{}^\lambda\mathcal{F}\left(S^N\left(\Mat_{n\times m}\right)\right) := \U(\fb_{n}) \left(\bigoplus_{{\nu \ge \lambda}} {}_\nu S^{N}\left(\Mat_{n\times m}\right)\right).
		\end{array}		
	\end{equation}
	
	In particular, isomorphism~\eqref{eq::Howe::duality} can be weakened to the following isomorphisms:
	$$
	\calF^{\lambda}\left/ \sum_{\nu\geq \lambda}\calF^{\nu}\right.\left(S^{N}(\Mat_{n\times m})\right) 
	\simeq 
	{}^{\lambda}\calF \left/ \sum_{\nu\geq \lambda}{}^{\nu}\calF\right. \left( S^{N}(\Mat_{n\times m})\right) \simeq   
	D_{\lambda_{-}}\otimes D_{\lambda_{+}}^{\op}.
	$$
	Moreover, both quotients are cyclic $\fb_n\ttt\fb_m$-modules generated by the monomial
	$$
	v^{A_\lambda}:= v_{n, 1}^{\lambda_1} \cdot v_{n-1, 2}^{\lambda_2}\cdot v_{n-3, 3}^{\lambda_3} \ldots. 
	$$ 
	where $A_{\lambda}$ is the unique $DL$-dense array of rectangular shape $n^{m}$ whose multiset of elements is equal to $\lambda$.
	(See Example~\ref{ex::rectangular} for details).
	
	The embedding of the space of staircase matrices into the space of rectangular matrices induces an embedding of the corresponding symmetric tensors, and the symmetric algebra inherits the left and right filtrations by weights:\nlabel{ni:imath}
	\begin{gather*}
		\imath_{\ov{n}}: \Mat_{\ov{n}} \hookrightarrow \Mat_{n\times m};
		\quad
		\imath_{\ov{n}}: S^{N}(\Mat_{\ov{n}}) \hookrightarrow S^{N}(\Mat_{n\times m}); \\
		\calF^{\lambda} S^{N}(\Mat_{\ov{n}}):= S^{N}(\Mat_{\ov{n}}) \bigcap \calF^{\lambda} S^{N}(\Mat_{n\times m}).
	\end{gather*}
	It follows that for each partition $\lambda\vdash N$ with $l(\lambda)\leq\min(n,m)$, we have the following quotient map of $\fb_{n}\ttt\fb_{m}$-bimodules:
	\begin{equation}
		\label{eq::subquotient::map}
		\ov{\imath}_{\ov{n}}^{\lambda}:  \calF^{\lambda}\left/\sum_{\eta\succ\lambda}\calF^{\eta} \right. \left(S^{N}(\Mat_{\ov{n}})\right) \hookrightarrow  
		\calF^{\lambda}\left/\sum_{\eta>\lambda} \calF^{\eta} \right.
		\left(S^{N}(\Mat_{n\times m})\right) \simeq V^{\gl_n}_{\lambda}\otimes (V^{\gl_m}_{\lambda})^{\op}.
	\end{equation}
	and the analogous quotient map for the left filtration.
	
	\begin{proposition}
		\label{prp::monom::Demazure}	
		For each $DL$-dense array $A\in\DL_{\ov{n}}(\lambda)$, the monomial $v^{A}$ belongs to $\calF^{\lambda}S^{N}(\Mat_{n\times m})$ (respectively to ${}^{\lambda}\calF\left(S^{N}(\Mat_{n\times m})\right)$). Moreover, 
		the $\fb_n\ttt\fb_m$-submodule of $D_{\lambda_-}\otimes D_{\lambda_+}^{\op}$ generated by the vector $\ov{\imath}_{\ov{n}}^{\lambda}(v^A)$ is isomorphic to the Demazure submodule:
		\begin{equation}
			\label{eq::Demazure::sub}	
			D_{\hor(A)}\otimes D_{\vrt(A)}^{\op}.
		\end{equation}
	\end{proposition}
	
	\begin{proof}
		The module $S^{N}(\Mat_{n\times m})$, as well as its submodule $V_{\lambda}\otimes V_{\lambda}^{\op}$, is a $\mgl_n\ttt\mgl_m$-module and consequently admits a linear action of the product of two Weyl groups, i.e., the symmetric groups $\bS_n\times\bS_m$. Moreover, this action maps the highest weight vector to an extremal vector in an irreducible representation.
		
		Since $A_{\lambda}\in \DL_{n^m}(\lambda)$ is the unique $DL$-dense array of rectangular shape, the monomial $v^{A_\lambda}$ represents the highest weight vector in $\calF^{\lambda}/\calF^{>\lambda}(S^{N}(\Mat_{n\times m}))$ (see Example~\ref{ex::rectangular}). The monomial $v^{B}$ represents an extremal vector in this subquotient whenever the array $B$ belongs to the $\bS_n\times\bS_m$-orbit of $A_\lambda$. Since $A$ is $DL$-dense, it belongs to this orbit, implying that $\imath_{\ov{n}}(v^A)$ is an extremal vector of $V_\lambda\otimes V_\lambda^{\op}$. Consequently, the $\fb_n\ttt\fb_m$-submodule generated by it is the Demazure submodule $D_{\hor(A)}\otimes D_{\vrt(A)}^{\op}$.
	\end{proof}
	
	Note that the Demazure submodules~\eqref{eq::Demazure::sub} often intersect inside $V_{\lambda}\otimes V_{\lambda}^{\op}$. For example,
	\begin{multline}
		\label{eq::DD::intersect}	
		A\preceq_{\Bruhat} B\in \DL_{\ov{n}}(\lambda) \ \Leftrightarrow \ 
		\begin{cases}
			D_{\hor(A)}\subset D_{\hor(B)},\\
			D_{\vrt(A)}^{\op}\supset D_{\vrt(B)}^{\op}
		\end{cases} \ \Rightarrow \ \\
		\Rightarrow \ 
		\left(D_{\hor(A)}\otimes D_{\vrt(A)}^{\op}\right) \bigcap \left(D_{\hor(B)}\otimes D_{\vrt(B)}^{\op}\right) =  D_{\hor(A)}\otimes D_{\vrt(B)}^{\op}.
	\end{multline}
	However, thanks to Theorem~\ref{thm::D::Dop::distrib}, we know that $D_{\hor(A)}\otimes D_{\vrt(B)}^{\op}$ are $\vee$-irreducible elements of the distributive lattice $\Lattice_{D\times D^{\op}}^{\ov{n}}(\lambda)$. 
	
	\begin{theorem}
		\label{thm::Howe::lattice}	
		For each partition $\lambda\vdash N$ with $l(\lambda)\leq \#\St_{\ov{n}}$, the associated graded component of the left (and right) filtrations~\eqref{eq::FIltrationBimodule} is isomorphic to the following sum of $\fb_n\ttt\fb_m$-Demazure (bi)-submodules:
		\begin{equation}
			\label{eq::subquot::hor::vrt}
			\gr\calF^{\lambda}:=\calF^{\lambda}\left/\sum_{\nu>\lambda} \calF^{\nu} \right. \left(S^{N}(\Mat_{\ov{n}})\right) \simeq \sum_{A\in\DL_{\ov{n}}(\lambda)} D_{\hor(A)}\otimes D_{\vrt(A)}^{\op} \subset V_{\lambda}^{\gl_n}\otimes (V_{\lambda}^{\gl_m})^{\op}. 
		\end{equation}   
		Moreover, the corresponding $\fb_n\ttt\fb_m$-sub-bimodule is an element of the distributive lattice $\Lattice_{D\times D^{\op}}^{\ov{n}}(\lambda)$ whose associated graded module with respect to the standard filtration $\calF_{\Lattice}$, defined in~\eqref{eq::filtration::distributive}, is isomorphic to the following sum:
		\begin{equation}
			\label{eq::Howe::gr}
			\bigoplus_{B\preceq C\in\DL_{\ov{n}}(\lambda)} {\KK}_{\St_{\ov{n}},\hor(B)}\otimes {\KK}_{\St_{\ov{n}},\vrt(C)}^{\op}.
		\end{equation}
	\end{theorem}
	
	\begin{proof}
		Thanks to Proposition~\ref{prp::monom::Demazure}, we know that the elements $v^A$ for $A\in\DL_{\ov{n}}(\lambda)$ belong to $\calF^{\lambda}$ and generate the summands on the right-hand side of~\eqref{eq::subquot::hor::vrt}. Consequently, the surjectivity of~\eqref{eq::subquot::hor::vrt} follows.
		On the other hand, this also implies that $v^{A}$ does not belong to $\calF^{\nu}$ for any $\nu>\lambda$. 
		After Proposition~\ref{prop:BimoduleGenerators}, we know that monomials assigned to $DL$-dense arrays generate the symmetric algebra.
		It follows that $\{v^{A}\colon A\in\DL_{\ov{n}}(\lambda)\}$ generate the subquotient $\calF^{\lambda}/\calF^{>\lambda}$, proving the injectivity of~\eqref{eq::subquot::hor::vrt}.
		
		To describe the set of minimal subquotients ${\KK}_{A,B}$ appearing in the filtration~$\calF_{\Lattice}$, it remains to identify the set of $\vee$-irreducible elements of the lattice $\Lattice_{D\times D^{\op}}^{\ov{n}}(\lambda)$ that belong to the right-hand side of~\eqref{eq::subquot::hor::vrt}. 
		As mentioned in~\eqref{eq::DD::intersect}, the $\vee$-irreducible subspace of $D_{\ov{a}}$ is $D_{\ov{b}}$ with $b\preceq a$, and on the other hand, $D_{\ov{c}}^{\op}$ is a subspace of $D_{\ov{a}}^{\op}$ if and only if $\ov{c}\succeq \ov{a}$.
		Thus, the $\vee$-irreducible subspaces of $D_{\hor(A)}\otimes D_{\vrt(A)}^{\op}$ consist of the products 
		$$D_{\hor{B}}\otimes D_{\vrt(C)}^{\op} \text{ with } (\hor(B)\preceq\hor(A))\ \&\ (\vrt(C)\succeq \vrt(A)) \ \Leftrightarrow \ B\preceq A \preceq C. $$
		Since, in the right-hand side of~\eqref{eq::subquot::hor::vrt}, we sum over all possible $A$, we can omit the intermediate array $A$ and conclude that products $D_{\hor{B}}\otimes D_{\vrt(C)}^{\op}$ form the complete set of $\vee$-irreducible elements of $\gr\calF^{\lambda}$.
		The indexing set of the set of minimal subquotients of $\gr\calF^{\lambda}$ coincides with the set of $\vee$-irreducible submodules in it, what ensures the summation~\eqref{eq::Howe::gr}.
	\end{proof}
	
	We now compute the $\fh_{n}\ttt\fh_{m}$-character of $\gr\calF^{\lambda}$ based on Theorem~\ref{thm::Howe::lattice}:
	Let us compute the $\fh_{n}\ttt\fh_{m}$-character of $\gr\calF^{\lambda}$ based on Theorem~\ref{thm::Howe::lattice}:
	\begin{multline}
		\label{eq::grF::key::polynom}	
		\ch_{\fh_{n}\ttt\fh_{m}}(\gr\calF^{\lambda}) \stackrel{\eqref{eq::Howe::gr}}{=} \sum_{B\preceq A\in\DL_{\ov{n}}(\lambda)} \ch_{\fh_n}\left({\KK}_{\St_{\ov{n}},\hor(B)}\right) \ch_{\fh_m}\left( {\KK}_{\St_{\ov{n}},\vrt(A)}^{\op}\right) = 
		\\
		= \sum_{A\in\DL_{\ov{n}}(\lambda)}
		\left(\sum_{B\preceq A\in\DL_{\ov{n}}(\lambda)} \ch_{\fh_n}\left({\KK}_{\St_{\ov{n}},\hor(B)}\right)\right) \ch_{\fh_m}\left( {\KK}_{\St_{\ov{n}},\vrt(A)}^{\op}\right) = \\
		= \sum_{A\in\DL_{\ov{n}}(\lambda)} \ch_{\fb_n}(D_{\hor(A)})\,  \ch_{\fh_m}\left( {\KK}_{\St_{\ov{n}},\vrt(A)}^{\op}\right)
		\stackrel{\eqref{eq::vdK::op::char}}{=} 
		\sum_{A\in\DL_{\ov{n}}(\lambda)} \kappa_{\hor(A)}(x)\left(
		\sum_{ \substack{{\ov{d}\in\bS_n\lambda}\\ {\bbs_{\ov{n}}(\ov{d}) = \vrt(A)}}} a^{\ov{d}}(y) \right)
		= \\
		\stackrel{\eqref{eq::vdK::op::char}}{=} 
		\sum_{A\in\DL_{\ov{n}}(\lambda)} \kappa_{\hor(A)}(x)\, \left(\sum_{B\succeq A} \mu^{\DL_{\ov{n}}(\lambda)}_{A,B} \kappa^{\vrt(B)}(y) \right) =
		\sum_{A\preceq B \in\DL_{\ov{n}}(\lambda)} \mu^{\DL_{\ov{n}}(\lambda)}_{A,B}\, \kappa_{\hor(A)}(x) \, \kappa^{\vrt(B)}(y).
	\end{multline}
	
	Finally, we reproduce the main identities from~\cite{FKM::Cauchy}, known as "\emph{Cauchy identities for staircase matrices}":
	\begin{corollary}
		\label{cor::Cauchy::identity}
		The following identities hold for any Young diagram $\bbY_{\ov{n}}$:
		\begin{gather}
			\label{eq::Cauchy::bs}	
			\prod_{(i,j)\in \bbY_{\overline{n}}}\frac{1}{1-x_iy_j}=
			\sum_{A\in\DL_{\ov{n}}} \kappa_{\hor(A)}(x) \cdot \left(\sum_{\ov{d} \colon \bbs_{\ov{n}}(\ov{d}) = \vrt(A)} a^{\ov{d}}(y)\right),\\
			\label{eq::Cauchy::Moebius}	
			\prod_{(i,j)\in \bbY_{\overline{n}}}\frac{1}{1-x_i y_j}= \sum_{N}
			\sum_{\substack{{\lambda\vdash N}\\ {l(\lambda)\leq \#\St_{\ov{n}}}} }
			\sum_{\substack{A\preceq B\\ A,B \in \DL_{\overline n}(\lambda)}} \mu^{{\DL_{\ov{n}}({\lambda})}}_{A,B}\, \kappa_{\hor(A)}(x)\, \kappa^{\vrt(B)}(y).
		\end{gather}
	\end{corollary}
	\begin{proof}
		The left-hand side of these identities represents the $\fh_{n}\ttt\fh_{m}$-character of the symmetric algebra $S^{\udot}\left(\Mat_{\ov{n}}\right)$. In contrast, the right-hand side follows from the two different presentations of the character of the associated graded factors $\gr\calF^{\lambda}$ described in~\eqref{eq::grF::key::polynom}.
	\end{proof}
	\begin{remark}
		\label{rem::hb}
		In this paper, we use the \emph{bubble-sort} operation $\bbs_{\ov{n}}$ instead of its variant $\hb$, referred to as \emph{half-bubble-sort} in our previous work~\cite{FKM::Cauchy}, which is somewhat more cumbersome.    
	\end{remark}
	
	\begin{remark}
		Our choice of how to draw the Young diagram is motivated by staircase matrices, where we consider the left and right actions of upper triangular matrices.
		Recall that the British style of drawing a Young diagram $\bbY_{\ov{m}}^{\mathsf{Br}}$ associated to a partition $\ov{m}:=(m_1\geq m_2\geq\ldots \geq m_n)$ 
		is typically organized from top to bottom and from left to right ($m_i$ is the number of cells in the $i$'th row). 
		By flipping the Young diagram of staircase shape along the $y$-axis, one can naturally arrive at the definitions of the set of staircase corners $\St_{\ov{m}}^{\mathsf{Br}}$ and 
		the set $\mathsf{DR}_{\ov{m}}(\lambda)$ of DOWN-RIGHT dense arrays whose multiset of nonzero entries corresponds to a given partition $\lambda$.
		Finally, the Cauchy-type identity takes the following form:
		\[
		\prod_{(i,j)\in \bbY_{\ov{m}}^{\mathsf{Br}}}\frac{1}{1-x_i y_j}= \sum_{N}
		\sum_{\substack{{\lambda\vdash N}\\ {l(\lambda)\leq \#\St_{\ov{m}}^{\mathsf{Br}}}} }
		\sum_{\substack{A\succeq B\\ A,B \in \mathsf{DR}_{\overline m}(\lambda)}} \mu^{{\mathsf{DR}_{\ov{n}}({\lambda})}}_{A,B}\, \kappa_{\hor(A)}(x)\, \kappa_{\vrt(B)}(y).
		\]
		One advantage of the British-style drawing is that the right-hand side involves only {\it (left)} key polynomials.
	\end{remark}
	
	Instead of working out a precise formal definition for British-style Young diagrams, staircase corners and $DR$-arrays, we present a pictorial example for $\ov{m}:=(6^2\,5\,2^2)$.
	
	\begin{example}
		Bullets represent elements of the poset \( \St_{\ov{m}}^{\mathsf{Br}} \), and arrows indicate the covering relations in this poset:
		\begin{equation}
			\label{pic::Y::ST::intro}
			\begin{array}{ccc}
				\begin{tikzpicture}[xscale=-1, scale=0.3,shift={(0,2)}]
					\draw[step=1cm] (0,0) grid (6,-2);
					\draw[step=1cm] (1,0) grid (6,-3);
					\draw[step=1cm] (4,0) grid (6,-5);
				\end{tikzpicture} 
				&
				\begin{tikzpicture}[xscale=-1, scale=0.3,shift={(0,2)}]
					\fill[cyan!10] (0,0) -- (0,-2) -- (1,-2)--(1,0);
					\fill[cyan!20] (0,-1) -- (0,-2) -- (6,-2)--(6,-1);
					\fill[yellow] (1,0) -- (1,-3) -- (2,-3)--(2,0);
					\fill[yellow] (1,-2) -- (1,-3) -- (6,-3)--(6,-2);
					\fill[magenta!25] (4,0) -- (4,-5) -- (5,-5)--(5,0);
					\fill[magenta!25] (4,-4) -- (4,-5) -- (6,-5)--(6,-4);
					\fill[green!30] (2,0) -- (2,-1) -- (4,-1)--(4,0);
					\fill[lime!40] (5,0) -- (5,-4) -- (6,-4)--(6,0);
					\draw[step=1cm] (0,0) grid (6,-2);
					\draw[step=1cm] (1,0) grid (6,-3);
					\draw[step=1cm] (4,0) grid (6,-5);
					\node (v21) at (0.5,-1.5) {{ \color{blue}$\bullet$ }};
					\node (v32) at (1.5,-2.5) {{ \color{blue}$\bullet$ }};
					\node (v13) at (2.5,-.5) {{ \color{blue}$\bullet$ }};
					\node (v54) at (4.5,-4.5) {{ \color{blue}$\bullet$ }};
					\node (v46) at (5.5,-3.5) {{ \color{blue}$\bullet$ }};  
				\end{tikzpicture}
				&
				\begin{tikzpicture}[xscale=-1, scale=0.3,shift={(0,2)}]
					\node (v21) at (0.5,-1.5) {{ \color{blue}$\bullet$ }};
					\node (v32) at (1.5,-2.5) {{ \color{blue}$\bullet$ }};
					\node (v13) at (2.5,-.5) {{ \color{blue}$\bullet$ }};
					\node (v54) at (4.5,-4.5) {{ \color{blue}$\bullet$ }};
					\node (v46) at (5.5,-3.5) {{ \color{blue}$\bullet$ }};  
					\draw [blue,line width=1.1pt,->] (2.5,-.5) -- (.5,-1.5);
					\draw [blue,line width=1.1pt,->] (2.5,-.5) -- (1.5,-2.5);
					\draw [blue,line width=1.1pt,->] (v54) -- (v46);
				\end{tikzpicture} 
				\\
				\begin{array}{c}
					\text{ British Young diagram }\\
					\bbY_{(6^2\,5\,2^2)}^{\mathsf{Br}}
				\end{array}
				& 
				\begin{array}{c}
					\text{ "rook placement" of}\\
					\text{ staircase corners } 
				\end{array}
				&
				\begin{array}{c}
					\text{ Hasse diagram}\\
					\text{ of } \St_{(6^2\,5\,2^2)}^{\mathsf{Br}}
				\end{array}
			\end{array}
		\end{equation}
	\end{example}

	\newpage
	
	\appendix
	
	\section{Pictorial examples}
	\label{sec::Example::DL}
	
	We present several illustrative examples of:  
	\begin{itemize}
		\item Arborescent posets \( \sS \) with a consistent (anti)linearization.
		\item The Hasse diagram of the poset \( \bD_{\sS}(\lambda) \) ordered by the Bruhat partial order.
		\item The Hasse diagram of the poset \( \DL_{\ov{n}}(\lambda) \) of \( DL \)-dense arrays, along with its vertical and horizontal embeddings into the Bruhat graph.
	\end{itemize}
	Note that all our pictorial descriptions of the posets are slightly reversed compared to standard notation.  
	Specifically, we draw the minimal element at the top and the maximal element at the bottom.  
	However, the corresponding Bruhat (sub)graphs are presented in the standard orientation.  
	This reversal occurs because we consider the dual order to the Bruhat order, based on comparisons of Demazure modules.
	
	\begin{example}
		\label{ex::tri::arb}
		Here is a list of arborescent posets of cardinality \( 3 \) equipped with a surjective, consistent anti-linearization:
		$$
		\begin{array}{ccccc}
			\sS^1=
			\begin{tikzpicture}[scale=0.5, shift={(0,-.5)}]
				\draw[step=1cm] (0,1) grid (3,0);
				\node (v1) at (.5,.4) {{ \color{blue}$\bullet$ }};
				\node (v2) at (1.5,.4) {{ \color{blue}$\bullet$ }};
				\node (v3) at (2.5,.4) {{ \color{blue}$\bullet$ }};
				\draw  (2.5,.6) edge[blue,line width=1.1pt,->,bend right=90] (.5,.6);
				\draw (2.5,.6) edge[blue,line width=1.1pt,->,bend right=60] (1.5,.6);
			\end{tikzpicture},
			& 
			\sS^2=
			\begin{tikzpicture}[scale=0.5, shift={(0,-.5)}]
				\draw[step=1cm] (0,1) grid (3,0);
				\node (v1) at (.5,.4) {{ \color{blue}$\bullet$ }};
				\node (v2) at (1.5,.4) {{ \color{blue}$\bullet$ }};
				\node (v3) at (2.5,.4) {{ \color{blue}$\bullet$ }};
				\draw  (1.5,.6) edge[blue,line width=1.1pt,->,bend right=60] (.5,.6);
				\draw (2.5,.6) edge[blue,line width=1.1pt,->,bend right=60] (1.5,.6);
			\end{tikzpicture},
			&
			\sS^3=
			\begin{tikzpicture}[scale=0.5, shift={(0,-.5)}]
				\draw[step=1cm] (0,1) grid (3,0);
				\node (v1) at (.5,.4) {{ \color{blue}$\bullet$ }};
				\node (v2) at (1.5,.4) {{ \color{blue}$\bullet$ }};
				\node (v3) at (2.5,.4) {{ \color{blue}$\bullet$ }};
				\draw (2.5,.6) edge[blue,line width=1.1pt,->,bend right=60] (1.5,.6);
			\end{tikzpicture},
			&
			\sS^4=
			\begin{tikzpicture}[scale=0.5, shift={(0,-.5)}]
				\draw[step=1cm] (0,1) grid (3,0);
				\node (v1) at (.5,.4) {{ \color{blue}$\bullet$ }};
				\node (v2) at (1.5,.4) {{ \color{blue}$\bullet$ }};
				\node (v3) at (2.5,.4) {{ \color{blue}$\bullet$ }};
				\draw (1.5,.6) edge[blue,line width=1.1pt,->,bend right=60] (.5,.6);
			\end{tikzpicture},
			&
			\sS^5=
			\begin{tikzpicture}[scale=0.5, shift={(0,-.5)}]
				\draw[step=1cm] (0,1) grid (3,0);
				\node (v1) at (.5,.4) {{ \color{blue}$\bullet$ }};
				\node (v2) at (1.5,.4) {{ \color{blue}$\bullet$ }};
				\node (v3) at (2.5,.4) {{ \color{blue}$\bullet$ }};
			\end{tikzpicture}, 
			\\
			& & & & \\
			\bD_{\sS^1}\simeq \bS_2,
			&
			\bD_{\sS^2}\simeq \bS_1,
			&
			\bD_{\sS^3}\simeq \bS_2\backslash\bS_3,
			&
			\bD_{\sS^4}\simeq \bS_2\backslash\bS_3,
			&
			\bD_{\sS^5}\simeq \bS_3.
		\end{array}
		$$
	\end{example}
	Note that if the Hasse diagram of an arborescent poset \( \sS = \{s_1, \ldots, s_k\} \) with a consistent antilinearization \( v:\sS\stackrel{s_i\mapsto i}{\longrightarrow}[1,k] \) is connected, then \( s_k \) is the unique minimal element of \( \sS \), and we obtain the following isomorphism of posets of dominant compositions:
	\[
	\bD_{\sS}^{v}((\lambda_1\geq\ldots\geq\lambda_k)) \simeq \bD_{\sS'}^{v}((\lambda_1\geq\ldots\geq\lambda_{k-1})), \quad \text{where } \sS':=\{s_1,\ldots,s_{k-1}\}.
	\]
	
	\newpage 
	\begin{example}
		There exists a unique arborescent poset \( \sS \) of cardinality \( 4 \) whose Hasse diagram has multiple connected components and is not a disjoint union of linearly ordered sets, which are covered by Example~\ref{ex::Double::coset}. 
		
		The corresponding poset is the set of staircase corners for the shapes \( \bbY_{(2 3^2 4)} \) and \( \bbY_{(1 3^2 4)} \). These two shapes differ by a transposition, leading to two distinct antilinearizations that can be visualized as vertical map $\vrt$ and opposite to the horisontal map $\hor$.
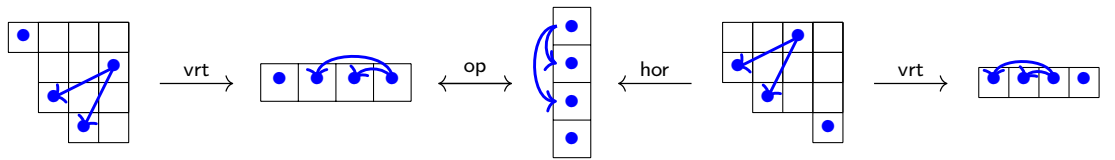
		$$
		\begin{tikzcd}
			\begin{tikzpicture}[scale=0.4,shift={(0,2)}]
				\draw[step=1cm] (0,0) grid (4,-1);
				\draw[step=1cm] (1,0) grid (4,-3);
				\draw[step=1cm] (2,0) grid (4,-4);
				\node (v21) at (0.5,-.7) {{ \color{blue}$\bullet$ }};
				\node (v32) at (1.5,-2.7) {{ \color{blue}$\bullet$ }};
				\node (v13) at (2.5,-3.7) {{ \color{blue}$\bullet$ }};
				\node (v44) at (3.5,-1.7) {{ \color{blue}$\bullet$ }};
				\draw  (3.5,-1.5) edge[blue,line width=1.1pt,->] (1.5,-2.5);
				\draw (3.5,-1.5) edge[blue,line width=1.1pt,->] (2.5,-3.5);  
			\end{tikzpicture}
			\arrow[r,"\vrt"]
			&
			\begin{tikzpicture}[scale=0.5, shift={(0,-.5)}]
				\draw[step=1cm] (-1,1) grid (3,0);
				\node (v1) at (.5,.4) {{ \color{blue}$\bullet$ }};
				\node (v2) at (1.5,.4) {{ \color{blue}$\bullet$ }};
				\node (v3) at (2.5,.4) {{ \color{blue}$\bullet$ }};
				\node (v4) at (-.5,.4) {{ \color{blue}$\bullet$ }};
				\draw  (2.5,.6) edge[blue,line width=1.1pt,->,bend right=90] (.5,.6);
				\draw (2.5,.6) edge[blue,line width=1.1pt,->,bend right=60] (1.5,.6);
			\end{tikzpicture}
			\arrow[r,"\op"]
			&
			\arrow[l]
			\begin{tikzpicture}[scale=0.5]
				\draw[step=1cm] (0,2) grid (1,-2);    
				\node (v1) at (.5,.3) {{ \color{blue}$\bullet$ }};
				\node (v2) at (.5,-.7) {{ \color{blue}$\bullet$ }};
				\node (v3) at (.5,1.3) {{ \color{blue}$\bullet$ }};
				\node (v5) at (.5,-1.7) {{ \color{blue}$\bullet$ }};
				\draw  (.1,1.5) edge[blue,line width=1.1pt,->,bend right=90] (.1,-.5);
				\draw (.1,1.5) edge[blue,line width=1.1pt,->,bend right=70] 
				(.1,.5);
			\end{tikzpicture}
			&
			\arrow[l,"\hor"']
			\begin{tikzpicture}[scale=0.4,shift={(0,2)}]
				\draw[step=1cm] (0,0) grid (4,-2);
				\draw[step=1cm] (1,0) grid (4,-3);
				\draw[step=1cm] (3,0) grid (4,-4);
				\node (v21) at (0.5,-1.7) {{ \color{blue}$\bullet$ }};
				\node (v32) at (1.5,-2.7) {{ \color{blue}$\bullet$ }};
				\node (v13) at (2.5,-.7) {{ \color{blue}$\bullet$ }};
				\node (v44) at (3.5,-3.7) {{ \color{blue}$\bullet$ }};
				\draw  (2.5,-.5) edge[blue,line width=1.1pt,->] (.5,-1.5);
				\draw (2.5,-.5) edge[blue,line width=1.1pt,->] (1.5,-2.5);  
			\end{tikzpicture}
			\arrow[r,"\vrt"]
			&
			\begin{tikzpicture}[scale=0.4, shift={(0,-.5)}]
				\draw[step=1cm] (0,1) grid (4,0);
				\node (v1) at (.5,.4) {{ \color{blue}$\bullet$ }};
				\node (v2) at (1.5,.4) {{ \color{blue}$\bullet$ }};
				\node (v3) at (2.5,.4) {{ \color{blue}$\bullet$ }};
				\node (v4) at (3.5,.4) {{ \color{blue}$\bullet$ }};
				\draw  (2.5,.6) edge[blue,line width=1.1pt,->,bend right=90] (.5,.6);
				\draw (2.5,.6) edge[blue,line width=1.1pt,->,bend right=60] (1.5,.6);
			\end{tikzpicture}.
		\end{tikzcd}
		$$
		These two different anti-linearisations appear as the vertical map $\vrt$ and the opposite to the horisontal map $\hor$ of the same shape $\bbY_{(2 3^2 4)}$.
		Below we draw the Bruhat graph for this set $\DL_{(2 3^2 4)}$ together with the vertical embedding (the elements in the image are filled in orange). Each $DL$-dense array is drawn in its own colour, and the fibre of the bubble-sort map $\bbs_{\ov{n}}$ over its vertical weight is shaded in the same colour.
		\colorlet{fibA}{Red} \colorlet{fibB}{Cyan} \colorlet{fibC}{LimeGreen} \colorlet{fibD}{Mulberry} \colorlet{fibE}{Goldenrod} \colorlet{fibF}{RoyalPurple} \colorlet{fibG}{RoyalBlue} \colorlet{fibH}{RawSienna}%
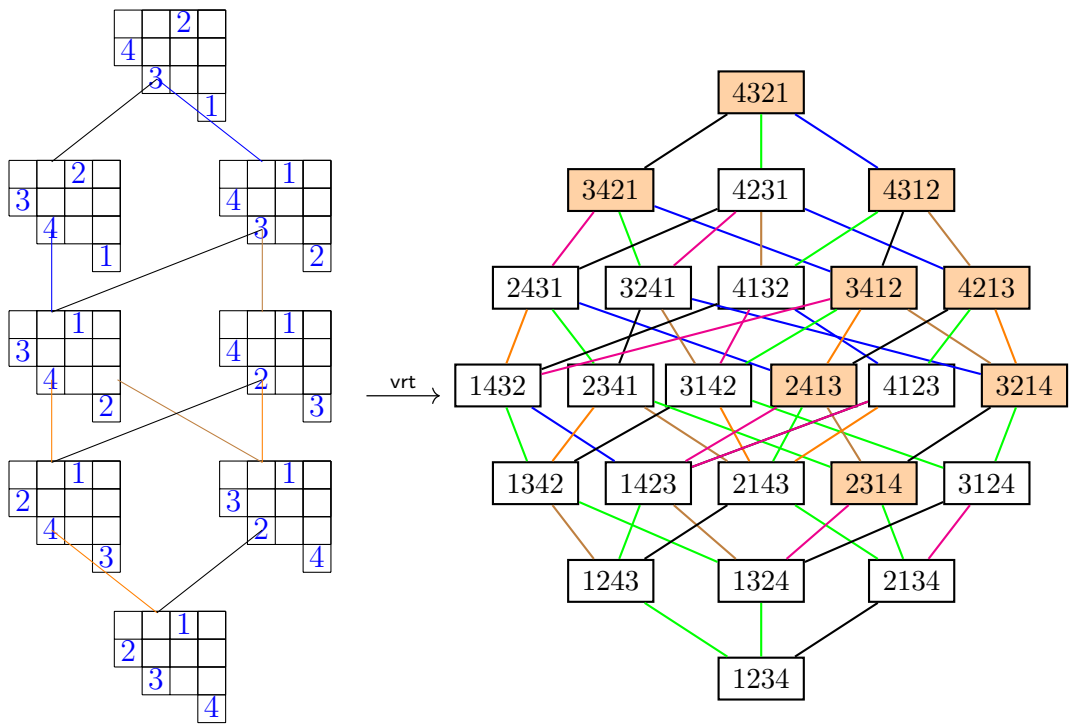
		$$
		\begin{tikzcd}
			\begin{tikzpicture}[xscale=0.7, yscale=1,shift={(0,-3)}]
				\node (4321) at (0,8) {$\gridDLc{fibA}4321$};
				\node (3421) at (-2,6) {$\gridDLc{fibC}3421$};
				\node (4312) at (2,6) {$\gridDLc{fibB}4312$};
				\node (3412) at (-2,4) {$\gridDLc{fibE}3412$};
				\node (4213) at (2,4) {$\gridDLc{fibD}4213$};
				\node (2413) at (-2,2) {$\gridDLc{fibG}2413$};
				\node (3214) at (2,2) {$\gridDLc{fibF}3214$};
				\node (2314) at (0,0) {$\gridDLc{fibH}2314$};
				\draw[black] (-0.25,7.2) edge (-2.25,6.1);
				\draw[blue] (-0.25,7.2) -- (1.75,6.1);
				\draw[blue] (-2.25,5.2) -- (-2.25,4.1);
				\draw[black] (1.75,5.2) -- (-2.25,4.1);
				\draw[brown] (1.75,5.2) -- (1.75,4.1);
				\draw[orange] (-2.25,3.2) -- (-2.25,2.1);
				\draw[brown] (-1,3.2) -- (1.75,2.1);
				\draw[black] (1.75,3.2) -- (-2.25,2.1);
				\draw[orange] (1.75,3.2) -- (1.75,2.1);
				\draw[orange] (-2.25,1.2) -- (-0.25,.1);
				\draw[black] (1.75,1.2) -- (-0.25,.1);
			\end{tikzpicture}
			\arrow[r,"\vrt"]
			&		
			\begin{tikzpicture}[xscale=1.15, yscale=1.35, shift={(0,-3)},
				every node/.style={draw, rectangle, rounded corners=2pt, fill=white, inner sep=2pt, font=\small},thick]
				\begin{scope}[transparency group, opacity=0.6] \filldraw[fibA!40, line width=7mm, line cap=round, line join=round] (0,6) -- (0,5); \end{scope}
				\begin{scope}[transparency group, opacity=0.6] \filldraw[fibB!40, line width=7mm, line cap=round, line join=round] (2,5) -- (0,4); \end{scope}
				\begin{scope}[transparency group, opacity=0.6] \filldraw[fibC!40, line width=7mm, line cap=round, line join=round] (-3.5,3) -- (-3,4) -- (-2,5) -- (-1.5,4); \end{scope}
				\begin{scope}[transparency group, opacity=0.6] \filldraw[fibD!40, line width=7mm, line cap=round, line join=round] (3,4) -- (2,3); \end{scope}
				\begin{scope}[transparency group, opacity=0.6] \filldraw[fibE!40, line width=7mm, line cap=round, line join=round] (1.5,4) -- (-0.7,3) -- (-3,2) -- (-2,3) -- cycle; \end{scope}
				\begin{scope}[transparency group, opacity=0.6] \filldraw[fibF!40, line width=7mm, line cap=round, line join=round] (3.5,3) -- (3,2); \end{scope}
				\begin{scope}[transparency group, opacity=0.6] \filldraw[fibG!40, line width=7mm, line cap=round, line join=round] (0.7,3) -- (0,2) -- (-2,1) -- (-1.5,2) -- cycle; \end{scope}
				\begin{scope}[transparency group, opacity=0.6] \filldraw[fibH!40, line width=7mm, line cap=round, line join=round] (1.5,2) -- (2,1) -- (0,0) -- (0,1) -- cycle; \end{scope}
				
				\node[draw=fibA, line width=1.2pt, fill=orange!35] (4321) at (0,6) {{4}{3}{2}{1}};
				\node[draw=fibC, line width=1.2pt, fill=orange!35] (3421) at (-2,5) {{3}{4}{2}{1}};
				\node[draw=fibA, line width=1.2pt] (4231) at (0,5) {{4}{2}{3}{1}};
				\node[draw=fibB, line width=1.2pt, fill=orange!35] (4312) at (2,5) {{4}{3}{1}{2}};
				\node[draw=fibC, line width=1.2pt] (2431) at (-3,4) {{2}{4}{3}{1}};
				\node[draw=fibC, line width=1.2pt] (3241) at (-1.5,4) {{3}{2}{4}{1}};
				\node[draw=fibB, line width=1.2pt] (4132) at (0,4) {{4}{1}{3}{2}};
				\node[draw=fibE, line width=1.2pt, fill=orange!35] (3412) at (1.5,4) {{3}{4}{1}{2}};
				\node[draw=fibD, line width=1.2pt, fill=orange!35] (4213) at (3,4) {{4}{2}{1}{3}};
				\node[draw=fibE, line width=1.2pt] (1432) at (-2,3) {{1}{4}{3}{2}};
				\node[draw=fibC, line width=1.2pt] (2341) at (-3.5,3) {{2}{3}{4}{1}};
				\node[draw=fibE, line width=1.2pt] (3142) at (-0.7,3) {{3}{1}{4}{2}};
				\node[draw=fibG, line width=1.2pt, fill=orange!35] (2413) at (0.7,3) {{2}{4}{1}{3}};
				\node[draw=fibD, line width=1.2pt] (4123) at (2,3) {{4}{1}{2}{3}};
				\node[draw=fibF, line width=1.2pt, fill=orange!35] (3214) at (3.5,3) {{3}{2}{1}{4}};
				\node[draw=fibE, line width=1.2pt] (1342) at (-3,2) {{1}{3}{4}{2}};
				\node[draw=fibG, line width=1.2pt] (1423) at (-1.5,2) {{1}{4}{2}{3}};
				\node[draw=fibG, line width=1.2pt] (2143) at (0,2) {{2}{1}{4}{3}};
				\node[draw=fibH, line width=1.2pt, fill=orange!35] (2314) at (1.5,2) {{2}{3}{1}{4}};
				\node[draw=fibF, line width=1.2pt] (3124) at (3,2) {{3}{1}{2}{4}};
				\node[draw=fibG, line width=1.2pt] (1243) at (-2,1) {{1}{2}{4}{3}};
				\node[draw=fibH, line width=1.2pt] (1324) at (0,1) {{1}{3}{2}{4}};
				\node[draw=fibH, line width=1.2pt] (2134) at (2,1) {{2}{1}{3}{4}};
				\node[draw=fibH, line width=1.2pt] (1234) at (0,0) {{1}{2}{3}{4}};
				
				\draw[black] (4321) -- (3421);
				\draw[green] (4321) -- (4231);
				\draw[blue] (4321) -- (4312);
				\draw[magenta] (3421) -- (2431);
				\draw[green] (3421) -- (3241);
				\draw[blue] (3421) -- (3412);
				\draw[black] (4231) -- (2431);
				\draw[magenta] (4231) -- (3241);
				\draw[brown] (4231) -- (4132);
				\draw[blue] (4231) -- (4213);
				\draw[black] (4312) -- (3412);
				\draw[brown] (4312) -- (4213);
				\draw[green] (4312) -- (4132);
				\draw[orange] (2431) -- (1432);
				\draw[green] (2431) -- (2341);
				\draw[blue] (2431) -- (2413);
				\draw[black] (3241) -- (2341);
				\draw[brown] (3241) -- (3142);
				\draw[blue] (3241) -- (3214);
				\draw[black] (4132) -- (1432);
				\draw[magenta] (4132) -- (3142);
				\draw[blue] (4132) -- (4123);
				\draw[magenta] (3412) -- (1432);
				\draw[green] (3412) -- (3142);
				\draw[orange] (3412) -- (2413);
				\draw[brown] (3412) -- (3214);
				\draw[black] (4213) -- (2413);
				\draw[green] (4213) -- (4123);			
				\draw[orange] (4213) -- (3214);
				\draw[green] (1432) -- (1342);
				\draw[blue] (1432) -- (1423);
				\draw[orange] (2341) -- (1342);
				\draw[brown] (2341) -- (2143);
				\draw[green] (2341) -- (2314);
				\draw[black] (3142) -- (1342);
				\draw[orange] (3142) -- (2143);
				\draw[green] (3142) -- (3124);
				\draw[magenta] (2413) -- (1423);
				\draw[green] (2413) -- (2143);
				\draw[brown] (2413) -- (2314);
				\draw[black] (4123) -- (1423);
				\draw[magenta] (4123) -- (1423);
				\draw[orange] (4123) -- (2143);
				\draw[black] (3214) -- (2314);
				\draw[green] (3214) -- (3124);
				\draw[brown] (1342) -- (1243);
				\draw[green] (1342) -- (1324);
				\draw[green] (1423) -- (1243);
				\draw[brown] (1423) -- (1324);
				\draw[black] (2143) -- (1243);
				\draw[green] (2143) -- (2134);
				\draw[magenta] (2314) -- (1324);
				\draw[green] (2314) -- (2134);
				\draw[black] (3124) -- (1324);
				\draw[magenta] (3124) -- (2134);
				\draw[green] (1243) -- (1234);
				\draw[green] (1324) -- (1234);
				\draw[black] (2134) -- (1234);
			\end{tikzpicture}
		\end{tikzcd}
		$$
		\begin{center}
			$\bbs_{\ov{n}}^{-1}(a)=\bigl\{w\succeq a \colon w\not\succeq b \text{ for all } b\in\vrt(\DL_{\ov{n}}),\ b\succ a\bigr\}$, \quad \colorbox{orange!35}{$a$}${}=\vrt(A)$, $A\in\DL_{\ov{n}}$
		\end{center}
		
		The next picture shows the horizontal embedding of the same poset $\DL_{(2,3,3,4)}$:
		$$
		
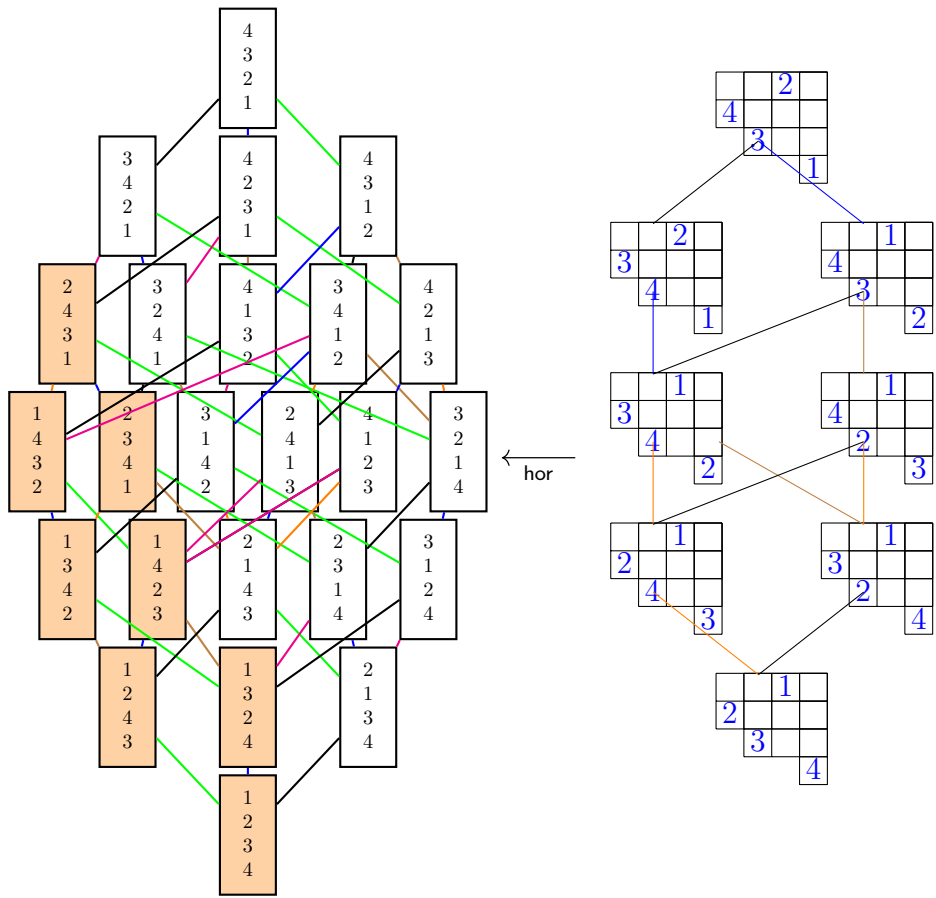
\begin{tikzcd}
			\begin{tikzpicture}[xscale=0.8, yscale=1.7, shift={(0,-3)},
				node distance=0.8cm and 1.2cm, 
				every node/.style={draw, rectangle, font=\small},
				thick]
				\node (4321) at (0, 6) 
				{\scalebox{0.7}{$\begin{array}{c} 4 \\ 3 \\ 2 \\ 1 \end{array}$}};
				
				\node (3421) at (-2, 5) {\scalebox{0.7}{$\begin{array}{c} 3 \\ 4 \\ 2 \\ 1 \end{array}$}};
				\node (4231) at (0, 5) {\scalebox{0.7}{$\begin{array}{c} 4 \\ 2 \\ 3 \\ 1 \end{array}$}};
				\node (4312) at (2, 5) {\scalebox{0.7}{$\begin{array}{c} 4 \\ 3 \\ 1 \\ 2 \end{array}$}};
				
				\node[draw, fill=orange!35] (2431) at (-3, 4) {\scalebox{0.7}{$\begin{array}{c} 2 \\ 4 \\ 3 \\ 1 \end{array}$}};
				\node (3241) at (-1.5, 4) {\scalebox{0.7}{$\begin{array}{c} 3 \\ 2 \\ 4 \\ 1 \end{array}$}};
				\node (4132) at (0, 4) {\scalebox{0.7}{$\begin{array}{c} 4 \\ 1 \\ 3 \\ 2 \end{array}$}};
				\node (3412) at (1.5, 4) {\scalebox{0.7}{$\begin{array}{c} 3 \\ 4 \\ 1 \\ 2 \end{array}$}};
				\node (4213) at (3, 4) {\scalebox{0.7}{$\begin{array}{c} 4 \\ 2 \\ 1 \\ 3 \end{array}$}};
				
				\node[draw, fill=orange!35] (1432) at (-3.5, 3) {\scalebox{0.7}{$\begin{array}{c} 1 \\ 4 \\ 3 \\ 2 \end{array}$}};
				\node[draw, fill=orange!35] (2341) at (-2, 3) {\scalebox{0.7}{$\begin{array}{c} 2 \\ 3 \\ 4 \\ 1 \end{array}$}};
				\node (3142) at (-0.7, 3) {\scalebox{0.7}{$\begin{array}{c} 3 \\ 1 \\ 4 \\ 2 \end{array}$}};
				\node (2413) at (0.7, 3) {\scalebox{0.7}{$\begin{array}{c} 2 \\ 4 \\ 1 \\ 3 \end{array}$}};
				\node (4123) at (2, 3) {\scalebox{0.7}{$\begin{array}{c} 4 \\ 1 \\ 2 \\ 3 \end{array}$}};
				\node (3214) at (3.5, 3) {\scalebox{0.7}{$\begin{array}{c} 3 \\ 2 \\ 1 \\ 4 \end{array}$}};
				
				\node[draw, fill=orange!35] (1342) at (-3, 2) {\scalebox{0.7}{$\begin{array}{c} 1 \\ 3 \\ 4 \\ 2 \end{array}$}};
				\node[draw, fill=orange!35] (1423) at (-1.5, 2) {\scalebox{0.7}{$\begin{array}{c} 1 \\ 4 \\ 2 \\ 3 \end{array}$}};
				\node (2143) at (0, 2) {\scalebox{0.7}{$\begin{array}{c} 2 \\ 1 \\ 4 \\ 3 \end{array}$}};
				\node (2314) at (1.5, 2) {\scalebox{0.7}{$\begin{array}{c} 2 \\ 3 \\ 1 \\ 4 \end{array}$}};
				\node (3124) at (3, 2) {\scalebox{0.7}{$\begin{array}{c} 3 \\ 1 \\ 2 \\ 4 \end{array}$}};
				
				\node[draw, fill=orange!35] (1243) at (-2, 1) {\scalebox{0.7}{$\begin{array}{c} 1 \\ 2 \\ 4 \\ 3 \end{array}$}};
				\node[draw, fill=orange!35] (1324) at (0, 1) {\scalebox{0.7}{$\begin{array}{c} 1 \\ 3 \\ 2 \\ 4 \end{array}$}};
				\node (2134) at (2, 1) {\scalebox{0.7}{$\begin{array}{c} 2 \\ 1 \\ 3 \\ 4 \end{array}$}};
				
				\node[draw, fill=orange!35] (1234) at (0, 0) {\scalebox{0.7}{$\begin{array}{c} 1 \\ 2 \\ 3 \\ 4 \end{array}$}};

				
				
				\draw[black] (4321) -- (3421);
				\draw[blue] (4321) -- (4231);
				\draw[green] (4321) -- (4312);
				
				\draw[magenta] (3421) -- (2431);
				\draw[blue] (3421) -- (3241);
				\draw[green] (3421) -- (3412);
				
				\draw[black] (4231) -- (2431);
				\draw[magenta] (4231) -- (3241);
				\draw[brown] (4231) -- (4132);
				\draw[green] (4231) -- (4213);
				
				\draw[black] (4312) -- (3412);
				\draw[brown] (4312) -- (4213);
				\draw[blue] (4312) -- (4132);
				
				\draw[orange] (2431) -- (1432);
				\draw[blue] (2431) -- (2341);
				\draw[green] (2431) -- (2413);
				
				\draw[black] (3241) -- (2341);
				\draw[brown] (3241) -- (3142);
				\draw[green] (3241) -- (3214);
				
				\draw[black] (4132) -- (1432);
				\draw[magenta] (4132) -- (3142);
				\draw[green] (4132) -- (4123);
				
				\draw[magenta] (3412) -- (1432);
				\draw[blue] (3412) -- (3142);
				\draw[orange] (3412) -- (2413);
				\draw[brown] (3412) -- (3214);

				\draw[black] (4213) -- (2413);
				\draw[blue] (4213) -- (4123);			
				\draw[orange] (4213) -- (3214);
				
				\draw[blue] (1432) -- (1342);
				\draw[green] (1432) -- (1423);
				
				\draw[orange] (2341) -- (1342);
				\draw[brown] (2341) -- (2143);
				\draw[green] (2341) -- (2314);
				
				\draw[black] (3142) -- (1342);
				\draw[orange] (3142) -- (2143);
				\draw[green] (3142) -- (3124);
				
				\draw[magenta] (2413) -- (1423);
				\draw[blue] (2413) -- (2143);
				\draw[brown] (2413) -- (2314);
				
				\draw[black] (4123) -- (1423);
				\draw[magenta] (4123) -- (1423);
				\draw[orange] (4123) -- (2143);
				
				\draw[black] (3214) -- (2314);
				\draw[blue] (3214) -- (3124);
				
				\draw[brown] (1342) -- (1243);
				\draw[green] (1342) -- (1324);
				
				\draw[blue] (1423) -- (1243);
				\draw[brown] (1423) -- (1324);
				
				\draw[black] (2143) -- (1243);
				\draw[green] (2143) -- (2134);
				
				\draw[magenta] (2314) -- (1324);
				\draw[blue] (2314) -- (2134);
				
				\draw[black] (3124) -- (1324);
				\draw[magenta] (3124) -- (2134);
				
				\draw[green] (1243) -- (1234);
				\draw[blue] (1324) -- (1234);
				\draw[black] (2134) -- (1234);
				
			\end{tikzpicture}
			&
			\begin{tikzpicture}[xscale=0.7, yscale=1,shift={(0,-3)}]
				\node (4321) at (0,8) {$\gridDL4321$};
				\node (3421) at (-2,6) {$\gridDL3421$};
				\node (4312) at (2,6) {$\gridDL4312$};
				\node (3412) at (-2,4) {$\gridDL3412$};
				\node (4213) at (2,4) {$\gridDL4213$};
				\node (2413) at (-2,2) {$\gridDL2413$};
				\node (3214) at (2,2) {$\gridDL3214$};
				\node (2314) at (0,0) {$\gridDL2314$};
				\draw[black] (-0.25,7.2) edge (-2.25,6.1);
				\draw[blue] (-0.25,7.2) -- (1.75,6.1);
				\draw[blue] (-2.25,5.2) -- (-2.25,4.1);
				\draw[black] (1.75,5.2) -- (-2.25,4.1);
				\draw[brown] (1.75,5.2) -- (1.75,4.1);
				\draw[orange] (-2.25,3.2) -- (-2.25,2.1);
				\draw[brown] (-1,3.2) -- (1.75,2.1);
				\draw[black] (1.75,3.2) -- (-2.25,2.1);
				\draw[orange] (1.75,3.2) -- (1.75,2.1);
				\draw[orange] (-2.25,1.2) -- (-0.25,.1);
				\draw[black] (1.75,1.2) -- (-0.25,.1);
			\end{tikzpicture}
			\arrow[l,"\hor"] 
		\end{tikzcd}
		$$	
	\end{example}
	
	\newpage 
	
	Here is an example of the Hasse diagram of the smaller poset $\DL_{\ov{n}}(\lambda)$ of the same shape $\ov{n}:=(2,3,3,4)$ but the partition $\lambda=(2,2,1,1)$ has equal elements. So the poset is embedded in the parabolic Bruhat graph:
	
	
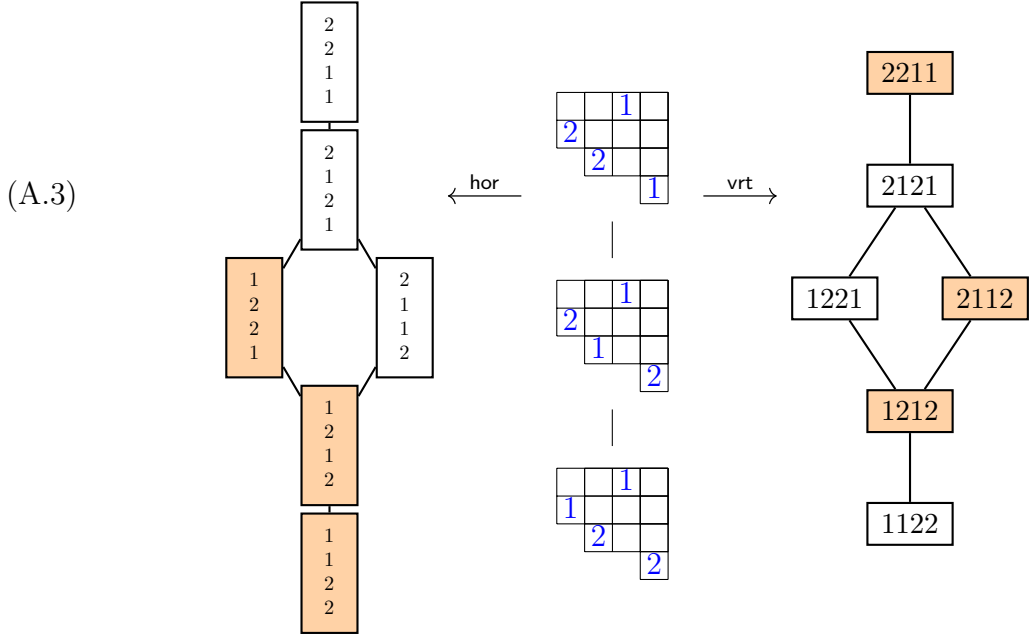
\begin{equation}
		\label{ex::Bruhat::2211}
		\begin{tikzcd}
			\begin{tikzpicture}[xscale=1, yscale=1.7, shift={(0,-3)},
				node distance=0.8cm and 1.2cm, 
				every node/.style={draw, rectangle, font=\small},thick]
				\node (2211) at (0, 4) {\scalebox{0.7}{$\begin{array}{c} 2 \\ 2 \\ 1 \\ 1 \end{array}$}}; 
				\node (2121) at (0, 3) {\scalebox{0.7}{$\begin{array}{c} 2 \\ 1 \\ 2 \\ 1 \end{array}$}}; 
				\node[draw, fill=orange!35] (1221) at (-1, 2) {\scalebox{0.7}{$\begin{array}{c} 1 \\ 2 \\ 2 \\ 1 \end{array}$}}; 
				\node (2112) at (1, 2) {\scalebox{0.7}{$\begin{array}{c} 2 \\ 1 \\ 1 \\ 2 \end{array}$}};
				\node[draw, fill=orange!35] (1212) at (0, 1) {\scalebox{0.7}{$\begin{array}{c} 1 \\ 2 \\ 1 \\ 2 \end{array}$}};
				\node[draw, fill=orange!35] (1122) at (0, 0) {\scalebox{0.7}{$\begin{array}{c} 1 \\ 1 \\ 2 \\ 2 \end{array}$}};
				\draw (2211) -- (2121);
				\draw (2121) -- (1221);
				\draw (2121) -- (2112);
				\draw (1221) -- (1212);
				\draw (2112) -- (1212);
				\draw (1212) -- (1122);
			\end{tikzpicture}
			&
			\begin{tikzpicture}[xscale=0.8, yscale=1.25,shift={(0,-3)}]
				\node (2211) at (0,4) {$\gridDL2211$};
				\node (2112) at (0,2) {$\gridDL2112$};
				\node (1212) at (0,0) {$\gridDL1212$};
				\draw (2211) -- (2112);
				\draw (2112) -- (1212);
			\end{tikzpicture}
			\arrow[r,"\vrt"]
			\arrow[l,"\hor"']
			&
			\begin{tikzpicture}[xscale=1, yscale=1.5, shift={(0,-3)},
				node distance=0.8cm and 1.2cm, 
				every node/.style={draw, rectangle, minimum size=3mm, font=\small},thick]
				\node[draw, fill=orange!35] (2211) at (0, 4) {{2}{2}{1}{1}};
				\node (2121) at (0, 3) {{2}{1}{2}{1}};
				\node (1221) at (-1, 2) {{1}{2}{2}{1}};
				\node[draw, fill=orange!35] (2112) at (1, 2) {{2}{1}{1}{2}};
				\node[draw, fill=orange!35] (1212) at (0, 1) {{1}{2}{1}{2}};
				\node (1122) at (0, 0) {{1}{1}{2}{2}};
				\draw (2211) -- (2121);
				\draw (2121) -- (1221);
				\draw (2121) -- (2112);
				\draw (1221) -- (1212);
				\draw (2112) -- (1212);
				\draw (1212) -- (1122);
			\end{tikzpicture}
		\end{tikzcd}
	\end{equation}

	\begin{example}
		Let us now present the next nontrivial example of an arborescent poset with cardinality \( 5 \). We consider the set of staircase corners \( \St_{\ov{n}} \) corresponding to the partition \( \ov{n} := (2 3^2 5^2) \). 
		
		Below, we illustrate the Hasse diagram of \( \St_{\ov{n}} \) along with the antilinearization \( \vrt \), as well as the Hasse diagram of the poset \( \DL_{\ov{n}} \).
		$$
		\begin{tikzcd}
			\begin{tikzpicture}[scale=0.5, shift={(0,-.5)}]
				\draw[step=1cm] (0,1) grid (5,0);
				\node (v1) at (.5,.4) {{ \color{blue}$\bullet$ }};
				\node (v2) at (1.5,.4) {{ \color{blue}$\bullet$ }};
				\node (v3) at (2.5,.4) {{ \color{blue}$\bullet$ }};
				\node (v4) at (3.5,.4) {{ \color{blue}$\bullet$ }};
				\node (v5) at (4.5,.4) {{ \color{blue}$\bullet$ }};
				\draw  (2.5,.6) edge[blue,line width=1.1pt,->,bend right=90] (.5,.6);
				\draw (2.5,.6) edge[blue,line width=1.1pt,->,bend right=60] (1.5,.6);
				\draw (4.5,.6) edge[blue,line width=1.1pt,->,bend right=60] (3.5,.6);
			\end{tikzpicture},
			\\
			\arrow[u,"\vrt"']
			\begin{tikzpicture}[scale=0.4,shift={(0,2)}]
				\draw[step=1cm] (0,0) grid (5,-2);
				\draw[step=1cm] (1,0) grid (5,-3);
				\draw[step=1cm] (3,0) grid (5,-5);
				\node (v21) at (0.5,-1.7) {{ \color{blue}$\bullet$ }};
				\node (v32) at (1.5,-2.7) {{ \color{blue}$\bullet$ }};
				\node (v13) at (2.5,-.7) {{ \color{blue}$\bullet$ }};
				\node (v54) at (3.5,-4.7) {{ \color{blue}$\bullet$ }};
				\node (v45) at (4.5,-3.7) {{ \color{blue}$\bullet$ }};
				\draw  (2.5,-.5) edge[blue,line width=1.1pt,->] (.5,-1.5);
				\draw (2.5,-.5) edge[blue,line width=1.1pt,->] (1.5,-2.5);  
				\draw (4.5,-3.5) edge[blue,line width=1.1pt,->] (3.5,-4.5);  
			\end{tikzpicture}
		\end{tikzcd}
		\quad
		\begin{tikzpicture}[scale=0.45,shift={(0,+7)}]
			\node (v0) at (0,0) {$\fgrid{5}{4}{3}{2}{1}$};
			\node (v10) at (-3,-2) {$\fgrid{5}{4}{2}{3}{1}$};
			\node (v12) at (3,-2) {$\fgrid{4}{5}{3}{2}{1}$};
			\node (v20) at (-6,-4) {$\fgrid{5}{4}{1}{3}{2}$};
			\node (v21) at (0,-4) {$\fgrid{5}{3}{2}{4}{1}$};
			\node (v22) at (6,-4) {$\fgrid{4}{5}{2}{3}{1}$};
			\node (v30) at (-10,-6) {$\fgrid{5}{3}{1}{4}{2}$};
			\node (v31) at (-4,-6) {$\fgrid{4}{5}{1}{3}{2}$};
			\node (v32) at (4,-6) {$\fgrid{4}{3}{2}{5}{1}$};
			\node (v33) at (10,-6) {$\fgrid{3}{5}{2}{4}{1}$};
			\node (v40) at (-10,-8) {$\fgrid{5}{2}{1}{4}{3}$};
			\node (v41) at (-4,-8) {$\fgrid{4}{3}{1}{5}{2}$};
			\node (v42) at (4,-8) {$\fgrid{3}{5}{1}{4}{2}$};
			\node (v43) at (10,-8) {$\fgrid{3}{4}{2}{5}{1}$};
			\node (v50) at (-6,-10) {$\fgrid{4}{2}{1}{5}{3}$};
			\node (v51) at (0,-10) {$\fgrid{2}{5}{1}{4}{3}$};
			\node (v52) at (6,-10) {$\fgrid{3}{4}{1}{5}{2}$};
			\node (v60) at (-3,-12) {$\fgrid{3}{2}{1}{5}{4}$};
			\node (v61) at (3,-12) {$\fgrid{2}{4}{1}{5}{3}$};
			\node (v80) at (0,-14) {$\fgrid{2}{3}{1}{5}{4}$};
			\draw (0,-.3) -- (-3,-1.4);
			\draw (0,-.3) -- (3,-1.4);
			\draw (-3,-2.3) -- (-6,-3.4);
			\draw (-3,-2.3) -- (0,-3.4);
			\draw (-3,-2.3) -- (6,-3.4);
			\draw (3,-2.3) -- (6,-3.4);
			\draw (-6,-4.3) -- (-10,-5.4);
			\draw (-6,-4.3) -- (-4,-5.4);
			\draw (0,-4.3) -- (-10,-5.4);
			\draw (0,-4.3) -- (4,-5.4);
			\draw (0,-4.3) -- (10,-5.4);
			\draw (6,-4.3) -- (-4,-5.4);
			\draw (6,-4.3) -- (4,-5.4);
			\draw (6,-4.3) -- (10,-5.4);
			\draw (-10,-6.3) -- (-10,-7.4);
			\draw (-10,-6.3) -- (-4,-7.4);
			\draw (-10,-6.3) -- (4,-7.4);
			\draw (-4,-6.3) -- (-4,-7.4);
			\draw (-4,-6.3) -- (4,-7.4);
			\draw (4,-6.3) -- (-4,-7.4);
			\draw (4,-6.3) -- (10,-7.4);
			\draw (10,-6.3) -- (4,-7.4);
			\draw (10,-6.3) -- (10,-7.4);
			\draw (-10,-8.3) -- (-6,-9.4);
			\draw (-10,-8.3) -- (0,-9.4);
			\draw (-4,-8.3) -- (-6,-9.4);
			\draw (-4,-8.3) -- (6,-9.4);
			\draw (4,-8.3) -- (0,-9.4);
			\draw (4,-8.3) -- (6,-9.4);
			\draw (10,-8.3) -- (6,-9.4);
			\draw (-6,-10.3) -- (-4,-11.4);
			\draw (-6,-10.3) -- (4,-11.4);
			\draw (0,-10.3) -- (4,-11.4);
			\draw (6,-10.3) -- (-4,-11.4);
			\draw (6,-10.3) -- (4,-11.4);
			\draw (-4,-12.3) -- (0,-13.4);
			\draw (4,-12.3) -- (0,-13.4);

		\end{tikzpicture}
		$$
	\end{example}

	\clearpage
	\section*{Index of notation}
	\phantomsection
	\label{sec::notation}

	The symbols are grouped by topic and, within a topic, roughly in the order in which they appear. The last column tells where the symbol is introduced: a section (\S), a numbered definition, notation, lemma, theorem or equation, together with the page. All references are hyperlinks. A few symbols have more than one meaning, always distinguished by the context:
	$\pi$ is used for the projection $\pi_\lambda$ and for the bijection $\pi_{ij}$; $\hor$ and $\vrt$ are defined both for cells (row and column indices) and for arrays (row and column sums); ${\KK}_\nu$ and ${\KK}_\nu$ both denote van der Kallen modules; $\Lattice$ (a lattice) is different from $\cL_\lambda$ (a line bundle).

\medskip
\begin{longtable}{@{}>{\raggedright\arraybackslash}p{0.15\textwidth}<{\ndots}@{\hspace{4pt}}>{\raggedright\arraybackslash}p{\dimexpr0.65\textwidth-2\tabcolsep-4pt\relax}>{\raggedright\arraybackslash}p{0.20\textwidth}@{}}
	\toprule
	\multicolumn{1}{@{}l@{\hspace{4pt}}}{\textbf{Symbol}} & \textbf{Meaning} & \textbf{Reference} \\
	\midrule
	\endfirsthead
	\toprule
	\multicolumn{1}{@{}l@{\hspace{4pt}}}{\textbf{Symbol}} & \textbf{Meaning} & \textbf{Reference} \\
	\midrule
	\endhead
	
	\ngroupfirst{A.\ Partitions and compositions}
	$\lambda\vdash N$ & partition $\lambda_1\ge\lambda_2\ge\dots$ of $N$ & \irs{ni:partition} \\
	$l(\lambda)$ & number of non-zero parts (length) of a partition $\lambda$ & \irs{ni:partition} \\
	$s_\lambda$ & Schur polynomial & \ire{eq::classical::Cauchy::intro} \\
	$\ov d \in \bZ^m_{\ge0}$ & compositions (weights of $\gl_m$); & \irs{ni:gln} \\
	$\lambda_+$ & dominant rearrangement of $\lambda$ (entries weakly decreasing) & \irs{ni:lambdaplus}; \ir{Fact}{fact::Bruhat::graph} \\
	$\lambda_-$ & antidominant rearrangement of $\lambda$ (entries weakly increasing); $\lambda_-=w_0\lambda_+$ & \irs{ni:lambdaplus}; \ir{Fact}{fact::Bruhat::graph} \\
	$\bS_m$ & symmetric group & \irs{ni:Sn} \\
	$\bS_m\lambda$ & set of all rearrangements of $\lambda\in\bZ^m_{\ge0}$ (the $\bS_m$-orbit of $\lambda$) & \irs{ni:Sn} \\
	$\preceq_{\Bruhat}$ & Bruhat order on $W\lambda$, in particular on rearrangements of a partition (smallest element $\lambda_+$) & \ir{Def.}{def::Bruhat}; \ire{eq::Bruhat::edge}; Intro: \irs{ni:keyintro} \\
	$l(\nu)$ & length (number of inversions) of $\nu\in\bS_m\lambda_+$; rank in the Bruhat graph & \ire{eq::length::S_n} \\
	$\delta$ & $\delta=(m-1,\dots,1,0)$ & \ire{eq::proj::parabolic} \\
	$\pi_\lambda$ & projection $\bS_m(\lambda+\delta)\twoheadrightarrow\bS_m\lambda$ & \ire{eq::proj::parabolic} \\
	$\psi_\lambda^{+}$ & left adjoint of the projection $\pi_\lambda$ & \ire{eq::proj::parabolic} \\
	$\psi_\lambda^{-}$ & right adjoint of the projection $\pi_\lambda$ & \ire{eq::proj::parabolic} \\
	
	\ngroup{B.\ Lie-theoretic notation}
	$\fg$ & semisimple Lie algebra & \irs{ni:Lie} \\
	$\fb^{\pm}$ & Borel subalgebras of $\fg$ & \irs{ni:Lie} \\
	$\fh$ & Cartan subalgebra of $\fg$ & \irs{ni:Lie} \\
	$\Phi_{\pm}$ & positive and negative roots & \irs{ni:Lie} \\
	$P$ & weight lattice & \irs{ni:Lie} \\
	$P_+$ & dominant weights & \irs{ni:Lie} \\
	$W$ & Weyl group & \irs{ni:Lie} \\
	$w_0$ & longest element of the Weyl group ($\omega_0$ in \S\ref{sec::Bruhat}); reverses a composition when $W=\bS_n$ & \irs{ni:w0} \\
	$\gl_n$ & Lie algebra of $n\times n$ matrices & \irs{ni:gln} \\
	$\fb_n$ & upper triangular Borel subalgebra of $\gl_n$ & \irs{ni:gln} \\
	$\fh_n$ & diagonal Cartan subalgebra of $\gl_n$ & \irs{ni:gln} \\
	$V_\lambda$ & irreducible finite-dimensional $\fg$-module with highest weight $\lambda$ & \irs{ni:Vlambda} \\
	$V_\lambda^{\gl_n}$ & irreducible finite-dimensional $\gl_n$-module with highest weight $\lambda$ & \irs{ni:Vgln} \\
	$v_\nu$ & extremal weight vector of weight $\nu$ in $V_\lambda$ & \irs{ni:Vlambda} \\
	$(\,\cdot\,)^{\op}$ & right module; opposite (right) versions of the objects below & \irs{ni:op} \\
	$\ch$ & character of a weight module & \ire{eq::character} \\
	$\ch_{\fh_n}$ & character of an $\fh_n$-weight module & \ire{eq::character} \\
	$\ch_{\fh_n\ttt\fh_m}$ & bigraded character of a $\fb_n\ttt\fb_m$-bimodule & \ire{eq::grF::key::polynom} \\
	
	\ngroup{C.\ Posets, lattices, M\"obius function (\S\ref{sec::Recollection})}
	$(x,y)$ & open interval of a poset & \irs{ni:interval} \\
	$[x,y]$ & closed interval of a poset & \irs{ni:interval} \\
	$\hat0$ & least element of a bounded poset & \ir{Def.}{def::bounded::graded} \\
	$\hat1$ & greatest element of a bounded poset & \ir{Def.}{def::bounded::graded} \\
	$\rk(x)$ & rank function of a graded poset & \ir{Def.}{def::bounded::graded} \\
	$\Gamma_X$ & Hasse graph (diagram) of the poset $X$ & \irs{ni:Hasse} \\
	thin & graded poset whose open intervals of length $2$ have exactly two elements & \ir{Def.}{def::thin} \\
	subthin & graded poset whose open intervals of length $2$ have at most two elements & \ir{Def.}{def::thin} \\
	$\euE$ & edge labelling of a Hasse graph & \ir{Def.}{def::ELlabelling} \\
	$\pth_{\euE}(x\prec y)$ & the distinguished chain of an $\EL$-labelling & \ir{Def.}{def::ELlabelling} \\
	$\EL$-shellable & poset admitting an $\EL$-labelling & \ir{Def.}{def::ELshellable} \\
	decreasing & map $F\colon X\to X$ with $F(x)\le x$ & \ir{Def.}{def::idempotent} \\
	increasing & map $F\colon X\to X$ with $F(x)\ge x$ & \ir{Def.}{def::idempotent} \\
	monotone & map $F\colon X\to X$ with $F(x)\le F(y)$ for $x\le y$ & \ir{Def.}{def::idempotent} \\
	idempotent & map $F\colon X\to X$ with $F(F(x))=F(x)$ & \ir{Def.}{def::idempotent} \\
	$\mu^X_{x,y}$ & M\"obius function of the poset $X$ & \ir{Def.}{def::Mobius} \\
	$\Delta(X)$ & order complex of $X$ & \ir{Def.}{def::ordercomplex} \\
	$\Lattice$ & lattice; distributive lattice & \ir{Def.}{def::lattice} \\
	$\vee$ & join in a lattice & \ir{Def.}{def::lattice} \\
	$\wedge$ & meet in a lattice & \ir{Def.}{def::lattice} \\
	$\Lattice(U)$ & lattice of all subspaces of a vector space $U$ & \irs{ni:LU} \\
	$\scriptstyle \Lattice(V_1,\dots,V_k;U)$ & sublattice of $\Lattice(U)$ generated by $V_1,\dots,V_k$ & \ir{Def.}{def::Lgenerated} \\
	$\euJ(\Lattice)$ & set of $\vee$-irreducible (join-irreducible) elements & \ir{Def.}{def::joinirr} \\
	$\Lattice_X$ & lattice of upper sets of a poset $X$ & \ir{Ex.}{ex::poset::lattice} \\
	$\calF_{\Lattice}^{i}V$ & standard filtration attached to $\euJ(\Lattice)$ & \ire{eq::filtration::distributive} \\
	$E^i$ & basis vectors lying in $V_i$ but in no smaller $V_j$ & \irs{ni:Ei} \\
	${\KK}_i$ & minimal subquotient of a distributive lattice & \ire{eq::min::subquot} \\
	$\bS_{\ov m}\backslash\bS_m/\bS_{\ov\lambda}$ & double coset poset; dominant compositions for a disjoint union of chains & \ir{Ex.}{ex::Double::coset} \\
	
	\ngroup{D.\ Arborescent posets, dominant compositions, bubble sort (\S\ref{sec::Bubble::sort::all})}
	$\sS$ & arborescent poset with consistent anti-linearization & \ir{Def.}{def::poset::linear} \\
	$v$ & consistent anti-linearization $v\colon\sS\hookrightarrow[1m]$ & \ir{Def.}{def::poset::linear} \\
	$\sT$ & arborescent poset with consistent linearization & \ir{Def.}{def::poset::linear} \\
	$h$ & consistent linearization $h\colon\sT\hookrightarrow[1n]$ & \ir{Def.}{def::poset::linear} \\
	$[1m]$ & $[1m]=\{1,\dots,m\}$ & \ir{Def.}{def::poset::linear} \\
	$\op(l)=\scriptstyle{m-l+1}$ & order-reversing involution of $[1m]$ & \ir{Rem.}{rk::opposite::linearization} \\
	$\ov d^{\op}$ & opposite composition & \ir{Rem.}{rk::opposite::linearization} \\
	$\sS_{\succeq s}$ & elements of $\sS$ that are $\succeq s$ & \ir{Def.}{def::poset::linear} \\
	$\sS_{\le k}$ & elements $s\in\sS$ with $v(s)\le k$ & \ir{Def.}{def::S::dominant} \\
	$\#\sS$ & number of elements of $\sS$ & \ir{Def.}{def::S::dominant} \\
	$\sS$-dominant & $d_l=0$ for $l\notin v(\sS)$ and $d_{v(s)}\ge d_{v(t)}$ for $s\succ t$ & \ir{Def.}{def::Sdominant} \\
	$\sS$-admissible & compositions with enough room for their non-zero entries: $\#\{j\le k\colon d_j>0\}\le\#\sS_{\le k}$ & \ir{Def.}{def::S::dominant} \\
	$\sS[l]$-dominant & dominance condition at the index $l$ & \ir{Def.}{def::S::dominant} \\
	$\scriptstyle\sS[\le k]$-{\small{dominant}}& $\sS[l]$-dominant for all $l\le k$ & \ir{Def.}{def::S::dominant} \\
	$\bA_{\sS}^{v}(\lambda)$ & set of $\sS$-admissible compositions with non-zero entries $\lambda$ & \ir{Not.}{not::ADsets} \\
	$\bD_{\sS}^{v}(\lambda)$ & set of $\sS$-dominant compositions with non-zero entries $\lambda$ & \ir{Not.}{not::ADsets} \\
	$\bAD_{\sS[\le k]}^{v}(\lambda)$ & set of $\sS[\le k]$-dominant admissible compositions with non-zero entries $\lambda$ & \ir{Not.}{not::ADsets} \\
	$\bD_{\sS}^{v}$ & the set $\bD_{\sS}^{v}(\lambda)$ for regular $\lambda$ (independent of $\lambda$) & \irs{ni:regular} \\
	$\tilde\sS_k$ & the poset $\sS_{\le k}\sqcup\{k+1,\dots,m\}$ & \irs{ni:Stilde} \\
	minimal $\sS$-disorder & pair of entries whose exchange is a covering relation & \ir{Def.}{def::disorder} \\
	$(ij)\cdot\ov d$ & composition obtained from $\ov d$ by exchanging $d_i$ and $d_j$ & \ir{Lemma}{lem::Bruhat::edge} \\
	$\bbs_{\sS}$ & bubble-sort map for an arborescent poset $\sS$ with a consistent anti-linearization & \ire{eq::bbs::comp} \\
	$\bbs_{\sS}^{k}$ & $k$-th step of the bubble-sort map & \ir{Alg.}{alg::bb:sort} \\
	$\bbs_{\sT}^{\op}$ & opposite bubble sort for a linearization & \ir{Thm.}{thm::bbs::op} \\
	$\lambda_+^{\sS}$ & smallest $\sS$-admissible rearrangement of $\lambda$ & \ir{Lemma}{lem::Admissible::poset} \\
	
	\ngroup{E.\ Demazure modules (\S\ref{sec::Demazure::all})}
	$D_\nu$ & Demazure module $\U(\fb^+)v_\nu\subset V_\lambda$ & \irs{ni:Demazure} \\
	$D_\nu^{\op}$ & right (opposite) Demazure module & \ir{Not.}{not::rightDemazure} \\	
	$\kappa_\nu(x)$ & key polynomial (character of $D_\nu$) & \irs{ni:keyintro}; \irs{ni:key} \\	
	$\kappa^\nu(y)$ & opposite key polynomial & \irs{ni:keyop}; \ire{eq:left-right} \\
	${\KK}_\nu$ & van der Kallen module $D_\nu/\sum_{\eta\prec\nu}D_\eta$ & \irs{ni:vdK} \\
	${\KK}_\nu^{\op}$ & right (opposite) van der Kallen module & \irs{ni:Kop} \\
	$a_\nu(x)$ & Demazure atom (character of ${\KK}_\nu$) & \irs{ni:keyintro}; \irs{ni:key} \\
	$a^\nu(y)$ & opposite Demazure atom & \irs{ni:keyop}; \ire{eq:left-right} \\
	$\Lattice_{D}(V_\lambda)$ & distributive lattice generated by the Demazure modules of $V_\lambda$ & \ir{Thm.}{thm::Demazure::intersect} \\
	$\Lattice_{D}^{\sT}(\lambda)$ & sublattice indexed by $\sT$-dominant weights & \ir{Cor.}{cor::distr::dominant} \\
	${\KK}_{\sT,\ov a}$ & generalized van der Kallen module & \ire{eq::vdK::general} \\
	${\KK}^{\op}_{\sS,\ov a}$ & generalized right van der Kallen module & \ir{Cor.}{cor::distr::dominant::op} \\
	$F=G/B$ & flag variety & \irs{ni:linebundle} \\
	$\cL_\lambda$ & line bundle on $F$ with $H^0(F;\cL_\lambda)\simeq V_\lambda^*$ & \irs{ni:linebundle} \\
	$X_w$ & Schubert variety & \irs{ni:linebundle} \\
	$D_{w\lambda}^{\perp}$ & annihilator of $D_{w\lambda}$ & \irs{ni:linebundle} \\
	$I_Y$ & ideal sheaf of $Y$ & \irs{ni:ideal} \\
	
	\ngroup{F.\ Staircase shapes and $DL$-dense arrays (\S\ref{sec::Staircase})}
	$\ov n=\scriptstyle(n_1\le\dots\le n_m)$ & column heights & \irs{ni:stairintro}; \irs{ni:nbar} \\
	$\bbY_{\ov n}$ & staircase shape (column $j$ consists of the top $n_j$ cells) & \irs{ni:stairintro}; \irs{ni:nbar} \\
	$\Mat_{\ov n}$ & space of staircase matrices of shape $\ov n$ & \ir{Not.}{not::Mat} \\
	$\ov{n^{ij}}$ & shape with row $i$ and column $j$ removed & \ire{eq::erased::diagram} \\
	$\pi_{ij}$ & induced bijection of cells & \ire{eq::erased::pi} \\
	$\St_{\ov n}$ & poset of staircase corners & \irs{ni:cornersintro}; \ir{Def.}{def::DL::indices}; \ire{eq::DL::order} \\
	$\hor$ (on cells) & row index of a cell; linearization of $\St_{\ov n}$ & \irs{ni:horvrtcell} \\
	$\vrt$ (on cells) & column index of a cell; anti-linearization of $\St_{\ov n}$ & \irs{ni:horvrtcell} \\
	$A$ & array & \ir{Def.}{def:arrays} \\
	$|A|$ & degree of an array & \ir{Def.}{def:arrays} \\
	$\hor(A)$ & horizontal weight (row sums) & \ir{Def.}{def:arrays} \\
	$\vrt(A)$ & vertical weight (column sums) & \ir{Def.}{def:arrays} \\
	\small{${DL}$-dense array} & order-preserving array supported on $\St_{\ov n}$ & \ir{Def.}{def::DL::dense} \\
	$\DL_{\ov n}(\lambda)$ & $DL$-dense arrays with non-zero entries $\lambda$ & \ir{Def.}{def::DL::dense::set} \\
	$\DL_{\ov n}$ & all $DL$-dense arrays & \ir{Rem.}{rem::DL::decomp} \\
	$\bD^{\hor}_{\St_{\ov n}}(\lambda)$ & dominant compositions for the linearization $\hor$ & \ir{Prop.}{prp::DL::hor::vrt} \\
	$\bD^{\vrt}_{\St_{\ov n}}(\lambda)$ & dominant compositions for the anti-linearization $\vrt$ & \ir{Prop.}{prp::DL::hor::vrt} \\
	minimal $DL$-disorder & minimal disorder of a $DL$-dense array & \ir{Def.}{def::min::DL::dis} \\
	$\Lattice^{\ov n}_{D\times D^{\op}}(\lambda)$ & double Demazure lattice & \ir{Thm.}{thm::D::Dop::distrib} \\
	${\KK}^{\ov n}_{A,B}$ & minimal subquotients of the double Demazure lattice & \irs{ni:KAB} \\
	${\KK}_{\St_{\ov n},\hor(A)}$ & generalized van der Kallen modules & \irs{ni:KAB} \\
	$\bbs_{\ov n}$ & bubble-sort map for $\St_{\ov n}$ (with the anti-linearization $\vrt$) & \irs{ni:bsintro}; \ir{Cor.}{cor::Cauchy::identity} \\
	$\hb$ & half-bubble-sort of~\cite{FKM::Cauchy} & \ir{Rem.}{rem::hb} \\
	
	\ngroup{G.\ Howe duality for staircase matrices (\S\ref{sec::Howe::duality})}
	$\Mat_{n\times m}$ & space of $n\times m$ matrices & \irs{ni:bimod} \\
	$S^N(\,\cdot\,)$ & $N$-th symmetric power & \ire{eq::Howe::intro} \\
	$v_{ij}$ & matrix unit & \irs{ni:vij} \\
	$v^A$ & monomial attached to an array $A$ & \irs{ni:vA} \\
	${}_\nu S(\Mat_{\ov n})_{\ov d}$ & weight subspaces of $S(\Mat_{\ov n})$ & \ir{Def.}{def::weightsub} \\
	$\ge$ on partitions & dominance order & \ir{Def.}{def::partition::order} \\
	$\calF^{\lambda}$ & right filtration of $S^N(\Mat_{\ov n})$ & \ire{eq::FIltrationBimodule} \\
	${}^{\lambda}\calF$ & left filtration of $S^N(\Mat_{\ov n})$ & \ire{eq::FIltrationBimodule} \\
	$\gr\calF^{\lambda}$ & associated graded piece & \ire{eq::subquot::hor::vrt} \\
	$A_\lambda$ & the $DL$-dense array of a rectangular shape with entries $\lambda$ & \ir{Ex.}{ex::rectangular} \\
	$\imath_{\ov n}$ & embedding $\Mat_{\ov n}\hookrightarrow\Mat_{n\times m}$ & \irs{ni:imath} \\
\end{longtable}

\end{document}